\documentclass[reqno]{amsart}
\usepackage{amssymb, latexsym}
\usepackage{hyperref}

\usepackage{mathrsfs}
\usepackage [latin1]{inputenc}
\usepackage{backref}

\allowdisplaybreaks

\numberwithin{equation}{section}
\newtheorem{Theorem}{Theorem}[section]
\newtheorem{Lemma}[Theorem]{Lemma}

\theoremstyle{definition}

\newtheorem{remark}[Theorem]{Remark}
\newtheorem{Proposition}[Theorem]{Proposition}

\providecommand{\norm}[1]{\left\Vert#1\right\Vert}

\newcounter{RomanNumber}

\def\be{\begin{equation}}
\def\en{\end{equation}}
\def\bs{\begin{split}}
\def\es{\end{split}}

\allowdisplaybreaks
\title[Optimal decay rates of a non--conservative compressible two--phase fluid model]
{Optimal decay rates of a non--conservative compressible two--phase
fluid model}

\author{Huaqiao Wang}
\address{College of Mathematics and Statistics, Chongqing University,
Chongqing 401331, China.}
\email{wanghuaqiao@cqu.edu.cn}
\author{Juan Wang}
\address{School of Mathematics and Statistics, Guangxi Normal University, Guilin, Guangxi 541004, P.R.
China} \email{wang$\_$juan98@126.com}
\author{Guochun Wu}
\address{Fujian Province University Key Laboratory of Computational Science, School of Mathematical Sciences, Huaqiao University, Quanzhou 362021, P.R. China.}
\email{guochunwu@126.com}
\author{Yinghui Zhang*}
\address{School of Mathematics and Statistics, Guangxi Normal University, Guilin, Guangxi 541004, P.R.
China} \email{yinghuizhang@mailbox.gxnu.edu.cn}

\subjclass[2010]{35Q30;\, 35K65;\, 76N10.}
\thanks{* Corresponding author: yinghuizhang@mailbox.gxnu.edu.cn}
\keywords{Non-conservative two--phase fluid model;\, optimal decay
rates;\, compressible.}\bigbreak

\date{\today}

\begin{document}
\begin{abstract}
We are concerned with the time decay rates of strong solutions to a
non--conservative compressible viscous two--phase fluid model in the
whole space $\mathbb R^3$. Compared to the previous related works,
the main novelty of this paper lies in the fact that it provides a
general framework that can be used to extract the optimal decay
rates of the solution as well as its all--order spatial derivatives
from one--order to the highest--order, which are the same as those of
the heat equation. Furthermore, for well--chosen initial data, we
also show the lower bounds on the decay rates. Our methods mainly
consist of Hodge decomposition, low--frequency and high--frequency
decomposition, delicate spectral analysis and energy method based on
finite induction.

\end{abstract}

\maketitle

\section{Introduction }
\setcounter{equation}{0}
As is well--known, most of the flows encountered in nature are
multifluid flows. They are widely used in nuclear power, chemical
processing, oil and gas manufacturing, and so on. The classic
approach to simplify the complexity of multiphase flows and satisfy
the engineer's need of some modeling tools is the well--known
volume--averaging method. This approach leads to so--called averaged
multi--phase models, see \cite{Bear, Brennen1, Raja} for details. As
a result of such a procedure, one can obtain the following generic
compressible two--phase fluid model:
\begin{equation}\label{1.1}
\left\{\begin{array}{l}
\alpha^{+}+\alpha^{-}=1, \\
\partial_{t}\left(\alpha^{\pm} \rho^{\pm}\right)+\operatorname{div}\left(\alpha^{\pm} \rho^{\pm} u^{\pm}\right)=0, \\
\partial_{t}\left(\alpha^{\pm} \rho^{\pm} u^{\pm}\right)+\operatorname{div}\left(\alpha^{\pm} \rho^{\pm} u^{\pm} \otimes u^{\pm}\right)
+\alpha^{\pm} \nabla P^{\pm}\left(\rho^{\pm}\right)=\operatorname{div}\left(\alpha^{\pm} \tau^{\pm}\right), \\
P^{+}\left(\rho^{+}\right)-P^{-}\left(\rho^{-}\right)=f\left(\alpha^{-}
\rho^{-}\right),
\end{array}\right.
\end{equation} where the variable $0 \leqq \alpha^{+}(x, t) \leqq 1$ is the
volume fraction of fluid $+$ (the liquid), and $0 \leqq
\alpha^{-}(x, t) \leqq 1$ that of fluid $-$ (the gas).
$\rho^{\pm}(x, t) \geqq 0, u^{\pm}(x, t)$ and
$P^{\pm}\left(\rho^{\pm}\right)=A^{\pm}\left(\rho^{\pm}\right)^{\bar{\gamma}^{\pm}}$
denote the densities, the velocity of each phase, and the two
pressure functions, respectively. $\bar{\gamma}^{\pm} \geqq 1,
A^{\pm}>0$ are positive constants. In what follows, we set
$A^{+}=A^{-}=1$ without loss of any generality. The main purpose of
this article is to get the optimal decay rates of the model
\eqref{1.1} in the whole space $\mathbb{R}^3$. As in \cite{Evje9},
we also assume that $f(\cdot) \in C^{3}([0, \infty))$ and $f$ is a
strictly decreasing function near the equilibrium. Moreover,
$\tau^{\pm}$ are the viscous stress tensors defined by
\begin{equation}\label{1.2}
\tau^{\pm}:=\mu^{\pm}\left(\nabla u^{\pm}+\nabla^{t}
u^{\pm}\right)+\lambda^{\pm} \operatorname{div} u^{\pm} \mathrm{Id},
\end{equation}
where the constants $\mu^{\pm}$ and $\lambda^{\pm}$ are shear and
bulk viscosity coefficients satisfying the physical condition:
$\mu^{\pm}>0$ and $2 \mu^{\pm}+3 \lambda^{\pm} \geqq 0,$ which
implies that $\mu^{\pm}+\lambda^{\pm}>0 .$ This system is known as a
two--fluid flow system with algebraic closure. For more information
about this model, we refer to \cite{Bresch1, Bresch2, Evje9, Friis1,
Ishii1, Prosperetti, Raja} and references therein. Particularly,
Bretsch et al. in the seminal work \cite{Bresch1} considered a model
similar to \eqref{1.1}. More specifically, they made the following
assumptions:
\begin{itemize} \item a common pressure $P^{+}=P^{-}=P$;
\item inclusion of viscous terms of the form \eqref{1.2} where $\mu^{\pm}$
depends on densities $\rho^{\pm}$ and $\lambda^{\pm}=0$;
\item inclusion of a third order derivative of $\alpha^{\pm}
\rho^{\pm}$, which are so--called  internal capillary forces
represented by the well--known Korteweg model on each phase.
\end{itemize}
They obtained the global weak solutions in the periodic domain with
$1<\overline{\gamma}^{\pm}< 6$. Later, Bresch, Huang, and Li
\cite{Bresch2} established the global existence of weak solutions in
one space dimension without the internal capillary forces when
$\overline{\gamma}^{\pm}>1$. Recently, Cui, Wang, Yao, and Zhu
\cite{c1} obtained the time--decay rates of classical solutions for
the three-dimensional Cauchy problem by combining  detailed analysis
of the Green's function to the linearized system with energy
estimates to the nonlinear system.
 \par
The relation between the pressures of \eqref{1.1} implies the
differential identity
\begin{equation}\label{1.3}
\mathrm{d} P^{+}-\mathrm{d} P^{-}=\mathrm{d} f\left(\alpha^{-}
\rho^{-}\right),
\end{equation}
where $P^{\pm}:=P^{\pm}\left(\rho^{\pm}\right).$ It is clear that
\[
\mathrm{d} P^{+}=s_{+}^{2} \mathrm{d} \rho^{+}, \quad \mathrm{d}
P^{-}=s_{-}^{2} \mathrm{d} \rho^{-}, \quad \text { where }
s_{\pm}^{2}:=\frac{\mathrm{d} P^{\pm}}{\mathrm{d}
\rho^{\pm}}\left(\rho^{\pm}\right)=\bar{\gamma}^{\pm}
\frac{P^{\pm}\left(\rho^{\pm}\right)}{\rho^{\pm}}.
\]
Here $s_{\pm}$ represent the sound speed of each phase, respectively.
As in \cite{Bresch1}, we introduce the variables
\begin{equation}\label{1.4}
R^{\pm}=\alpha^{\pm} \rho^{\pm},
\end{equation}
which together with $\eqref{1.1}_{1}$ gives
\begin{equation}\label{1.5}
\mathrm{d} \rho^{+}=\frac{1}{\alpha_{+}}\left(\mathrm{d}
R^{+}-\rho^{+} \mathrm{d} \alpha^{+}\right), \quad \mathrm{d}
\rho^{-}=\frac{1}{\alpha_{-}}\left(\mathrm{d} R^{-}+\rho^{-}
\mathrm{d} \alpha^{+}\right).
\end{equation}
By virtue of \eqref{1.3} and \eqref{1.4}, we finally get
\begin{equation}\label{1.6}
\mathrm{d} \alpha^{+}=\frac{\alpha^{-} s_{+}^{2}}{\alpha^{-}
\rho^{+} s_{+}^{2}+\alpha^{+} \rho^{-} s_{-}^{2}} \mathrm{d}
R^{+}-\frac{\alpha^{+} \alpha^{-}}{\alpha^{-} \rho^{+}
s_{+}^{2}+\alpha^{+} \rho^{-}
s_{-}^{2}}\left(\frac{s_{-}^{2}}{\alpha^{-}}+f^{\prime}\right)
\mathrm{d} R^{-}. \end{equation}
 Substituting \eqref{1.6} into \eqref{1.5}, we deduce
the following expressions:
\[
\mathrm{d} \rho^{+}=\frac{\rho^{+} \rho^{-}
s_{-}^{2}}{R^{-}\left(\rho^{+}\right)^{2}
s_{+}^{2}+R^{+}\left(\rho^{-}\right)^{2} s_{-}^{2}}\left(\rho^{-}
\mathrm{d} R^{+}+\left(\rho^{+}+\rho^{+} \frac{\alpha^{-}
f^{\prime}}{s_{-}^{2}}\right) \mathrm{d} R^{-}\right),
\]
and
\[
\mathrm{d} \rho^{-}=\frac{\rho^{+} \rho^{-}
s_{+}^{2}}{R^{-}\left(\rho^{+}\right)^{2}
s_{+}^{2}+R^{+}\left(\rho^{-}\right)^{2} s_{-}^{2}}\left(\rho^{-}
\mathrm{d} R^{+}+\left(\rho^{+}-\rho^{-} \frac{\alpha^{+}
f^{\prime}}{s_{+}^{2}}\right) \mathrm{d} R^{-}\right),
\]
which together with \eqref{1.3} give the pressure differential
$\mathrm{d} P^{\pm}$:
\[
\mathrm{d} P^{+}=\mathcal{C}^{2}\left(\rho^{-} \mathrm{d}
R^{+}+\left(\rho^{+}+\rho^{+} \frac{\alpha^{-}
f^{\prime}}{s_{-}^{2}}\right) \mathrm{d} R^{-}\right) ,\] and
\[
\mathrm{d} P^{-}=\mathcal{C}^{2}\left(\rho^{-} \mathrm{d}
R^{+}+\left(\rho^{+}-\rho^{-} \frac{\alpha^{+}
f^{\prime}}{s_{+}^{2}}\right) \mathrm{d} R^{-}\right) ,\] where
\[
\mathcal{C}^{2}:=\frac{s_{-}^{2} s_{+}^{2}}{\alpha^{-} \rho^{+}
s_{+}^{2}+\alpha^{+} \rho^{-} s_{-}^{2}}.\]\par Next, by noting the
fundamental relation: $\alpha^++\alpha^-=1$, we can get the
following equality:
\begin{equation}\label{1.7}
\frac{R^+}{\rho^+}+\frac{R^-}{\rho^-}=1, ~~\hbox{and thus}~~
\rho^-=\frac{R^-\rho^+}{\rho^+-R^+}.\end{equation} Then, we have
from the pressure relation $\eqref{1.1}_4$ that
\begin{equation}\label{1.8}
\varphi(\rho^+, R^+,
R^-):=P^+(\rho^+)-P^-{\left(\frac{R^-\rho^+}{\rho^+-R^+}\right)}-f({R^-})=0.
\end{equation}

\noindent Thus, we can employ the implicit function theorem to
define $\rho^{+}$. To see this, by differentiating the above
equation with respect to $\rho^{+}$ for given $R^{+}$ and $R^{-}$,
we get
\[
\frac{\partial\varphi}{\partial\rho^+}(\rho^+, R^+,
R^-)=s_{+}^{2}+s_{-}^{2} \frac{R^{-}
R^{+}}{\left(\rho^{+}-R^{+}\right)^{2}},
\]
which is positive for any $\rho^{+}\in(R^+, +\infty)$ and
$R^{\pm}>0.$ This together with the implicit function theorem
implies that $\rho^{+}=\rho^{+}\left(R^{+}, R^{-}\right)
\in\left(R^{+},+\infty\right)$ is the unique solution of the
equation \eqref{1.8}. By virtue of \eqref{1.4}, \eqref{1.8} and
$\eqref{1.1}_{1}, \rho^{-}$ and $\alpha^{\pm}$ can be defined by
\[
\begin{aligned}
\rho^{-}\left(R^{+}, R^{-}\right) &=\frac{R^{-} \rho^{+}\left(R^{+}, R^{-}\right)}{\rho^{+}\left(R^{+}, R^{-}\right)-R^{+}}, \\
\alpha^{+}\left(R^{+}, R^{-}\right) &=\frac{R^{+}}{\rho^{+}\left(R^{+}, R^{-}\right)}, \\
\alpha^{-}\left(R^{+}, R^{-}\right)
&=1-\frac{R^{+}}{\rho^{+}\left(R^{+},
R^{-}\right)}=\frac{R^{-}}{\rho^{-}\left(R^{+}, R^{-}\right)}.
\end{aligned}
\]
We refer the readers to [\cite{Bresch2}, P. 614] for more details.
\par
Therefore, we can rewrite system \eqref{1.1} into the following
equivalent form:
\begin{equation}\label{1.9}
\left\{\begin{array}{l}
\partial_{t} R^{\pm}+\operatorname{div}\left(R^{\pm} u^{\pm}\right)=0, \\
\partial_{t}\left(R^{+} u^{+}\right)+\operatorname{div}\left(R^{+} u^{+} \otimes u^{+}\right)+\alpha^{+} \mathcal{C}^{2}\left[\rho^{-} \nabla R^{+}+\left(\rho^{+}+\rho^{+} \frac{\alpha^{-} f^{\prime}}{s_{-}^{2}}\right) \nabla R^{-}\right] \\
\hspace{2.5cm}=\operatorname{div}\left\{\alpha^{+}\left[\mu^{+}\left(\nabla u^{+}+\nabla^{t} u^{+}\right)
+\lambda^{+} \operatorname{div} u^{+} \operatorname{Id}\right]\right\}, \\
\partial_{t}\left(R^{-} u^{-}\right)+\operatorname{div}\left(R^{-} u^{-} \otimes u^{-}\right)+\alpha^{-} \mathcal{C}^{2}\left[\rho^{-} \nabla R^{+}+\left(\rho^{+}-\rho^{-} \frac{\alpha^{+} f^{\prime}}{s_{+}^{2}}\right) \nabla R^{-}\right] \\
\hspace{2.5cm}=\operatorname{div}\left\{\alpha^{-}\left[\mu^{-}\left(\nabla
u^{-}+\nabla^{t} u^{-}\right)+\lambda^{-} \operatorname{div} u^{-}
\operatorname{Id}\right]\right\}.
\end{array}\right.
\end{equation}
In the present paper, we consider the initial value problem to
\eqref{1.9} in the whole space $\mathbb R^3$ with the initial data
\begin{equation}\label{1.10} (R^{+}, u^{+}, R^{-}, u^{-})(x,
0)=(R_{0}^{+}, u_{0}^{+}, R_{0}^{-},
u_{0}^{-})(x)\rightarrow(R_{\infty}^{+}, 0, R_{\infty}^{-}, 0) \quad
\hbox{as}\quad |x|\rightarrow\infty \in \mathbb{R}^{3},
\end{equation}
where $R^{\pm}_\infty>0$ denote the background doping profiles, and
for simplicity, are taken as 1 in this paper.
\par To put our results into context, let us highlight some recent progress on the topics of
non--conservative compressible viscous two--phase fluid model and
related model. By taking the following simplifications:
\begin{itemize}
  \item Due to the fact that the liquid phase is much heavier than the gas phase, typically to the order $\frac{\rho_+}{\rho_-}\sim 10^{3}$,
  we can neglect the gas phase in the momentum equation corresponding to the gas;
  \item A non--slip condition is assumed, i.e., $u_+=u_-=u$,
\end{itemize}
and setting $m=\alpha_+\rho_+$ and $n=\alpha_-\rho_-$, one can get
the simplified version of \eqref{1.1}: \begin{equation} \ \ \
\left\{
\begin{array}{ll}
m_t+\hbox{div}(m u)=0,\\
n_t+\hbox{div}(n u)=0,\\
(m u)_t+\hbox{div}(m u\otimes u)+\nabla P(m,n)=\hbox{div}
\left(\mu(\nabla u+\nabla u^T)+\lambda(\hbox{div}u)I\right).
\end{array}
\right.   \label{1.11}
\end{equation}
For the simplified model \eqref{1.11} and related model that are
subject to various initial and initial--boundary conditions have
been explored thoroughly during the past decades, the global
existence and asymptotic behavior of the solutions (weak, strong,
classic) were proved. We refer the readers to \cite{Evje1, Evje2,
Evje3, Evje4, Evje5, Evje6, Evje7, Evje8, Evje10, Fan1, Guo2, Hao1,
Huang1, Liu1, ruan1, ruan2, zw2, Vasseur, zw3, Wen1, Wu9, Yao1,
Yao2, Yao3, Yao4, Yao5, ZYH3, Zhang4} and references therein.
\par However, due to difficulties coming from different nonlinear density--pressure laws
corresponding to the different phases and the appearance of the
non--conservative pressure terms $\alpha^{\pm}\nabla P^{\pm}$, so
far there is few results on the non--conservative compressible
viscous two--phase fluid model \eqref{1.9}. Recently, Evje, Wang and
Wen \cite{Evje9} studied the global well--posedness and decay rates
of the Cauchy problem \eqref{1.9}--\eqref{1.10}. Their main result
can be stated in the following theorem:
\begin{Theorem}\label{0mainth}Under the
condition
\begin{equation}\label{1.12}-\frac{s_{-}^{2}(1,1)}{\alpha^{-}(1,1)}<f^{\prime}(1)<\frac{\eta-s_{-}^{2}(1,1)}{\alpha^{-}(1,1)}<0,\end{equation}
where $\eta$ is a positive, small fixed constant, there exists a
constant $\varepsilon$ such that if
\begin{equation}\label{1.13}
\left\|\left(R_{0}^{+}-1, u_{0}^{+}, R_{0}^{-}-1,
u_{0}^{-}\right)\right\|_{H^{2}\left(\mathbb{R}^{3}\right)} \leq
\varepsilon,
\end{equation}
then the initial--value problem \eqref{1.9}--\eqref{1.10} admits a
solution $\left(R^{+}, u^{+}, R^{-}, u^{-}\right)$ globally in time
in the sense that \[
\begin{array}{l}
R^{+}-1, R^{-}-1 \in C^{0}\left([0, \infty) ; H^{2}\left(\mathbb{R}^{3}\right)\right) \cap C^{1}\left([0, \infty) ;
H^{1}\left(\mathbb{R}^{3}\right)\right), \\
u^{+}, u^{-} \in C^{0}\left([0, \infty) ;
H^{2}\left(\mathbb{R}^{3}\right)\right) \cap C^{1}\left([0, \infty)
; L^{2}\left(\mathbb{R}^{3}\right)\right).
\end{array}
\]
Moreover, if in addition the initial data $\left(R_{0}^{+}-1,
u_{0}^{+}, R_{0}^{-}-1, u_{0}^{-}\right)$ is bounded in
$L^{1}\left(\mathbb{R}^{3}\right),$ the solution enjoys the
following decay--in--time estimates:
\begin{align}\label{1.14}
\left\|\left(R^{+}-1, u^{+}, R^{-}-1, u^{-}\right)(t)\right\|_{L^{2}\left(\mathbb{R}^{3}\right)} &
\leq C(1+t)^{-\frac{3}{4}} \text { for all } t \geq 0, \\
\label{1.15}\left\|\nabla\left(R^{+}, u^{+}, R^{-},
u^{-}\right)(t)\right\|_{H^{1}\left(\mathbb{R}^{3}\right)} & \leq
C(1+t)^{-\frac{5}{4}} \text { for all } t \geq 0,
\end{align}
for some positive constant $C$.\end{Theorem}
\par
\bigskip
Noting the decay rate in \eqref{1.15}, for the second--order (i.e.
\textbf{the highest--order}) spatial derivative  of the solution, it
holds that
\begin{equation}\label{1.16}\left\|\nabla^2\left(R^{+}, u^{+}, R^{-},
u^{-}\right)(t)\right\|_{L^{2}\left(\mathbb{R}^{3}\right)}\leq
C(1+t)^{-\frac{5}{4}}.\end{equation} On the other hand, let us
revisit the following classical result of the heat equation£º
\begin{equation}\label{1.17}
\left\{\begin{array}{l}
\partial_{t} u-\Delta u=0, ~\text{ in } ~\mathbb{R}^{3}, \\
u|_{t=0}=u_0.
\end{array}\right.
\end{equation}
If $u_0\in H^N(\mathbb{R}^{3})\cap L^1(\mathbb{R}^{3})$ with $N\geq
0$ be an integer, then for any $0\leq \ell\leq N$, the solution of
the heat equation \eqref{1.17} has the following decay rate:
\begin{equation}\label{1.18}\|\nabla^\ell u(t)\|_{L^2(\mathbb{R}^{3})}\leq C(1+t)^{-\frac{3}{4}-\frac{\ell}{2}}.\end{equation}
Particularly, taking $N=2$ and $\ell=2$, one has
\begin{equation}\label{1.19}\|\nabla^2 u(t)\|_{L^2(\mathbb{R}^{3})}\leq C(1+t)^{-\frac{7}{4}}.\end{equation}
Therefore, in view of \eqref{1.16} and \eqref{1.19}, it is clear
that the decay rate of the second--order spatial derivative of the
solution in \eqref{1.16} is slower than that of the heat equation as
in \eqref{1.19}. So, the decay rate of the second--order spatial
derivative of the solution in \eqref{1.15} is not optimal in this
sense.\par
 The main motivation of this paper is to provide a general framework that can be used to extract the optimal decay
rates of the solution to the Cauchy problem \eqref{1.9}--\eqref{1.10}
as well as its all--order spatial derivatives from one--order to the
highest--order. More precisely, we obtain the optimal decay rates of
the solution to the Cauchy problem \eqref{1.9}--\eqref{1.10} as well
as its all--order spatial derivatives from one--order to the
highest--order, which are the same as those of the heat equation.
Moreover, for well--chosen initial data, we also show the lower
bounds on the decay rates. Our methods mainly include Hodge
decomposition, low--frequency and high--frequency decomposition,
delicate spectral analysis and energy method based on finite
induction.\par
 Before stating
our main result, let us first introduce the notations and
conventions used throughout this paper. We use $H^k(\mathbb R^3)$ to
denote the usual Sobolev spaces with norm $\|\cdot\|_{H^k}$ and
$L^p$, $1\leq p\leq \infty$ to denote the usual $L^p(\mathbb R^3)$
spaces with norm $\|\cdot\|_{L^p}$.  For the sake of conciseness, we
do not precise in functional space names when they are concerned
with scalar-valued or vector-valued functions, $\|(f, g)\|_X$
denotes $\|f\|_X+\|g\|_X$.  We will employ the notation $a\lesssim
b$ to mean that $a\leq Cb$ for a universal constant $C>0$ that only
depends on the parameters coming from the problem. We denote
$\nabla=\partial_x=(\partial_1,\partial_2,\partial_3)$, where
$\partial_i=\partial_{x_i}$, $\nabla_i=\partial_i$ and put
$\partial_x^\ell f=\nabla^\ell f=\nabla(\nabla^{\ell-1}f)$.  Let
$\Lambda^s$ be the pseudo differential operator defined by
\begin{equation}\Lambda^sf=\mathfrak{F}^{-1}(|{\bf \xi}|^s\widehat f),~\hbox{for}~s\in \mathbb{R},\nonumber\end{equation}
where $\widehat f$ and $\mathfrak{F}(f)$ are the Fourier transform
of $f$. The homogenous Sobolev space $\dot{H}^s(\mathbb{R}^3)$ with
norm given by $\|f\|_{\dot{H}^s}\overset{\triangle}=\|\Lambda^s
f\|_{L^2}$. For a radial function $\phi\in C_0^\infty(\mathbb
R^3_{{\bf \xi}})$  such that $\phi({\bf \xi})=1$ when $|{\bf
\xi}|\leq \frac{\eta_0}{2}$ and $\phi({\bf \xi})=0$ when $|{\bf
\xi}|\geq \eta_0$ with a given positive constant $\eta_0$, we define
the low--frequency part of $f$ by
$$f^l=\mathfrak{F}^{-1}[\phi({\bf \xi})\widehat f]$$
and the high--frequency part of $f$ by
$$f^h=\mathfrak{F}^{-1}[(1-\phi({\bf \xi}))\widehat f].$$
It is direct to check that $f=f^l+f^h$ if the Fourier transform of $f$ exists.

\medskip
Now, we are in a position to state our main result.
\smallskip
\begin{Theorem}\label{1mainth}Assume that $R_{0}^{+}-1, u_{0}^{+}, R_{0}^{-}-1,
u_{0}^{-}\in H^N(\mathbb{R}^3)$ for an integer $N\geq 2$ and
\eqref{1.12} holds for a given small positive constant $\eta$.
There exists a constant $\delta_0$ such that if
\begin{equation}\label{1.20}
\left\|\left(R_{0}^{+}-1, u_{0}^{+}, R_{0}^{-}-1,
u_{0}^{-}\right)\right\|_{H^{2}} \leq \delta_0,
\end{equation}
then the Cauchy problem \eqref{1.9}--\eqref{1.10} admits a unique
solution $\left(R^{+}, u^{+}, R^{-}, u^{-}\right)$ globally in time
in the sense that \[
\begin{array}{l}
R^{+}-1, R^{-}-1 \in C^{0}\left([0, \infty) ;
H^{N}\left(\mathbb{R}^{3}\right)\right) \cap C^{1}\left([0, \infty)
;
H^{N-1}\left(\mathbb{R}^{3}\right)\right), \\
u^{+}, u^{-} \in C^{0}\left([0, \infty) ;
H^{N}\left(\mathbb{R}^{3}\right)\right) \cap C^{1}\left([0, \infty)
; H^{N-2}\left(\mathbb{R}^{3}\right)\right).
\end{array}
\]
Moreover, the following convergence rates hold true.
\smallskip
\begin{itemize}
\item {\bf Upper bounds.} If additionally
\begin{equation}\label{1.21}N_0=\left\|\left(R_{0}^{+}-1, u_{0}^{+},
R_{0}^{-}-1, u_{0}^{-}\right)\right\|_{L^{1}}<+\infty,\end{equation}
 then for all $t\geq 0,$
\begin{equation}\label{1.22}\left\|\nabla^{\ell}\left(R^{+}-1, u^{+}, R^{-}-1,
u^{-}\right)(t)\right\|_{L^{2}} \leq
C(N_0)(1+t)^{-\frac{3}{4}-\frac{\ell}{2}},
\end{equation}
and
\begin{equation}\label{1.23}\left\|\left(R^{+}-1, u^{+}, R^{-}-1,
u^{-}\right)(t)\right\|_{L^{p}}\leq
C(N_0)(1+t)^{-\frac{3}{2}\left(1-\frac{1}{p}\right)},
\end{equation}
for $0\leq \ell\leq N$ and $2\le p\le \infty$.
\smallskip
\item {\bf Lower bounds.} Let  $(n_0^+,u_0^+,R_0^-,u_0^-)=(\alpha_1(R_0^+-1),\sqrt{\alpha_1}
u_0^+,\alpha_4(R_0^--1),\sqrt{\alpha_4}u_0^-)$ where the definitions
of the positive constants $\alpha_1$ and $\alpha_4$ are given in
Section 2 and assume that the Fourier transform of functions
$(n_0^+,u_0^+,n_0^-,u_0^-)$ satisfy
\begin{equation}\label{1.24}\quad\quad\widehat{n}_0^-(\xi)=0,~
\wedge^{-1}\text{\rm div}\ \widehat{u}_0^+(\xi)=\wedge^{-1}\text{\rm div}\
\widehat{u}_0^-(\xi)=0,~\text{\rm and }~ |\widehat{n}_0^+(\xi)|\ge
N_0\sqrt{\delta_0},
\end{equation}
for any $|\xi|\le \eta_1$. Then there is a positive constant $c_0$
independent of time such that for any large enough $t$,
\begin{equation}\begin{split}\label{1.25}
\min&\left\{\|\nabla^\ell(R^+-1)(t)\|_{L^2},\|\nabla^\ell
u^+(t)\|_{L^2},\|\nabla^\ell(R^--1)(t)\|_{L^2},\|\nabla^\ell
u^-(t)\|_{L^2}\right\}\\ \quad\quad\geq
&~c_0(1+t)^{-\frac{3}{4}-\frac{\ell}{2}},
\end{split}\end{equation}
and
\begin{equation}\label{1.26}
\hspace{1.5cm}\min\left\{\|(R^+-1)(t)\|_{L^p},\|u^+(t)\|_{L^p},\|(R^--1)(t)\|_{L^p},\|u^-(t)\|_{L^p}\right\}\ge
c_0(1+t)^{-\frac{3}{2}\left(1-\frac{1}{p}\right)},
\end{equation}
for $0\leq \ell\leq N$ and $\ 2\leq p\leq\infty.$
\end{itemize}
\end{Theorem}

\begin{remark} Compared to Theorem \ref{0mainth} of Evje, Wang and Wen
\cite{Evje9}, the main new contribution of Theorem \ref{1mainth}
lies in that it provides a general framework that can be used to
derive the optimal decay rates of the solution as well as its
all--order spatial derivatives from one--order to the
highest--order, which are the same as those of the heat equation.
More specifically, under the assumptions that the initial data
belongs to $H^N$ with any integer $N\geq 2$, $H^2$--norm of the
initial data is sufficiently small and $L^1$--norm of the initial
data is bounded, but the higher--order norms can be arbitrarily
large, our approach
 shows that
the optimal decay rates of the solution as well as its all--order
spatial derivatives from one--order to the highest--order
($N$--order) are the same as those of the heat equation.
Particularly, by taking $N=2$ in Theorem \ref{1mainth}, it is easy
to see that Theorem \ref{0mainth} is a direct corollary of Theorem
\ref{1mainth}. Moreover, it is clear that in \eqref{1.22}, the
second--order (the highest--order) spatial derivative of the
solution decays at the $L^2$--rate $(1+t)^{-\frac{7}{4}}$, which is
faster than the $L^2$--rate $(1+t)^{-\frac{5}{4}}$ in \eqref{1.15}.
On the other hand, for the general case $N>2$, under the assumption
that $H^2$--norm of the initial data is small but its higher--order
norms can be arbitrarily large, we obtain the optimal decay rates of
the solution as well as its all--order spatial derivatives from
one--order to the highest-order ($N$--order), which are the same as
those of the heat equation. Finally, for well--chosen initial data,
we also show the lower bounds on the decay rates. It is worth
mentioning that it is rather difficult to obtain the lower bounds on
the decay rates by following the method of Evje, Wang and Wen
\cite{Evje9}.
\end{remark}

\begin{remark} In \cite{Bresch1, c1}, $\Delta P=P^+-P^-=0$
and a third order derivative of $\alpha^+\rho^+$ accounting for
internal capillary pressure forces are involved. In \cite{Evje9},
they considered the capillary pressure $\Delta
P=P^+-P^-=f(\alpha^-\rho^-)$, where $f$ satisfies the assumption
\eqref{1.12}. It should be mentioned that both in \cite{Bresch1, c1}
and in \cite{Evje9}, the capillary term  play a key role in the
proofs of their main results. Therefore, a natural and important
problem is that what will happen if $\Delta P=P^+-P^-=0$ and no
capillary term is involved. In fact, in this case, the linear system
of the model has zero eigenvalue which makes the problem become much
more difficult and complicated. This will be our future work.
\end{remark}

\begin{remark} Employing the pure energy method developed in
\cite{Guo1, zw2}, one can obtain the optimal decay rates but except
the highest--order one and low bounds on the optimal decay rates.
However, it seems impossible to get the optimal decay rate of the
highest--order spatial derivative of the solution and the low bounds
on the decay rates due to the strong ``degenerate" and ``nonlinear"
structure of the system.
\end{remark}

\smallskip
\smallskip

\indent Now, let us sketch the strategy of proving Theorem
\ref{1mainth} and explain some of the main difficulties and
techniques involved in the process. As mentioned before, the main
purpose of the present paper is to put forward a general framework
to derive the optimal decay rates of the solution as well as its
all--order spatial derivatives from one--order to the highest--order,
which are the same as those of the heat equation. Therefore,
compared to Evje, Wang and Wen \cite{Evje9}, we need to develop new
ingredients in the proof to handle the optimal decay rate of the
highest--order spatial derivative of the solution and the lower bound
on the optimal decay rate of the solution, which requires some new
ideas. More precisely, we will employ Hodge decomposition, the
low--frequency and high--frequency decomposition, delicate spectral
analysis and energy methods based on finite induction. Roughly
speaking, our proof mainly involves the following five steps.\par

First, we rewrite the system \eqref{1.9}--\eqref{1.10} in
perturbation form and analyze the spectral of the solution semigroup
to the corresponding linear system. We therefore encounter a
fundamental obstacle that the matrix $\mathcal A(\xi)$ in
\eqref{2.15} is an $8$--order matrix and is  not self--adjoint.
Particularly, it is easy to check that the the matrix $\mathcal
A(\xi)$ can not be diagonalizable (see \cite{Sideris} pp.807 for
example). Therefore, it seems impossible to apply the usual time
decay investigation through spectral analysis. To overcome this
difficulty, we will employ the Hodge decomposition technique
developed in \cite{Dan1, Dan2, Dan3} to split the linear system into
three systems. One has four distinct eigenvalues and the other two
are classic heat equations. By making careful pointwise estimates on
the Fourier transform of Green's function to the linearized
equations, we can obtain the desired linear $L^2$--estimates.\par
Second, we deduce the optimal convergence rate on
$\|\nabla^j(n^{+,l},u^{+, l},n^{-,l},u^{-, l})\|_{L^2}$ with $0\leq
j\leq N$. Our method is to use Duhamel's principle, linear
$L^2$--estimates, nonlinear energy estimates. However, we encounter
a difficulty from the fact that when we estimate the highest-order
term $\|\nabla^N(n^{+,l},u^{+, l},n^{-,l},u^{-, l})\|_{L^2}$, it
requires us to control the terms involving $\nabla^{N+1}(u^+, u^-)$
which however don't belong to the solution space. To get around this
difficulty, we separate the time interval into two parts and make
full use of the benefit of the low--frequency and high--frequency
decomposition to get our desired convergence rates (See the proof of
\eqref{3.52} for details).\par Third, we prove the optimal decay
rates for $\textbf{Case I: N=2}$. By virtue of Theorem
\ref{0mainth}, it suffices to prove \eqref{3.3}. To illustrate the
main idea of our approach, one may consider the linear system of
\eqref{2.11} and note that the corresponding linear energy equality
reads as: for $j=0, 1, 2$,
\begin{align}\label{1.27}
&\frac{1}{2}\frac{d}{dt}\int_{\mathbb{R}^3}\frac{\beta_1}{\beta_2}\left(|\nabla^jn^+|^2
+|\nabla^j u^+|^2\right)+\frac{\beta_4}{\beta_3}\left(|\nabla^j n^-|^2+|\nabla^j u^-|^2\right)dx\notag\\
&\quad+\int_{\mathbb{R}^3}\frac{\beta_1}{\beta_2}\left(\nu_{1}^{+}\left|\nabla^{j+1}
u^{+}\right|^{2}+\nu_{2}^{+}\left|\nabla^{j} \operatorname{div}
u^{+}\right|^{2}\right)+\frac{\beta_4}{\beta_3}\left(\nu_{1}^{-}\left|\nabla^{j+1}
u^{-}\right|^{2}+\nu_{2}^{-}\left|\nabla^{j} \operatorname{div}
u^{-}\right|^{2}\right)dx\\
&\quad\quad\quad\quad\quad\hspace{2.15cm} +\left\langle\nabla^{j}
\nabla n^{-}, \beta_{1} \nabla^{j}
u^{+}\right\rangle+\left\langle\nabla^{j} \nabla n^{+}, \beta_{4}
\nabla^{j} u^{-}\right\rangle=0.\notag
\end{align}
Notice that the energy equality \eqref{1.27} only gives the
dissipative estimates for $u^+$ and $u^-$. In order to explore the
dissipative estimates for the variables $n^+$ and $n^-$, one may
follow the method of Evje, Wang and Wen \cite{Evje9} by constructing
the interactive energy functionals between $u^+$ and $n^+$, and
$u^-$ and $n^-$, respectively to get for $j=0,1$
\begin{align}\begin{split}\label{1.28}\displaystyle \frac{\mathrm{d}}{\mathrm{d} t} &
\displaystyle\left\{\frac{\nu_1^++\nu_2^+}{2\beta_1\beta_2}\int_{\mathbb{R}^3}\left|\nabla\nabla^j
n^{+}\right|^{2}dx
+\left\langle\nabla^{j} u^{+}, \frac{1}{\beta_{2}} \nabla\nabla^{j} n^{+}\right\rangle\right\}\\
&\quad+\frac{\beta_{1}}{\beta_{2}}\int_{\mathbb{R}^3}\left|\nabla\nabla^{j} n^{+}\right|^{2}dx+
\left\langle\nabla\nabla^j n^{-}, \nabla\nabla^{j} n^{+}\right\rangle\\
&\leq\displaystyle
C\left(\frac{\beta_{1}}{\beta_{2}}\left\|\nabla^j\operatorname{div}u^{+}\right\|^{2}_{L^2}+
\delta_0\left(\left\|\nabla\nabla^{j}
\left(n^+,n^{-}\right)\right\|^{2}_{L^2}+\left\|\nabla^{j}
u^{+}\right\|^{2}_{H^1}\right)\right),\end{split}\end{align}
\begin{align}\begin{split}\label{1.29}\displaystyle \frac{\mathrm{d}}{\mathrm{d} t} &
\displaystyle\left\{\frac{\nu_1^-+\nu_2^-}{2\beta_3\beta_4}\int_{\mathbb{R}^3}\left|\nabla\nabla^j
n^{-}\right|^{2}dx
+\left\langle\nabla^{j} u^{-}, \frac{1}{\beta_{3}} \nabla\nabla^{j} n^{-}\right\rangle\right\}\\
&\quad+\frac{\beta_{4}}{\beta_{3}}\int_{\mathbb{R}^3}\left|\nabla\nabla^{j}
n^{+}\right|^{2}dx+
\left\langle\nabla\nabla^j n^{+}, \nabla\nabla^{j} n^{-}\right\rangle\\
&\leq\displaystyle
C\left(\frac{\beta_{4}}{\beta_{3}}\left\|\nabla^j\operatorname{div}u^{-}\right\|^{2}_{L^2}+
\delta_0\left(\left\|\nabla\nabla^{j}
\left(n^+,n^{-}\right)\right\|^{2}_{L^2}+\left\|\nabla^{j}
u^{-}\right\|^{2}_{H^1}\right)\right).\end{split}\end{align} This
means that to obtain the dissipative estimates for
$\nabla^2(n^+,n^-)$, one needs to do the energy estimate \eqref{1.27}
at both the $1$th and the $2$th levels. Consequently, one can only
show that the convergence rate of the second--order (highest--order)
spatial derivative of solution, i.e. $\|\nabla^{2}(n^+, u^+, n^-,
u^-)\|_{L^2}$ is the same as that of $\|\nabla(n^+, u^+, n^-,
u^-)\|_{L^2}$ (See the proof of Theorem \ref{0mainth} in Evje, Wang
and Wen \cite{Evje9} for details). In particular, this implies that
the convergence rate of the highest--order spatial derivative of the
solution is not optimal. To tackle with this difficulty, we will
make full use of the benefit of the low--frequency and high--frequency
decomposition and the key linear $L^2$--estimates. More precisely,
instead of using \eqref{1.28}--\eqref{1.29}, we will employ the new
interactive energy functionals between $u^{+,h}$ and $n^{+,h}$, and
$n^{-,h}$ and $u^{-,h}$, respectively to obtain
\begin{align}\begin{split}\label{1.30}\displaystyle \frac{\mathrm{d}}{\mathrm{d} t} &
\displaystyle\left\{\frac{\nu_1^++\nu_2^+}{2\beta_1\beta_2}\left\|\nabla^2
n^{+, h}\right\|^{2}_{L^2}
+\left\langle\nabla u^{+, h}, \frac{1}{\beta_{2}} \nabla^{2} n^{+, h}\right\rangle\right\}\\
&\quad+\frac{\beta_{1}}{\beta_{2}}\left\|\nabla^{2} n^{+, h}\right\|^{2}_{L^2}+\left\langle\nabla^2 n^{-, h}, \nabla^{2} n^{+, h}\right\rangle\\
&\lesssim\displaystyle
\left\|\nabla\operatorname{div}u^{+}\right\|^{2}_{L^2}+
\delta_0\left(\left\|\nabla^{2}
\left(n^+,n^{-}\right)\right\|^{2}_{L^2}+\left\|\nabla^{2}
u^{+}\right\|^{2}_{H^1}\right),\end{split}\end{align}
\begin{align}\begin{split}\label{1.31}\displaystyle \frac{\mathrm{d}}{\mathrm{d} t} &
\displaystyle\left\{\frac{\nu_1^-+\nu_2^-}{2\beta_3\beta_4}\left\|\nabla^2
n^{-, h}\right\|^{2}_{L^2}
+\left\langle\nabla u^{-, h}, \frac{1}{\beta_{3}} \nabla^{2} n^{-, h}\right\rangle\right\}\\
&\quad+\frac{\beta_{4}}{\beta_{3}}\left\|\nabla^{2} n^{-, h}\right\|^{2}_{L^2}+\left\langle\nabla^2 n^{+, h}, \nabla^{2} n^{-, h}\right\rangle\\
&\lesssim\displaystyle
\left\|\nabla^j\operatorname{div}u^{-}\right\|^{2}_{L^2}+
\delta_0\left(\left\|\nabla^{2}
\left(n^+,n^{-}\right)\right\|^{2}_{L^2}+\left\|\nabla^{2}
u^{-}\right\|^{2}_{H^1}\right).\end{split}\end{align} Next, we
choose three sufficiently small positive constants $C_1$, $C_3$ and
$C_4$, and define the temporal energy functional
\begin{equation}\nonumber\begin{split}
{\mathcal{E}}(t)&= \frac{\beta_{1}}{2 \beta_{2}} \left\|\nabla^{2}
n^{+}\right\|^{2}_{L^2} +\frac{\left(\nu_{1}^{+}+\nu_{2}^{+}\right)
C_{1}}{2 \beta_{1} \beta_{2}}\left\|\nabla^{2} n^{+,
h}\right\|^{2}_{L^2} +\frac{\beta_{1}}{2 \beta_{2}}
\left\|\nabla^{2}
u^{+}\right\|^{2}_{L^2}+\frac{C_{1}}{\beta_{2}}\left\langle\nabla
u^{+, h}, \nabla^2n^{+, h}\right\rangle\\
&\quad+\frac{C_3\beta_4}{2\beta_3}\left\|\nabla^{2}
n^{-}\right\|^{2}_{L^2}
+\frac{(\nu_1^-+\nu_2^-)C_4}{2\beta_3\beta_4}\left\|\nabla^{2} n^{-,
h}\right\|^{2}_{L^2} +\frac{C_3\beta_{4}}{2 \beta_{3}}
\left\|\nabla^{2} u^{-}\right\|^{2}_{L^2}
+\frac{C_4}{\beta_{3}}\left\langle\nabla u^{-, h}, \nabla^2n^{-,
h}\right\rangle,\end{split}\end{equation} which is equivalent to
$\|\nabla^2(n^+, u^+, n^-, u^-)\|_{L^{2}}^2$. Taking $j=2$ in
\eqref{1.27} and combining \eqref{1.30}--\eqref{1.31}, we deduce that
 \begin{equation}\label{1.32}\frac{{d}}{{d}t}\mathcal{E}(t)+C\mathcal{E}(t)\lesssim
 \|\nabla^2(n^{+, l},u^{+, l},n^{-, l},u^{-, l})(t)\|^2_{L^{2}}.
\end{equation}
Consequently, by combining \eqref{1.32} with the optimal convergence
rate of $\|\nabla^2(n^{+, l},u^{+, l},n^{-, l},u^{-, l})\|_{L^{2}}$
obtained in Step 2, we can prove \eqref{3.3} and thus complete the
proof of the optimal decay rates for $\textbf{Case I: N=2}$.
However, for the nonlinear problem \eqref{2.11}, it is much more
complicated due to the nonlinear estimates. Compared to
\cite{Evje9}, our nonlinear energy estimate is new and different.
This is another point of this article. Indeed, in Lemma \ref{3.2}
and Lemma \ref{3.3} of \cite{Evje9}, the terms $\|\nabla u^{\pm}\|_{L^2}^2$ are
involved in the right--hand side of energy inequality. Therefore, we
can not follow the method of \cite{Evje9} directly to
close the energy estimate of $\nabla^2(n^{+},u^{+},n^{-},u^{-})$. To get around
this problem, one main observation of this paper is to take full
advantage of the smallness of $\|\nabla u^{\pm}\|_{L^\infty}$. More
precisely, by Sobolev inequality, one has
\begin{equation}\nonumber\|\nabla u^{\pm}\|_{L^\infty}\lesssim
\|\nabla^2 u^{\pm}\|_{L^2}^{\frac{1}{2}}\|\nabla^3
u^{\pm}\|_{L^2}^{\frac{1}{2}}.
\end{equation}
Noticing \eqref{3.2} and the fact that the term involving $\nabla^3
u^{\pm}$ can be absorbed by the left--hand side of energy
inequality, it is clear that the term $\|\nabla
u^{\pm}(t)\|_{L^\infty}$ can provide a smallness of the order
$\delta_0^{\frac{1}{2}}$. In view of this key observation, we can
close our desired energy estimates (See the proofs of
\eqref{3.12}--\eqref{3.13} and \eqref{3.27}--\eqref{3.28} for
details).
\par Fourth, we prove the optimal decay
rates for {\bf{Case II:} $\bf{N>2}$}. The main novelty of the
present paper in this step is to apply finite mathematical induction to close our energy estimates. More specifically, for any
$0\leq\ell\leq N$, we define the time--weighted energy functional as
\begin{equation}\nonumber \mathcal{E}_{\ell}^{N}(t)=\sup\limits_{0\leq\tau\leq t}\left
\{(1+\tau)^{\frac{3}{4}+\frac{\ell}{2}}\|\nabla^{\ell}(n^+, u^+,
n^-, u^-)(\tau)\|_{H^{N-\ell}}\right \}.
\end{equation}
Therefore, it suffices to show that for any $0\le \ell\le N$,
$\mathcal{E}_{\ell}^{N}(t)$ has an uniform time--independent bound.
We will employ mathematical induction to achieve this
goal (See the proofs of Lemma \ref{Lemma3.1} and Lemma
\ref{Lemma3.2}). We remark that the benefit of the low--frequency
and high--frequency decomposition and the key linear $L^2$--estimates
obtained in Step 2 play a crucial role in the process. Moreover,
compared to {\bf{Case I:} $\bf{N=2}$}, the energy estimates are much
complicated and subtle. For example, as in {\bf{Case I:}
$\bf{N=2}$}, we also need the smallness of the term $\|\nabla
u^{\pm}\|_{L^\infty}$. However, we can not follow the method used in
{\bf{Case I:} $\bf{N=2}$} since the term involving $\nabla^3
u^{\pm}$ can not be absorbed by the left--hand side of the energy
inequality. Instead, we will apply \eqref{3.2} with $N=3$ and
\eqref{3.3} to get
\begin{equation}\nonumber\|\nabla u^{\pm}\|_{L^\infty}\lesssim
\|\nabla^2 u^{\pm}\|_{L^2}^{\frac{1}{2}}\|\nabla^3
u^{\pm}\|_{L^2}^{\frac{1}{2}}\lesssim (1+t)^{-\frac{7}{8}},
\end{equation}
which implies that the term  $\|\nabla u^{\pm}\|_{L^\infty}$ can
provide a smallness of the order $(1+t)^{-\frac{7}{8}}$ if $t$ large
enough. We remark that the smallness of $\|\nabla
u^{\pm}\|_{L^\infty}$ plays a crucial role in closure of our energy estimates.

\par In the last step, we show the lower bounds on the convergence
rates of solutions. To do this, we will employ Plancherel theorem
and careful analysis on the solution semigroup. First, we derive the
convergence rate in $\dot H^{-1}$. Then, we can prove the lower
bound on the convergence rates in both $L^r$--norm for $2\leq r
\leq\infty$ and $L^2$--norm for the higher--order spatial
derivatives by an interpolation trick.\par The rest of this paper is
organized as follows. In the next section, we first rewrite the
Cauchy problem \eqref{1.9}--\eqref{1.10}. Then, we take Hodge
decomposition to the corresponding linear system, and get desired
linear estimates by making careful spectral analysis on the linear
system. In section 3, we first prove the optimal convergence rates
of the solutions for {\bf{Case I:} $\bf{N=2}$}. Then, we deal with
{\bf{Case II:} $\bf{N>2}$} by induction, and then show the
lower--bounds on the convergence rates in both $L^r$--norm for
$2\leq r \leq\infty$ and $L^2$--norm for the higher--order spatial
derivatives by an interpolation trick.

\bigskip

\section{Spectral analysis and linear $L^2$--estimates}
\setcounter{equation}{0}
\subsection{Reformulation} In this subsection, we first reformulate the system.
Setting   \[ n^{\pm}=R^{\pm}-1,
\]
system \eqref{1.9} can be rewritten as
\begin{equation}\label{2.1}
\left\{\begin{array}{l}
\partial_{t} n^++\operatorname{div}u^+=F_1, \\
\partial_{t}u^{+}+\alpha_1\nabla n^++\alpha_2\nabla
n^--\nu^+_1\Delta u^+-\nu^+_2\nabla\operatorname{div} u^+=F_2, \\
\partial_{t} n^-+\operatorname{div}u^-=F_3, \\
\partial_{t}u^{-}+\alpha_3\nabla n^++\alpha_4\nabla
n^--\nu^-_1\Delta u^--\nu^-_2\nabla\operatorname{div} u^-=F_4, \\
\end{array}\right.
\end{equation}
where $\nu_{1}^{\pm}=\frac{\mu^{\pm}}{\rho^{\pm}(1,1)}$,
$\nu_{2}^{\pm}=\frac{\mu^{\pm}+\lambda^{\pm}}{\rho^{\pm}(1,1)}>0$,
$\alpha_{1}=\frac{\mathcal{C}^{2}(1,1)
\rho^{-}(1,1)}{\rho^{+}(1,1)}$,
$\alpha_{2}=\mathcal{C}^{2}(1,1)+\frac{\mathcal{C}^{2}(1,1)
\alpha^{-}(1,1) f^{\prime}(1)}{s_{-}^{2}(1,1)}$,
$\alpha_{3}=\mathcal{C}^{2}(1,1)$,
$\alpha_{4}=\frac{\mathcal{C}^{2}(1,1)
\rho^{+}(1,1)}{\rho^{-}(1,1)}-\frac{\mathcal{C}^{2}(1,1)
\alpha^{+}(1,1) f^{\prime}(1)}{s_{+}^{2}(1,1)}$,
 and
the nonlinear terms are given by

\begin{align}
\label{2.2}F_{1}=&-\operatorname{div}\left(n^{+} u^{+}\right), \\
F_{2}^{i}=&-g_+\left(n^{+}, n^{-}\right) \partial_{i} n^{+}-\bar{g}_{+}\left(n^{+}, n^{-}\right) \partial_{i} n^{-}
-\left(u^{+} \cdot \nabla\right) u_{i}^{+} \nonumber\\
&+\mu^{+} h_{+}\left(n^{+}, n^{-}\right) \partial_{j}
n^{+} \partial_{j} u_{i}^{+}+\mu^{+} k_{+}\left(n^{+}, n^{-}\right) \partial_{j} n^{-} \partial_{j} u_{i}^{+} \nonumber\\
\label{2.3}&+\mu^{+} h_{+}\left(n^{+}, n^{-}\right) \partial_{j}
n^{+} \partial_{i} u_{j}^{+}+\mu^{+} k_{+}\left(n^{+}, n^{-}\right)
\partial_{j} n^{-}
 \partial_{i} u_{j}^{+}\\
&+\lambda^{+} h_{+}\left(n^{+}, n^{-}\right) \partial_{i} n^{+}
\partial_{j} u_{j}^{+}+\lambda^{+} k_{+}\left(n^{+},
n^{-}\right) \partial_{i} n^{-} \partial_{j} u_{j}^{+} \nonumber\\
&+\mu^{+} l_{+}\left(n^{+}, n^{-}\right) \partial_{j}^{2} u_{i}^{+}+\left(\mu^{+}+\lambda^{+}\right) l_{+}\left(n^{+}, n^{-}\right) \partial_{i}
 \partial_{j} u_{j}^{+}, \nonumber\\
\label{2.4}F_{3}=&-\operatorname{div}\left(n^{-} u^{-}\right), \\
F_{4}^{i}=&-g_-\left(n^{+}, n^{-}\right) \partial_{i} n^{-}-
\bar{g}_{-}\left(n^{+}, n^{-}\right) \partial_{i} n^{+}-\left(u^{-} \cdot \nabla\right) u_{i}^{-}\nonumber \\
&+\mu^{-} h_{-}\left(n^{+}, n^{-}\right) \partial_{j} n^{+} \partial_{j} u_{i}^{-}+\mu^{-} k_{-}\left(n^{+}, n^{-}\right)
\partial_{j} n^{-} \partial_{j} u_{i}^{-} \nonumber\\
\label{2.5}&+\mu^{-} h_{-}\left(n^{+}, n^{-}\right) \partial_{j}
n^{+}
\partial_{i} u_{j}^{-}+\mu^{-} k_{-}\left(n^{+}, n^{-}\right)
 \partial_{j} n^{-} \partial_{i} u_{j}^{-} \\
&+\lambda^{-} h_{-}\left(n^{+}, n^{-}\right) \partial_{i} n^{+} \partial_{j} u_{j}^{-}+\lambda^{-} k_{-}\left(n^{+}, n^{-}\right)
 \partial_{i} n^{-} \partial_{j} u_{j}^{-}\nonumber \\
&+\mu^{-} l_{-}\left(n^{+}, n^{-}\right) \partial_{j}^{2}
u_{i}^{-}+\left(\mu^{-}+\lambda^{-}\right) l_{-}\left(n^{+},
n^{-}\right) \partial_{i} \partial_{j} u_{j}^{-},\nonumber
\end{align}

where
\begin{equation}\label{2.6}
\left\{\begin{array}{l}
g_{+}\left(n^{+}, n^{-}\right)=\frac{\left(\mathcal{C}^{2} \rho^{-}\right)\left(n^{+}+1, n^{-}+1\right)}{\rho^{+}\left(n^{+}+1, n^{-}+1\right)}-\frac{\left(\mathcal{C}^{2} \rho^{-}\right)(1,1)}{\rho^{+}(1,1)}, \\
g_{-}\left(n^{+}, n^{-}\right)=\frac{\left(\mathcal{C}^{2} \rho^{+}\right)\left(n^{+}+1, n^{-}+1\right)}{\rho^{-}\left(n^{+}+1, n^{-}+1\right)}-\frac{\left(\mathcal{C}^{2} \rho^{+}\right)(1,1)}{\rho^{-}(1,1)}-\frac{f^{\prime}\left(n^{-}+1\right)\left(\mathcal{C}^{2} \alpha^{+}\right)\left(n^{+}+1, n^{-}+1\right)}{s_{+}^{2}\left(n^{+}+1, n^{-}+1\right)} \\
\hspace{2.2cm}+\frac{f^{\prime}(1)\left(\mathcal{C}^{2} \alpha^{+}\right)(1,1)}{s_{+}^{2}(1,1)}, \\
\end{array}\right.\end{equation}

\begin{equation}\label{2.7}
\left\{\begin{array}{l}
\bar{g}_{+}\left(n^{+}, n^{-}\right)=\mathcal{C}^{2}\left(n^{+}+1, n^{-}+1\right)-\mathcal{C}^{2}\left(1, 1\right)
+\frac{f^{\prime}\left(n^{-}+1\right)\left(\mathcal{C}^{2} \alpha^{-}\right)\left(n^{+}+1, n^{-}+1\right)}{s_{-}^{2}\left(n^{+}+1, n^{-}+1\right)}\\
\hspace{2.2cm}-\frac{f^{\prime}(1)\left(\mathcal{C}^{2} \alpha^{-}\right)(1,1)}{s_{-}^{2}(1,1)},\\
\bar{g}_{-}\left(n^{+}, n^{-}\right)=\mathcal{C}^{2}\left(n^{+}+1, n^{-}+1\right)-\mathcal{C}^{2}(1,1),\\
\end{array}\right.
\end{equation}

\begin{equation}\label{2.8}
\left\{\begin{array}{l}
h_{+}\left(n^{+}, n^{-}\right)=\frac{\left(\mathcal{C}^{2}\alpha^{-}\right)\left(n^{+}+1, n^{-}+1\right)}{(n^++1)s_{-}^{2}\left(n^{+}+1, n^{-}+1\right)},\\
h_{-}\left(n^{+}, n^{-}\right)=-\frac{\left(\mathcal{C}^{2} \right)\left(n^{+}+1, n^{-}+1\right)}{(\rho^-s_{-}^{2})\left(n^{+}+1, n^{-}+1\right)},
\end{array}\right.
\end{equation}

\begin{equation}\label{2.9}
\left\{\begin{array}{l}
k_{+}\left(n^{+}, n^{-}\right)=-\left[\frac{\mathcal{C}^{2}\left(n^{+}+1, n^{-}+1\right)}{(n^++1)(s_{+}^{2}\rho^+)\left(n^{+}+1, n^{-}+1\right)}+\frac{f^{\prime}(n^-+1)\mathcal{C}^{2}\left(n^{+}+1, n^{-}+1\right)}{(\rho^+\rho^-s_{+}^{2}s_{-}^{2})\left(n^{+}+1, n^{-}+1\right)}\right],\\
k_{-}\left(n^{+}, n^{-}\right)=-\frac{\left(\alpha^+\mathcal{C}^{2}\right)\left(n^{+}+1, n^{-}+1\right)}{(n^-+1)s_{+}^{2}\left(n^{+}+1, n^{-}+1\right)}+\frac{f^{\prime}(n^-+1)\left(\alpha^+\mathcal{C}^{2}\right)\left(n^{+}+1, n^{-}+1\right)}{(\rho^-s_{+}^{2}s_{-}^{2})\left(n^{+}+1, n^{-}+1\right)},\\
\end{array}\right.
\end{equation}

\begin{equation}\label{2.10}
l_{\pm}(n^+, n^-)=\frac{1}{\rho_{\pm}\left(n^{+}+1, n^{-}+1\right)}-\frac{1}{\rho_{\pm}\left(1, 1\right)}.
\end{equation}
Taking change of variables by
\[
n^{+} \rightarrow \alpha_{1} n^{+}, \quad u^{+} \rightarrow \sqrt{\alpha_{1} u^{+}}, \quad n^{-} \rightarrow \alpha_{4} n^{-}, \quad u^{-} \rightarrow \sqrt{\alpha_{4} u^{-}},
\]
and setting
\[
\beta_{1}=\sqrt{\alpha_{1}}, \quad \beta_{2}=\frac{\alpha_{2} \sqrt{\alpha_{1}}}{\alpha_{4}},
\quad \beta_{3}=\frac{\alpha_{3} \sqrt{\alpha_{4}}}{\alpha_{1}}, \quad \beta_{4}=\sqrt{\alpha_{4}}
\]
and
\[
\beta^{+}=\sqrt{\frac{\beta_{1}}{\beta_{2}}}, \quad \beta^{-}=\sqrt{\frac{\beta_{4}}{\beta_{3}}},
\]
the Cauchy problem \eqref{2.1} and \eqref{1.10} can be reformulated as
\begin{equation}\label{2.11}
\left\{\begin{array}{l}
\partial_{t} n^{+}+\beta_{1} \operatorname{div} u^{+}=\mathcal{F}_{1}, \\
\partial_{t} u^{+}+\beta_{1} \nabla n^{+}+\beta_{2} \nabla n^{-}-v_{1}^{+} \Delta u^{+}-v_{2}^{+} \nabla \operatorname{div} u^{+}=\mathcal{F}_{2}, \\
\partial_{t} n^{-}+\beta_{4} \operatorname{div} u^{-}=\mathcal{F}_{3}, \\
\partial_{t} u^{-}+\beta_{3} \nabla n^{+}+\beta_{4} \nabla n^{-}-v_{1}^{-} \Delta u^{-}-v_{2}^{-} \nabla \operatorname{div} u^{-}=\mathcal{F}_{4},
\end{array}\right.
\end{equation}
with initial data
\begin{equation}\label{2.12}
\left(n^{+}, u^{+}, n^{-}, u^{-}\right)(x, 0)=\left(n_{0}^{+},
u_{0}^{+}, n_{0}^{-}, u_{0}^{-}\right)(x) \rightarrow(0,
\overrightarrow{0}, 0, \overrightarrow{0}), \quad \text { as }|x|
\rightarrow+\infty,
\end{equation}
where the nonlinear terms are given by
\[
\mathcal{F}_{1}=\alpha_{1} F_{1}\left(\frac{n^{+}}{\alpha_{1}}, \frac{u^{+}}{\sqrt{\alpha_{1}}}\right), \quad \mathcal{F}_{2}=\sqrt{\alpha_{1}} F_{2}
\left(\frac{n^{+}}{\alpha_{1}}, \frac{u^{+}}{\sqrt{\alpha_{1}}}, \frac{n^{-}}{\alpha_{4}}, \frac{u^{-}}{\sqrt{\alpha_{4}}}\right),
\]
and
\[
\mathcal{F}_{3}=\alpha_{4} F_{3}\left(\frac{n^{-}}{\alpha_{4}}, \frac{u^{-}}{\sqrt{\alpha_{4}}}\right), \quad \mathcal{F}_{4}=\sqrt{\alpha_{4}} F_{4}\left(\frac{n^{+}}{\alpha_{1}}, \frac{u^{+}}{\sqrt{\alpha_{1}}}, \frac{n^{-}}{\alpha_{4}}, \frac{u^{-}}{\sqrt{\alpha_{4}}}\right).
\]
Noticing that
\begin{equation}\label{2.13}
\beta_{1} \beta_{4}-\beta_{2} \beta_{3}=-\frac{\mathcal{C}^{2}(1,1) f^{\prime}(1)}{\sqrt{\alpha_{1} \alpha_{4}} \rho^{+}(1,1)}>0,
\end{equation}
it is clear that $\beta^+\beta^->1$. It should be mentioned  that
the relation \eqref{2.13} is possible, since the representation of
capillary pressure includes $f\not\equiv 0$ which is a strictly
decreasing function near 1. We remark that the relation \eqref{2.13}
plays a fundamental role in the proof of Theorem \ref{1mainth}.\par
Define
$\tilde{U}=(\tilde{n}^+,\tilde{u}^+,\tilde{n}^-,\tilde{u}^-)^t$. In
terms of the semigroup theory for evolutionary equation, we will
investigate the following initial value problem for the
corresponding linearized system of \eqref{2.11}:
\begin{equation}
\begin{cases}
\tilde{U}_t=\mathcal B\tilde{U},\\
\tilde{U}\big|_{t=0}={U}_0,
\end{cases}   \label{2.14}
\end{equation}
where the operator $\mathcal B$ is given by
\begin{equation}\nonumber\mathcal B=\begin{pmatrix}
0&-\beta_1\text{div}&0&0\\
-\beta_1\nabla&\nu_{1}^{+} \Delta+\nu_{2}^{+} \nabla \otimes \nabla&-\beta_2\nabla&0\\
0&0&0&-\beta_4\text{div}\\
-\beta_3\nabla&0&-\beta_4\nabla&\nu_{1}^{-} \Delta+\nu_{2}^{-} \nabla \otimes \nabla
\end{pmatrix}.\end{equation}
Applying the Fourier transform to the system \eqref{2.14}, one has
\begin{equation}
\begin{cases}
\widehat {\widetilde{{U}}}_t=\mathcal A(\xi)\widehat {\widetilde{U}},\\
\widehat {\widetilde{U}}\big|_{t=0}=\widehat U_0,
\end{cases}   \label{2.15}
\end{equation}
where $\widehat {\widetilde{U}}(\xi,t)=\mathfrak{F}(\widetilde{U}(x,t))$, $\xi=(\xi^1,\xi^2,\xi^3)^t$ and $\mathcal A(\xi)$ is defined by
\begin{equation}\nonumber\mathcal A(\xi)=\begin{pmatrix}
0&-{i}\beta_1\xi^t&0&0\\
-{i}\beta_1\xi&-\nu_1^+|\xi|^2{\rm I}_{3\times 3}-\nu_2^+\xi\otimes\xi&-{i}\beta_2\xi&0\\
0&0&0&- {i}\beta_4\xi^t\\
- {i}\beta_3\xi&0&- {i}\beta_4\xi&-\nu_1^-|\xi|^2{\rm I}_{3\times 3}-\nu_2^-\xi\otimes\xi
\end{pmatrix}.\end{equation}

To derive the linear time--decay estimates,  by using a real method
as in \cite{Li1,Mat1}, one need to make a detailed analysis on the
properties of the semigroup. Unfortunately, it seems untractable,
since the system \eqref{2.15} has eight equations and the matrix
$\mathcal A(\xi)$ can not be diagonalizable (see \cite{Sideris}
pp.807 for example). To tackle this issue, we employ the Hodge
decomposition of the system \eqref{2.14} such that the system
\eqref{2.14} can be decoupled into three systems. One has four
equations whose characteristic polynomial possesses four distinct
roots, and the other two are classic heat equations. This key
observation allows us to derive the optimal linear convergence
rates.

 Let
$\varphi^{\pm}=\Lambda^{-1}{\rm div}\tilde{u}^{\pm}$
be the ``compressible part" of the velocities $\tilde{u}^{\pm}$, and denote $\phi^{\pm}=\Lambda^{-1}{\rm curl}\tilde{u}^{\pm}$ (with $({\rm curl} z)_i^j
=\partial_{x_j}z^i-\partial_{x_i}z^j$) by the ``incompressible part"
of the velocities $\tilde{u}^{\pm}$. Then, we can rewrite
the system \eqref{2.14} as follows:
\begin{equation}\label{2.16}
\begin{cases}
\partial_t{\tilde{n}^+}+\beta_1\Lambda{\varphi^+}=0,\\
\partial_t{\varphi^+}-\beta_1\Lambda{\tilde{n}^+}-\beta_2\Lambda{\tilde{n}^-}+\nu^+\Lambda^2{\varphi^+}=0,\\
\partial_t{\tilde{n}^-}+\beta_4\Lambda{\varphi^-}=0,\\
\partial_t{\varphi^-}-\beta_3\Lambda{\tilde{n}^+}-\beta_4\Lambda{\tilde{n}^-}+\nu^-\Lambda^2{\varphi^-}=0,\\
(\tilde{n}^+, \varphi^+, \tilde{n}^-, \varphi^-)\big|_{t=0}=({n}^+_0, \Lambda^{-1}{\rm div}{u}^{+}_0, {n}^-_0, \Lambda^{-1}{\rm div}{u}^{-}_0)(x),\\
\end{cases}
\end{equation}
and
\begin{equation}\label{2.17}
\begin{cases}
\partial_t\phi^++\nu^+_1\Lambda^2\phi^+=0,\\
\partial_t\phi^-+\nu^-_1\Lambda^2\phi^-=0,\\
(\phi^+,\phi^-)\big|_{t=0}=(\Lambda^{-1}{\rm curl}{u}^{+}_0, \Lambda^{-1}{\rm curl}{u}^{-}_0)(x),
\end{cases}
\end{equation}
where $\nu^{\pm}=\nu^{\pm}_1+\nu^{\pm}_2$.
\smallskip
\subsection{Spectral analysis for IVP \eqref{2.16}} In terms of the semigroup theory, we may represent the IVP \eqref{2.16} for $\mathcal U=(\tilde{n}^+, \varphi^+, \tilde{n}^-, \varphi^-)^t$ as
\begin{equation}
\begin{cases}
\mathcal U_t=\mathcal B_1\mathcal U,\\
\mathcal U\big|_{t=0}=\mathcal U_0,
\end{cases}   \label{2.18}
\end{equation}
where the operator $\mathcal B_1$ is defined
by
\begin{equation}\nonumber\mathcal B_1=\begin{pmatrix}
0&-\beta_1\Lambda&0&0\\
\beta_1\Lambda&-\nu^+\Lambda^2&\beta_2\Lambda&0\\
0&0&0&-\beta_4\Lambda\\
\beta_3\Lambda&0&\beta_4\Lambda&-\nu^-\Lambda^2
\end{pmatrix}.\end{equation}
Taking the Fourier transform to the system \eqref{2.18}, we obtain
\begin{equation}
\begin{cases}
\widehat {\mathcal U}_t=\mathcal A_1(\xi)\widehat {\mathcal U},\\
\widehat {\mathcal U}\big|_{t=0}=\widehat {\mathcal U}_0,
\end{cases}   \label{2.19}
\end{equation}
where $\widehat {\mathcal U}(\xi,t)=\mathfrak{F}({\mathcal U}(x,t))$  and $\mathcal A_1(\xi)$ is given by
\begin{equation}\nonumber\mathcal A_1(\xi)=\begin{pmatrix}
0&-\beta|\xi|&0&0\\
\beta_1|\xi|&-\nu^+|\xi|^2&\beta_2|\xi|&0\\
0&0&0&-\beta_4|\xi|\\
\beta_3|\xi|&0&\beta_4|\xi|&-\nu^-|\xi|^2
\end{pmatrix}.\end{equation}
We compute the eigenvalues of the matrix $\mathcal A_1(\xi)$  from the determinant
\begin{equation}\begin{split}\label{2.20}&{\rm det}(\lambda{\rm I}-\mathcal A_1(\xi))\\
&=\lambda^4+(\nu^+|\xi|^2+\nu^-|\xi|^2)\lambda^3+(\beta_1^2|\xi|^2+\beta_4^2|\xi|^2+\nu^+\nu^-|\xi|^4)\lambda^2+(\nu^+\beta_4^2|\xi|^4+\nu^-\beta_1^2|\xi|^4)\lambda\\
&\quad+\beta_1^2\beta_4^2|\xi|^4-\beta_1\beta_2\beta_3\beta_4|\xi|^4\\
&=0,
\end{split}\end{equation}
which implies that the matrix $\mathcal A_1(\xi)$ possesses four different
eigenvalues:
\begin{equation}\nonumber
 \lambda_1=\lambda_1(|\xi|),\quad \lambda_2=\lambda_2(|\xi|),\quad
 \lambda_3=\lambda_3(|\xi|),\quad \lambda_4=\lambda_4(|\xi|).
\end{equation}
Consequently, the semigroup $e^{t\mathcal A_1}$ can be decomposed
into
\begin{equation}\label{2.21}
 \text{e}^{t\mathcal A_1(\xi)}=\sum_{i=1}^4\text{e}^{\lambda_it}P_i(\xi),
\end{equation}
where the projector $P_i(\xi)$ is defined by
\begin{equation}\label{2.22}
 P_i(\xi)=\prod_{j\neq i}\frac{\mathcal A_1(\xi)-\lambda_jI}{\lambda_i-\lambda_j}, \quad i,j=1,2,3,4.
\end{equation}
Thus, the solution of IVP \eqref{2.19} can be expressed as
\begin{equation}
\widehat {\mathcal U}(\xi,t)=\text{e}^{t\mathcal A_1(\xi)}\widehat {\mathcal U}_0(\xi)=\left(\sum_{i=1}^4
\text{e}^{\lambda_it}P_i(\xi)\right)\widehat {\mathcal U}_0(\xi).\label{2.23}
\end{equation}

To derive long time properties of  the semigroup
$\text{e}^{t\mathcal A_1}$ in $L^2$--framework, one need to analyze
the asymptotical expansions of $\lambda_i$, $P_i$ $(i =1, 2, 3, 4)$
and $\text{e}^{t\mathcal A_1(\xi)}$ in the low--frequency part.
Employing the similar argument of Taylor series expansion as in
\cite{Li1,Mat1}, we have the following lemma from tedious
calculations.
\begin{Lemma}\label{lemma2.1}
There exists a positive constants $\eta_1\ll 1 $ such that, for
$|\xi|\leq \eta_1$, the spectral has the following Taylor series
expansion:
\begin{equation}\label{2.24}
\left\{\begin{array}{lll}\displaystyle \lambda_1=-\left[\frac{\nu^++\nu^-}{4}-\frac{\nu^+(\beta_1^2-\beta_4^2)+\nu^-(\beta_4^2-\beta_1^2)}{8\kappa_1}\right]|\xi|^2+\sqrt{\kappa_2-\kappa_1}\text{i}|\xi|+\mathcal O(|\xi|^3),\\
\displaystyle\lambda_2=-\left[\frac{\nu^++\nu^-}{4}-\frac{\nu^+(\beta_1^2-\beta_4^2)+\nu^-(\beta_4^2-\beta_1^2)}{8\kappa_1}\right]|\xi|^2-\sqrt{\kappa_2-\kappa_1}\text{i}|\xi|+\mathcal O(|\xi|^3),
\\ \displaystyle  \lambda_3=-\left[\frac{\nu^++\nu^-}{4}+\frac{\nu^+(\beta_1^2-\beta_4^2)+\nu^-(\beta_4^2-\beta_1^2)}{8\kappa_1}\right]|\xi|^2
+\sqrt{\kappa_2+\kappa_1}\text{i}|\xi|+\mathcal O(|\xi|^3),
\\ \displaystyle  \lambda_4=-\left[\frac{\nu^++\nu^-}{4}+\frac{\nu^+(\beta_1^2-\beta_4^2)+\nu^-(\beta_4^2-\beta_1^2)}{8\kappa_1}\right]|\xi|^2-\sqrt{\kappa_2+\kappa_1}\text{i}|\xi|+\mathcal O(|\xi|^3),
\end{array}\right.
\end{equation}
where $\kappa_1=\sqrt{\frac{(\beta_1^2-\beta_4^2)^2}{4}+\beta_1\beta_2\beta_3\beta_4}$ and $\displaystyle\kappa_2=\frac{\beta_1^2+\beta_4^2}{2}$.
\end{Lemma}

By virtue of \eqref{2.22}--\eqref{2.24}, we can establish the
following estimates for the low--frequency part of the solutions
$\widehat{\mathcal U}(t,\xi)$ to the IVP \eqref{2.19}:
\begin{Lemma}\label{lemma2.2} Let $\displaystyle\bar{\nu}_1=\frac{\nu^++\nu^-}{4}-\frac{\nu^+(\beta_1^2-\beta_4^2)+\nu^-(\beta_4^2-\beta_1^2)}{8\kappa_1}>0$,
 we have
\begin{equation}\label{2.25}
 |\widehat{\tilde{n}^+}|,~|\widehat{\varphi^+}|,~ |\widehat{\tilde{n}^-}|,~ |\widehat{\varphi^-}|\lesssim \text{e}^{-\bar{\nu}_1|\xi|^2t}
 ( |\widehat{{n}^+_0}|+|\widehat{\varphi^+_0}|+ |\widehat{{n}^-_0}|+|\widehat{\varphi^-_0}|),
\end{equation}
for any $|\xi|\le \eta_1$.
\end{Lemma}
\begin{proof} By virtue of formula \eqref{2.23} and Taylor series expansion of $\lambda_i$ $(1\leq i\leq 4)$ in \eqref{2.24}, we can represent $P_i$ $(1\leq i\leq 4)$ as follows:
{\small \begin{equation}\label{2.26}\begin{split}P_1(\xi)=&\begin{pmatrix}
\frac{2\kappa_1+\beta_4^2-\beta_1^2}{8\kappa_1}&
\frac{\beta_1(2\kappa_1+\beta_4^2-\beta_1^2)}{8\kappa_1\sqrt{\kappa_2-\kappa_1}}{i}&\frac{-\beta_1\beta_2}{4\kappa_1}&\frac{-\beta_1\beta_2\beta_4}{4\kappa_1\sqrt{\kappa_2-\kappa_1}}{i}\\
\frac{\beta_1(\beta_1^2-\beta_4^2-2\kappa_1)+2\beta_2\beta_3\beta_4}{8\kappa_1
\sqrt{\kappa_2-\kappa_1}}{i}&\frac{2\kappa_1+\beta_4^2-\beta_1^2}{8\kappa_1}&\frac{\beta_2\sqrt{\kappa_2-\kappa_1}}{4\kappa_1}{i}&\frac{-\beta_2\beta_4}{4\kappa_1}\\
\frac{-\beta_3\beta_4}{4\kappa_1}&\frac{-\beta_1\beta_3\beta_4}{4\kappa_1\sqrt{\kappa_2-\kappa_1}}{i}&\frac{2\kappa_1+\beta_1^2-\beta_4^2}{8\kappa_1}&\frac{\beta_4(2\kappa_1
+\beta_1^2-\beta_4^2)}{8\kappa_1\sqrt{\kappa_2-\kappa_1}}{i}\\
\frac{\beta_3\sqrt{\kappa_2-\kappa_1}}{4\kappa_1}{i}&\frac{-\beta_1\beta_3}{4\kappa_1}&\frac{\beta_4(\beta_4^2-\beta_1^2-2\kappa_1)+2\beta_1\beta_2\beta_3}{8\kappa_1\sqrt{\kappa_2-
\kappa_1}}{i}&\frac{2\kappa_1+\beta_1^2-\beta_4^2}{8\kappa_1}
\end{pmatrix}+\mathcal O(|\xi|),\end{split}\end{equation}

\begin{equation}\label{2.27}\begin{split}P_2(\xi)=&\begin{pmatrix}
\frac{2\kappa_1+\beta_4^2-\beta_1^2}{8\kappa_1}&\frac{-\beta_1(2\kappa_1+\beta_4^2-\beta_1^2)}{8\kappa_1\sqrt{\kappa_2-\kappa_1}}{i}&
\frac{-\beta_1\beta_2}{4\kappa_1}&\frac{\beta_1\beta_2\beta_4}{4\kappa_1\sqrt{\kappa_2-\kappa_1}}{i}\\
-\frac{\beta_1(\beta_1^2-\beta_4^2-2\kappa_1)+2\beta_2\beta_3\beta_4}{8\kappa_1\sqrt{\kappa_2-\kappa_1}}{i}&
\frac{2\kappa_1+\beta_4^2-\beta_1^2}{8\kappa_1}&-\frac{\beta_2\sqrt{\kappa_2-\kappa_1}}{4\kappa_1}{i}&\frac{-\beta_2\beta_4}{4\kappa_1}\\
\frac{-\beta_3\beta_4}{4\kappa_1}&\frac{\beta_1\beta_3\beta_4}{4\kappa_1\sqrt{\kappa_2-\kappa_1}}{i}&\frac{2\kappa_1+\beta_1^2-\beta_4^2}
{8\kappa_1}&-\frac{\beta_4(2\kappa_1+\beta_1^2-\beta_4^2)}{8\kappa_1\sqrt{\kappa_2-\kappa_1}}{i}\\
-\frac{\beta_3\sqrt{\kappa_2-\kappa_1}}{4\kappa_1}{i}&\frac{-\beta_1\beta_3}{4\kappa_1}&-\frac{\beta_4(\beta_4^2-\beta_1^2-2\kappa_1)+2\beta_1\beta_2\beta_3}
{8\kappa_1\sqrt{\kappa_2-\kappa_1}}{i}&\frac{2\kappa_1+\beta_1^2-\beta_4^2}{8\kappa_1}
\end{pmatrix}+\mathcal O(|\xi|),\end{split}\end{equation}

\begin{equation}\label{2.28}\begin{split}P_3(\xi)=&\begin{pmatrix}
\frac{2\kappa_1+\beta_1^2-\beta_4^2}{8\kappa_1}&\frac{\beta_1(2\kappa_1+\beta_1^2-\beta_4^2)}
{8\kappa_1\sqrt{\kappa_2+\kappa_1}}{i}&\frac{\beta_1\beta_2}{4\kappa_1}&\frac{\beta_1\beta_2\beta_4}{4\kappa_1\sqrt{\kappa_2+\kappa_1}}{i}\\
-\frac{\beta_1(\beta_1^2-\beta_4^2+2\kappa_1)+2\beta_2\beta_3\beta_4}{8\kappa_1\sqrt{\kappa_2+\kappa_1}}{i}&\frac{2\kappa_1+\beta_1^2-\beta_4^2}
{8\kappa_1}&-\frac{\beta_2\sqrt{\kappa_2+\kappa_1}}{4\kappa_1}{i}&\frac{\beta_2\beta_4}{4\kappa_1}\\
\frac{\beta_3\beta_4}{4\kappa_1}&\frac{\beta_1\beta_3\beta_4}{4\kappa_1\sqrt{\kappa_2+\kappa_1}}{i}&
\frac{2\kappa_1+\beta_4^2-\beta_1^2}{8\kappa_1}&\frac{\beta_4(2\kappa_1+\beta_4^2-\beta_1^2)}{8\kappa_1\sqrt{\kappa_2+\kappa_1}}{i}\\
-\frac{\beta_3\sqrt{\kappa_2+\kappa_1}}{4\kappa_1}{i}&\frac{\beta_1\beta_3}{4\kappa_1}&-\frac{\beta_4(\beta_4^2-\beta_1^2+2\kappa_1)
+2\beta_1\beta_2\beta_3}{8\kappa_1\sqrt{\kappa_2-\kappa_1}}{i}&\frac{2\kappa_1+\beta_4^2-\beta_1^2}{8\kappa_1}
\end{pmatrix}+\mathcal O(|\xi|),\end{split}\end{equation}

and
\begin{equation}\label{2.29}\begin{split}P_4(\xi)=&\begin{pmatrix}
\frac{2\kappa_1+\beta_1^2-\beta_4^2}{8\kappa_1}&-\frac{\beta_1(2\kappa_1+\beta_1^2-\beta_4^2)}
{8\kappa_1\sqrt{\kappa_2+\kappa_1}}{i}&\frac{\beta_1\beta_2}{4\kappa_1}&-\frac{\beta_1\beta_2\beta_4}{4\kappa_1\sqrt{\kappa_2+\kappa_1}}{i}\\
\frac{\beta_1(\beta_1^2-\beta_4^2+2\kappa_1)+2\beta_2\beta_3\beta_4}{8\kappa_1\sqrt{\kappa_2+\kappa_1}}{i}&\frac{2\kappa_1+\beta_1^2-\beta_4^2}
{8\kappa_1}&\frac{\beta_2\sqrt{\kappa_2+\kappa_1}}{4\kappa_1}{i}&\frac{\beta_2\beta_4}{4\kappa_1}\\
\frac{\beta_3\beta_4}{4\kappa_1}&-\frac{\beta_1\beta_3\beta_4}{4\kappa_1\sqrt{\kappa_2+\kappa_1}}{i}&
\frac{2\kappa_1+\beta_4^2-\beta_1^2}{8\kappa_1}&-\frac{\beta_4(2\kappa_1+\beta_4^2-\beta_1^2)}{8\kappa_1\sqrt{\kappa_2+\kappa_1}}{i}\\
\frac{\beta_3\sqrt{\kappa_2+\kappa_1}}{4\kappa_1}{i}&\frac{\beta_1\beta_3}{4\kappa_1}&\frac{\beta_4(\beta_4^2-\beta_1^2+2\kappa_1)+2\beta_1\beta_2\beta_3}
{8\kappa_1\sqrt{\kappa_2-\kappa_1}}{i}&\frac{2\kappa_1+\beta_4^2-\beta_1^2}{8\kappa_1}
\end{pmatrix}+\mathcal O(|\xi|),\end{split}\end{equation}}
for any $|\xi|\le \eta_1$. Therefore, \eqref{2.25} follows from
\eqref{2.23}-\eqref{2.24} and \eqref{2.26}-\eqref{2.29} immediately.
\end{proof}

With the help of \eqref{2.25}, we can get the following proposition which is concerned with
the optimal $L^2$--convergence rate on the low--frequency part of the
solution.
\begin{Proposition}[$L^2$--theory]\label{Prop2.3} For any $k> -\frac{3}{2}$, it holds that
\begin{equation}\|\nabla ^k\text{e}^{t\mathcal{A}_1}*\mathcal U^l(0)\|_{L^2}\lesssim(1+t)^{-\frac{3}{4}-\frac{k}{2}}\|\widehat {\mathcal U}^l(0)\|_{L^\infty},\label{2.30}\end{equation}
 for any $t\geq 0$.
\end{Proposition}
\begin{proof} Due to \eqref{2.25} and Plancherel theorem, we have
\begin{equation}\nonumber \begin{split}\|\nabla ^k\text{e}^{t\mathcal{A}_1}*\mathcal U^l(0)\|_{L^2}^2=
&~\big{\|}|\xi|^k\text{e}^{t\mathcal{A}_1(\xi)}\widehat{\mathcal U}^l(0)\big{\|}_{L^2}^2
\\ \lesssim &\int_{|\xi|\le \eta_1} \text{e}^{-2\bar{\nu}_1|\xi|^2t}|\xi|^{2k}|\widehat{\mathcal U}^l(0)|^2\mathrm{d}\xi \\
\lesssim &~(1+t)^{-\frac{3}{4}-\frac{k}{2}}\|\widehat {\mathcal
U}^l(0)\|_{L^\infty}^2,
\end{split}\end{equation}
which implies \eqref{2.30}. Therefore, the proof of Proposition
\ref{Prop2.3} has been completed.
\end{proof}

It is worth mentioning that the $L^2$-convergence rates derived
above are optimal. Indeed, we have the lower-bound on the
convergence rates which is stated in the following proposition.

\begin{Proposition}\label{Prop2.4} Assume that $({n}^+_0,\varphi^+_0,{n}^-_0,\varphi^-_0)\in L^1$ satisfies \begin{equation}\widehat{\varphi^+_0}(\xi)=\widehat{{n}^-_0}(\xi)=\widehat{\varphi^-_0}(\xi)=0 \quad \text{and} \quad |\widehat{{n}^+_0}(\xi)|\ge c_0,\label{2.31}\end{equation} for any $|\xi|\le\eta_1$. Then the global
solution $(\tilde{n}^+,\varphi^+,\tilde{n}^-,\varphi^-)$ of the IVP \eqref{2.19}
satisfies
\begin{equation}\min\left\{\|\tilde{n}^{+,l}\|_{L^2},\|{\varphi}^{+,l}\|_{L^2},\|\tilde{n}^{-,l}\|_{L^2},\|{\varphi}^{-,l}\|_{L^2}\right\}\gtrsim
c_0(1+t)^{-\frac{3}{4}}.\label{2.32}\end{equation}
for large enough $t$ .
\end{Proposition}
\begin{proof} Let $\bar{\nu}_2=\frac{\nu^++\nu^-}{4}+\frac{\nu^+(\beta_1^2-\beta_4^2)+\nu^-(\beta_4^2-\beta_1^2)}{8\kappa_1}>0$. Due to \eqref{2.31},
it follows from \eqref{2.23}--\eqref{2.24} and
\eqref{2.26}--\eqref{2.29} that
\begin{equation}\begin{split}\nonumber\widehat{\tilde{n}^{+,l}}\sim &\frac{2\kappa_1+\beta_4^2-
\beta_1^2}{4\kappa_1}\text{e}^{-\bar{\nu}_1|\xi|^2t}\cos\left(\sqrt{\kappa_2-\kappa_1}
|\xi|t+\mathcal
O(|\xi|^3)t\right)\widehat{\tilde{n}_0^{+,l}}\\&+\frac{2\kappa_1+\beta_1^2-\beta_4^2}{4\kappa_1}
\text{e}^{-\bar{\nu}_2|\xi|^2t}\cos\left(\sqrt{\kappa_2+\kappa_1}|\xi|t+\mathcal
O(|\xi|^3)t\right)\widehat{\tilde{n}_0^{+,l}}.\end{split}\end{equation}
This together with Plancherel theorem and the double angle formula
gives that
\begin{equation}\label{2.33}\begin{split}\|\tilde{n}^{+,l}\|_{L^2}^2=&~\|\widehat{\tilde{n}^{+,l}}\|_{L^2}^2\\
\ge
&~\frac{(2\kappa_1+\beta_4^2-\beta_1^2)^2c_0^2}{16\kappa_1^2}\int_{|\xi|\le\frac{\eta_1}{2}}\text{e}^{-2\bar{\nu}_1|\xi|^2t}\cos^2
\left[\sqrt{\kappa_2-\kappa_1}|\xi|+O(|\xi|^3)\right]t{d}\xi
\\ &+\frac{(2\kappa_1+\beta_1^2-\beta_4^2)^2c_0^2}{16\kappa_1^2}\int_{|\xi|\le\frac{\eta_1}{2}}\text{e}^{-
2\bar{\nu}_2|\xi|^2t}\cos^2
\left[\sqrt{\kappa_2+\kappa_1}|\xi|+O(|\xi|^3)\right]t{d}\xi\\& +\frac{[4\kappa_1^2+(\beta_1^2-\beta_4^2)^2]c_0^2}{8\kappa_1^2}\int_{|\xi|\le\frac{\eta_1}{2}}\text{e}^{-\frac{\nu^++\nu^-}{2}|\xi|^2t}\cos
\left[\sqrt{\kappa_2-\kappa_1}|\xi|+O(|\xi|^3)\right]t\\
&\quad\cdot\cos\left[\sqrt{\kappa_2+\kappa_1}|\xi|+O(|\xi|^3)\right]t{d}\xi\\
=&~\frac{(2\kappa_1+\beta_4^2-\beta_1^2)^2c_0^2}{32\kappa_1^2}\int_{|\xi|\le\frac{\eta_1}{2}}\text{e}^{-2\bar{\nu}_1|\xi|^2t}{d}\xi+
\frac{(2\kappa_1+\beta_1^2-\beta_4^2)^2c_0^2}{32\kappa_1^2}\int_{|\xi|\le\frac{\eta_1}{2}}\text{e}^{-
2\bar{\nu}_2|\xi|^2t}{d}\xi
\\ &+~\frac{(2\kappa_1+\beta_4^2-\beta_1^2)^2c_0^2}{32\kappa_1^2}\int_{|\xi|\le\frac{\eta_1}{2}}\text{e}^{-2\bar{\nu}_1|\xi|^2t}
\cos\left[{2{\sqrt{\kappa_2-\kappa_1}|\xi|+O(|\xi|^3)}}\right]t{d}\xi
\\ &+\frac{(2\kappa_1+\beta_1^2-\beta_4^2)^2c_0^2}{32\kappa_1^2}\int_{|\xi|\le\frac{\eta_1}{2}}\text{e}^{-
2\bar{\nu}_2|\xi|^2t}{\cos
\left[2\sqrt{\kappa_2+\kappa_1}|\xi|+O(|\xi|^3)\right]t}{d}\xi
\\& +\frac{[4\kappa_1^2+(\beta_1^2-\beta_4^2)^2]c_0^2}{16\kappa_1^2}\int_{|\xi|\le\frac{\eta_1}{2}}\text{e}^{-\frac{\nu^++\nu^-}{2}|\xi|^2t}{\cos
\left[\sqrt{\kappa_2-\kappa_1}|\xi|+\sqrt{\kappa_2+\kappa_1}|\xi|+O(|\xi|^3)\right]t}{d}\xi\\& +\frac{[4\kappa_1^2+(\beta_1^2-\beta_4^2)^2]c_0^2}{16\kappa_1^2}\int_{|\xi|\le\frac{\eta_1}{2}}\text{e}^{-\frac{\nu^++\nu^-}{2}|\xi|^2t}{\cos
\left[\sqrt{\kappa_2-\kappa_1}|\xi|-\sqrt{\kappa_2+\kappa_1}|\xi|+O(|\xi|^3)\right]t}{d}\xi
 \\ \gtrsim &c_0^2(1+t)^{-\frac{3}{2}},\end{split}\end{equation}
 if $t$ large enough.
 Using a similar procedure as in \eqref{2.33} to handle $\|(\varphi^{+,l},\tilde{n}^{-,l},\varphi^{-,l})\|_{L^2}$, one has \eqref{2.32}. Therefore, we have completed the proof of
 Proposition \ref{Prop2.4}.
\end{proof}

\smallskip
From the classic theory of the heat equation, it is clear that the solution $\mathcal V=(\phi^+, \phi^-)^t$  to the IVP \eqref{2.17} satisfies the following decay
estimates.
\begin{Proposition}[$L^2$--theory]\label{Prop2.5} For any $k> -\frac{3}{2}$, there exists a positive
constant $C$ which is independent of $t$ such that
\begin{equation}\|\nabla ^k\text{e}^{t\mathcal{A}_2}*\mathcal V^l(0)\|_{L^2}\leq
C(1+t)^{-\frac{3}{4}-\frac{k}{2}}\|\widehat {\mathcal V}^l(0)\|_{L^\infty},\nonumber\end{equation}
 for any $t\geq 0$.
\end{Proposition}

\smallskip

By virtue of the definition of $\varphi^{\pm}$ and $\phi^{\pm}$, and the
fact that the relations
$$\tilde{u}^{\pm}=-\wedge^{-1}\nabla\varphi^{\pm}-\wedge^{-1}\text{div}\phi^{\pm}$$
involve pseudo-differential operators of degree zero, the estimates
in space $H^k(\mathbb R^3)$ for the original function $\tilde{u}^{\pm}$ will
be the same as for $(\varphi^{\pm}, \phi^{\pm})$. Combining Propositions
\ref{Prop2.3}, \ref{Prop2.4} and  \ref{Prop2.5}, we have the
following result concerning long time properties for the solution
semigroup $\text{e}^{-t\mathcal{A}}$.
\begin{Proposition}\label{Prop2.6} For any $k> -\frac{3}{2}$ and $2\leq r\leq \infty$. Assume
that the initial data $U_0\in L^1(\mathbb R^3)$,   then for any $t\ge 0$, the global solution $\tilde{U}=(\tilde{n}^+,\tilde{u}^+,\tilde{n}^-,\tilde{u}^-)^t$ of the IVP \eqref{2.14} satisfies
\begin{equation}\|\nabla ^k\text{e}^{t\mathcal{B}}\tilde{U}^l(0)\|_{L^2}\leq
C(1+t)^{-\frac{3}{4}-\frac{k}{2}}\|  \widehat
{\tilde{U}}^l(0)\|_{L^\infty}\leq
C(1+t)^{-\frac{3}{4}-\frac{k}{2}}\|
U(0)\|_{L^1}.\label{2.34}\end{equation} If additionally the initial
data satisfies \eqref{1.24}, we also have the following lower-bound
on convergence rate:
\begin{equation}\min\left\{\|\tilde{n}^{+,l}(t)\|_{L^2},\|\tilde{u}^{+,l}(t)\|_{L^2},\|\tilde{n}^{-,l}(t)\|_{L^2},\|\tilde{u}^{-,l}(t)\|_{L^2}\right\}\geq C_1N_0\sqrt{\delta_0}(1+t)^{-\frac{3}{4}},\label{2.35}
\end{equation}
 if $t$ large enough.
\end{Proposition}

\bigskip

\section{Optimal convergence rate}\label{1section_appendix}

In this section, we shall show the optimal convergence rate of the
solution stated in Theorem \ref{1mainth}. The global existence and
uniqueness of the solution in $H^2$ to the Cauchy problem
\eqref{2.11}--\eqref{2.12} have been proven in \cite{Evje9} based on
the classical energy method developed in \cite{Li1, Mat1}. Thus, we
can follow the proof of \cite{Evje9} step by step to obtain the
global existence and uniqueness of the solution in $H^N$ with an
integer $N\geq 2$, and thus we omit the details for the sake of
simplicity.

\smallskip
\begin{Theorem}\label{3.1mainth}

 Assume that $(n^+_0,u^+_0,n^-_0, u^-_0)\in H^N(\mathbb R^3)$ for an integer $N \geq 2$ and \eqref{1.12} holds.  There exists a constant $\delta_0>0$ such that if
\begin{equation}\label{3.1}\|(n^+_0,u^+_0,n^-_0, u^-_0)\|_{H^2}\leq \delta_0,\end{equation}
then the Cauchy problem \eqref{2.11}--\eqref{2.12} admits a unique globally classical solution $(n^+,u^+,n^-, u^-)$ such that for any $t\in[0,\infty),$
\begin{equation}\label{3.2}\begin{split}\|(n^+,u^+,n^-, u^-)(t)\|_{H^N}^2&+\int_0^t\left(\|\nabla(n^+, n^-)(\tau)\|^2_{H^{N-1}}+\|\nabla(u^+, u^-)(\tau)\|_{H^N}^2\right)\mathrm{d}\tau\\ \lesssim &~\|(n^+_0,u^+_0,n^-_0, u^-_0)\|_{H^N}^2.
\end{split}\end{equation}
\end{Theorem}

\bigskip
In what follows, we focus our attention on the proof of the optimal
convergence rate of the solution stated in Theorem \ref{1mainth}. We
first prove the upper-bound on the optimal convergence rate of the
solution stated in \eqref{1.22}--\eqref{1.23}. We will split the
proof into two cases: \textbf{Case I: $N=2$} and \textbf{Case II:
$N>2$}. To begin with, we deal with \textbf{Case I: $N=2$}. Owing to
Theorem \ref{0mainth}, it suffices to prove the following theorem.

\begin{Theorem}\label{3.2mainth}({\bf{Case I:} $\bf{N=2}$}) Assume that the hypotheses of Theorem \ref{3.1mainth}
and \eqref{1.21} are in force. Then there exists a positive constant
$C$, which is independent of $t$, such that
\begin{equation}\label{3.3}\|\nabla^2(n^+,u^+,n^-, u^-)(t)\|_{L^2}\leq ~C(1+t)^{-\frac{7}{4}},
\end{equation}
for all $t\geq 0$.
\end{Theorem}
\begin{proof}
We will make full use of the benefit of the low--frequency and
high--frequency decomposition to prove Theorem \ref{3.2mainth}.The
process involves the following three steps.\par
\textbf{{Step 1. $L^2$--estimate of $\nabla^2(n^+,u^+,n^-, u^-)$.}}
Multiplying $\nabla^{2}\eqref{2.11}_{1}, \nabla^{2}\eqref{2.11}_{2}, \nabla^{2}\eqref{2.11}_{3}$ and $\nabla^{2}\eqref{2.11}_{4}$ by $\frac{\beta_1}{\beta_{2}} \nabla^{2} n^{+}$,
$\frac{\beta_1}{\beta_2} \nabla^{2} u^{+}, \frac{\beta_4}{\beta_3} \nabla^{2} n^{-}$ and $\frac{\beta_4}{\beta_3} \nabla^{2} u^{-}$ respectively, and then integrating over $\mathbb{R}^{3},$ we obtain

\begin{align}\begin{split}\label{3.4}
\displaystyle\frac{1}{2} \frac{d}{d t} &\left\{\frac{\beta_1}{\beta_2}\left\|\nabla^{2} n^{+}\right\|^{2}+\frac{\beta_1}{\beta_2}\left\|\nabla^{2} u^{+}\right\|^{2}\right\}+\frac{\beta_1}{\beta_2}\left(\nu_{1}^{+}\left\|\nabla^{3} u^{+}\right\|^{2}+\nu_{2}^{+}\left\|\nabla^{2} \operatorname{div} u^{+}\right\|^{2}\right) \\
&=\left\langle\nabla^{2} \mathcal{{F}}_{1}, \frac{\beta_1}{\beta_2} \nabla^{2} n^{+}\right\rangle+\left\langle\nabla^{2} \mathcal{F}_{2}, \frac{\beta_1}{\beta_2} \nabla^{2} u^{+}\right\rangle-\left\langle\nabla^{2} \nabla n^{-}, \beta_{1} \nabla^{2} u^{+}\right\rangle \\
&=:I_{1}+I_{2}+I_{3},
\end{split}\end{align}

and
\begin{align}\begin{split}\label{3.5}
\frac{1}{2} \frac{\mathrm{d}}{\mathrm{d} t}&\left\{\frac{\beta_4}{\beta_3}\left\|\nabla^{2} n^{-}\right\|^{2}+\frac{\beta_4}{\beta_3}\left\|\nabla^{2} u^{-}\right\|^{2}\right\}+\frac{\beta_4}{\beta_3}\left(\nu_{1}^{-}\left\|\nabla^{3}u^{-}\right\|^{2}+\nu_{2}^{-}\left\|\nabla^{2} \operatorname{div} u^{-}\right\|^{2}\right)\\
&=\left\langle\nabla^{2} \mathcal{F}_{3}, \frac{\beta_{4}}{\beta_{3}} \nabla^{2} n^{-}\right\rangle+\left\langle\nabla^{2} \mathcal{F}_{4}, \frac{\beta_{4}}{\beta_{3}} \nabla^{2} u^{-}\right\rangle-\left\langle\nabla^{2} \nabla n^{+}, \beta_{4} \nabla^{2} u^{-}\right\rangle\\
&=:I_{4}+I_{5}+I_{6}.\end{split}\end{align}
\par \noindent By virtue of \eqref{3.2} and Lemmas \ref{1interpolation}--\ref{es-product}, we have from integration by parts and Young's inequality that
\begin{align}\label{3.6}
\left|I_{1}\right| &\lesssim \left|\left\langle\nabla^{2} n^{+}, \nabla^{2}\left(n^{+} \operatorname{div} u^{+}\right)\right\rangle\right|
+\left|\left\langle\nabla^{2} n^{+}, \nabla^{2}\left(\nabla n^{+} \cdot u^{+}\right)\right\rangle\right| \notag\\
& \lesssim\left\|\nabla^{2} n^{+}\right\|_{L^2}\left(\left\|\nabla^{2} n^{+}\right\|_{L^2}\left\|\nabla u^{+}\right\|_{L^{\infty}}
+\left\|n^{+}\right\|_{L^{\infty}}\left\|\nabla^{3} u^{+}\right\|_{L^2}\right) \notag\\
&\quad+\left\|\nabla^{2} n^{+}\right\|^{2}_{L^2}\left\|\nabla u^{+}\right\|_{L^{\infty}}+
\left\|\nabla^{2} n^{+}\right\|_{L^2}\left\|\nabla^{2}\left(\nabla n^{+} \cdot u^{+}\right)-\nabla^{3}  n^{+} \cdot u^{+}\right\|_{L^2} \notag\\
& \lesssim\left\|\nabla^{2} n^{+}\right\|_{L^2}\left(\left\|\nabla^{2} n^{+}\right\|_{L^2}\left\|\nabla^2 u^{+}\right\|_{L^2}^{\frac{1}{2}}\left\|\nabla^3 u^{+}\right\|_{L^2}^{\frac{1}{2}}
+\left\|n^{+}\right\|_{H^2}\left\|\nabla^{3} u^{+}\right\|_{L^2}\right)\\
&\quad+\left\|\nabla^{2} n^{+}\right\|_{L^2}\left(\left\|\nabla^{2} n^{+}\right\|_{L^2}\left\|\nabla^2 u^{+}\right\|_{L^2}^{\frac{1}{2}}\left\|\nabla^3 u^{+}\right\|_{L^2}^{\frac{1}{2}}
+\left\|\nabla n^{+}\right\|_{L^3}\left\|\nabla^{2} u^{+}\right\|_{L^6}\right) \notag\\
& \lesssim \delta_0\left(\left\|\nabla^{2} n^{+}\right\|_{L^2}^{2}+\left\|\nabla^3 u^{+}\right\|_{L^2}^{2}\right).\notag
\end{align}
Similarly, for the term $I_4$, we have
\begin{equation}\label{3.7}\left|I_{4}\right| \lesssim \delta_0\left(\left\|\nabla^{2} n^{-}\right\|_{L^2}^{2}+\left\|\nabla^3 u^{-}\right\|_{L^2}^{2}\right).\end{equation}
Employing the similar arguments used in \eqref{3.6}, we also have
\begin{align}\label{3.8}
\left|I_{2}\right| &\lesssim \left|\left\langle\nabla\left[g_{+}\left(n^{+}, n^{-}\right) \nabla n^{+}\right], \nabla^{3}
u^{+}\right\rangle\right|\notag \\
&\quad+\left|\left\langle\nabla\left[\bar{g}_{+}\left(n^{+}, n^{-}\right) \nabla n^{-}\right], \nabla^{3} u^{+}\right\rangle\right|\notag \\
&\quad+\left|\left\langle\nabla^{2}\left[\left(u^{+} \cdot \nabla\right) u^{+}\right], \nabla^{2} u^{+}\right\rangle\right| \notag\\
&\quad+\left|\left\langle\nabla\left[h_{+}\left(n^{+}, n^{-}\right)\left(\nabla n^{+} \cdot \nabla\right) u^{+}\right], \nabla^{3} u^{+}\right\rangle\right| \notag\\
&\quad+\left|\left\langle\nabla\left[k_{+}\left(n^{+}, n^{-}\right)\left(\nabla n^{-} \cdot \nabla\right) u^{+}\right], \nabla^{3} u^{+}\right\rangle\right|\notag \\
&\quad+\left|\left\langle\nabla\left[h_{+}\left(n^{+}, n^{-}\right) \nabla n^{+} \cdot \nabla^{t} u^{+}\right], \nabla^{3} u^{+}\right\rangle\right| \notag\\
&\quad+\left|\left\langle\nabla\left[k_{+}\left(n^{+}, n^{-}\right) \nabla n^{-} \cdot \nabla^{t} u^{+}\right], \nabla^{3} u^{+}\right\rangle\right| \notag\\
&\quad+\left|\left\langle\nabla\left[h_{+}\left(n^{+}, n^{-}\right) \nabla n^{+} \operatorname{div} u^{+}\right], \nabla^{3} u^{+}\right\rangle\right| \notag\\
&\quad+\left|\left\langle\nabla\left[k_{+}\left(n^{+}, n^{-}\right) \nabla n^{-} \operatorname{div} u^{+}\right], \nabla^{3} u^{+}\right\rangle\right| \notag\\
&\quad+\left|\left\langle\nabla\left[l_{+}\left(n^{+}, n^{-}\right) \Delta u^{+}\right], \nabla^{3} u^{+}\right\rangle\right| \notag\\
&\quad+\left|\left\langle\nabla\left[l_{+}\left(n^{+}, n^{-}\right) \nabla \operatorname{div} u^{+}\right], \nabla^{3} u^{+}\right\rangle\right|\notag\\
&\lesssim\left\|\nabla^{3} u^{+}\right\|_{L^2}\left(\left\|g_{+}\left(n^{+}, n^{-}\right)\right\|_{L^\infty}\left\|\nabla^{2} n^{+}\right\|_{L^2}+\left\|\nabla g_{+}\left(n^{+}, n^{-}\right)\right\|_{L^{6}}\left\|\nabla n^{+}\right\|_{L^{3}}\right)\\
&\quad+\left\|\nabla^{3} u^{+}\right\|_{L^2}\left(\left\|\bar{g}_{+}\left(n^{+}, n^{-}\right)\right\|_{L^{\infty}}
\left\|\nabla^{2} n^{-}\right\|_{L^2}+\left\|\nabla\bar{g}_{+}\left(n^{+}, n^{-}\right)\right\|_{L^{6}}\left\|\nabla n^{-}\right\|_{L^{3}}\right)\notag\\
&\quad+\left\|\nabla^{2} u^{+}\right\|_{L^{6}}\left(\left\|u^{+}\right\|_{L^{3}}\left\|\nabla^{3} u^{+}\right\|_{L^2}
+\left\|\nabla^{2} u^{+}\right\|_{L^2}\left\|\nabla u^{+}\right\|_{L^{3}}\right)\notag\\
&\quad+\left\|\nabla^{3} u^{+}\right\|_{L^2}\left(\left\|\nabla\left[h_{+}\left(n^{+}, n^{-}\right) \nabla u^{+}\right]\right\|_{L^{6}}\left\|\nabla n^{+}\right\|_{L^{3}}\right.\notag\\
&\quad\left.+\left\|h_{+}\left(n^{+}, n^{-}\right) \nabla u^{+}\right\|_{L^{\infty}}\left\|\nabla^{2} n^{+}\right\|_{L^2}\right)\notag\\
&\quad+\left\|\nabla^{3} u^{+}\right\|_{L^2}\left(\left\|\nabla\left[k_{+}\left(n^{+}, n^{-}\right) \nabla u^{+}\right]\right\|_{L^{6}}\left\|\nabla n^{-}\right\|_{L^{3}}\right.\notag\\
&\quad\left.+\left\|k_{+}\left(n^{+}, n^{-}\right) \nabla u^{+}\right\|_{L^{\infty}}\left\|\nabla^{2} n^{-}\right\|\right)\notag\\
&\quad+\left\|\nabla^{3} u^{+}\right\|_{L^2}\left(\left\|l_{+}\left(n^{+}, n^{-}\right)\right\|_{L^{\infty}}\left\|\nabla^{3} u^{+}\right\|_{L^2}\right.\notag\\
&\quad\left.+\left\|\nabla l_{+}\left(n^{+}, n^{-}\right)\right\|_{L^{6}}\left\|\nabla^{2} u^{+}\right\|_{L^{3}}\right)\notag\\
&\lesssim \delta_0\left(\left\|\nabla^{2}\left(n^{+}, n^{-}\right)\right\|^{2}_{L^2}+\left\|\nabla^2 u^{+}\right\|_{H^1}^{2}\right).\notag
\end{align}
Similarly, for the term $I_5$, we have
\begin{equation}\label{3.9}\left|I_{5}\right| \lesssim \delta_0\left(\left\|\nabla^{2}\left(n^{+}, n^{-}\right)\right\|^{2}+\left\|\nabla^2 u^{-}\right\|_{H^1}^{2}\right).\end{equation}
For the terms $I_3$ and $I_6$, we have from Young's inequality that
\begin{equation}\label{3.10}\left|I_{3}\right| \lesssim \frac{\beta_1\beta_2}{\nu_1^+}\left\|\nabla^{2} n^{-}\right\|_{L^2}^{2}+\frac{\beta_1}{4\beta_2}\nu_1^+\left\|\nabla^3 u^{+}\right\|_{L^2}^{2},\end{equation}
and
\begin{equation}\label{3.11}\left|I_{6}\right| \lesssim \frac{\beta_3\beta_4}{\nu_1^-}\left\|\nabla^{2} n^{+}\right\|_{L^2}^{2}+\frac{\beta_4}{4\beta_3}\nu_1^-\left\|\nabla^3 u^{-}\right\|_{L^2}^{2}.\end{equation}
Combining the estimates \eqref{3.4}--\eqref{3.11} and noting the smallness of $\delta_0$, we conclude that
\begin{align}\begin{split}\label{3.12}
&\displaystyle\frac{1}{2} \frac{\mathrm{d}}{\mathrm{d} t}\left\{\frac{\beta_{1}}{\beta_{2}}\left\|\nabla^{2} n^{+}\right\|^{2}_{L^2}
+\displaystyle\frac{\beta_{1}}{\beta_2}\left\|\nabla^{2} u^{+}\right\|^{2}_{L^2}\right\}+
\frac{\beta_{1}}{2{\beta_2}}\left(\nu_{1}^{+}\left\|\nabla^{3} u^{+}\right\|^{2}_{L^2}
+\nu_{2}^{+}\left\|\nabla^{2} \operatorname{div} u^{+}\right\|^{2}_{L^2}\right) \\
&\quad\quad \leq C \delta_0\left(\left\|\nabla^{2}\left(n^{+}, n^{-}\right)\right\|^{2}+\left\|\nabla^2 u^{+}\right\|_{L^2}^{2}\right)+\frac{\beta_{1} \beta_{2}}{\nu_{1}^{+}}\left\|\nabla^{2} n^{-}\right\|^{2}_{L^2},
\end{split}\end{align}

and
\begin{align}\begin{split}\label{3.13}
&\displaystyle\frac{1}{2} \frac{\mathrm{d}}{\mathrm{d} t}\left\{\frac{\beta_{4}}{\beta_{3}}\left\|\nabla^{2} n^{-}\right\|^{2}_{L^2}
+\displaystyle\frac{\beta_{4}}{\beta_3}\left\|\nabla^{2} u^{-}\right\|^{2}_{L^2}\right\}+
\frac{\beta_{4}}{2{\beta_3}}\left(\nu_{1}^{-}\left\|\nabla^{3} u^{-}\right\|^{2}_{L^2}
+\nu_{2}^{-}\left\|\nabla^{2} \operatorname{div} u^{-}\right\|^{2}_{L^2}\right) \\
&\quad\quad \leq C \delta_0\left(\left\|\nabla^{2}\left(n^{+}, n^{-}\right)\right\|^{2}+\left\|\nabla^2 u^{-}\right\|_{L^2}^{2}\right)+\frac{\beta_{3} \beta_{4}}{\nu_{1}^{-}}\left\|\nabla^{2} n^{+}\right\|^{2}_{L^2},
\end{split}\end{align}
for some positive constant $C$ independent of $\delta_0$.\par
\textbf{{Step 2. Dissipation of $\nabla^2(n^{+,h},n^{-,h})$.}}
Applying the operator $\nabla \mathcal{{F}}^{-1}(1-\phi(\xi))$ to $\eqref{2.11}_2$ and $\eqref{2.11}_4$ and then multiplying the resulting equations by
$\frac{1}{\beta_{2}} \nabla^{2}  n^{+, h}$ and
$\frac{1}{\beta_{3}} \nabla^{2}  n^{-, h}$ respectively, integrating over $\mathbb{R}^{3},$  we get
\begin{align}\begin{split}\label{3.14}\displaystyle &\frac{\mathrm{d}}{\mathrm{d} t}
\left\langle\nabla u^{+, h}, \frac{1}{\beta_{2}} \nabla^{2} n^{+, h}\right\rangle
+\frac{\beta_{1}}{\beta_{2}}\left\|\nabla^{2} n^{+, h}\right\|^{2}_{L^2}+
\left\langle\nabla^{2}  n^{-, h}, \nabla^{2} n^{+, h}\right\rangle \\
&= \frac{1}{\beta_{2}}\left\langle\nabla u^{+, h}, \partial_{t} \nabla^{2} n^{+, h}\right\rangle+\frac{\nu_{1}^{+}}
{\beta_{2}}\left\langle\nabla\Delta u^{+, h}, \nabla^{2} n^{+, h}\right\rangle+\frac{\nu_{2}^{+}}
{\beta_{2}}\left\langle\nabla^{2} \operatorname{div} u^{+, h}, \nabla^{2}  n^{+, h}\right\rangle \\
 &\quad+\frac{1}{\beta_{2}}\left\langle\nabla\mathcal{F}_{2}^h, \nabla^{2} n^{+, h}\right\rangle \\
&=: J_{1}+J_{2}+J_{3}+J_{4},\end{split}\end{align}
and
\begin{align}\begin{split}\label{3.15}\displaystyle &\frac{\mathrm{d}}{\mathrm{d} t}
\left\langle\nabla u^{-, h}, \frac{1}{\beta_{3}} \nabla^{2} n^{-, h}\right\rangle
+\frac{\beta_{4}}{\beta_{3}}\left\|\nabla^{2} n^{-, h}\right\|^{2}_{L^2}+
\left\langle\nabla^{2}  n^{+, h}, \nabla^{2} n^{-, h}\right\rangle \\
&= \frac{1}{\beta_{3}}\left\langle\nabla u^{-, h}, \partial_{t} \nabla^{2} n^{-, h}\right\rangle+\frac{\nu_{1}^{-}}
{\beta_{3}}\left\langle\nabla\Delta u^{-, h}, \nabla^{2} n^{-, h}\right\rangle+\frac{\nu_{2}^{-}}
{\beta_{3}}\left\langle\nabla^{2} \operatorname{div} u^{-, h}, \nabla^{2}  n^{-, h}\right\rangle \\
 &\quad+\frac{1}{\beta_{3}}\left\langle\nabla\mathcal{F}_{4}^h, \nabla^{2} n^{-, h}\right\rangle \\
&=: J_{5}+J_{6}+J_{7}+J_{8}.\end{split}\end{align} Due to
\eqref{3.2}, Lemma \ref{1interpolation}--\ref{lh2}, we can use
integration by parts, and Young's inequality to deduce that
\begin{align}\begin{split}\label{3.16}
\left|J_{1}\right|&=\left|-\frac{1}{\beta_{2}}\left\langle\nabla u^{+, h}, \beta_{1} \nabla^{2}\operatorname{div} u^{+, h}
\right\rangle+\frac{1}{\beta_{2}}\left\langle\nabla u^{+,h}, \nabla^{2}\mathcal{F}_{1}^h\right\rangle\right|\\
&=\left|\frac{\beta_{1}}{\beta_{2}}\left\|\nabla\operatorname{div} u^{+, h}\right\|^{2}_{L^2}-\frac{1}{\beta_{2}}\left\langle\nabla\operatorname{div} u^{+,h}, \nabla\mathcal{F}_{1}^h\right\rangle\right|\\
&\leq \frac{\beta_{1}}{\beta_{2}}\left\|\nabla\operatorname{div} u^{+, h}\right\|^{2}_{L^2}+C\left\|\nabla\operatorname{div}u^{+, h}\right\|_{L^2}
\left(\left\|n^{+}\right\|_{L^{3}}\left\|\nabla^{2} u^{+}\right\|_{L^6}+\left\|u^{+}\right\|_{L^{\infty}}\left\|\nabla^{2}n^{+}\right\|_{L^2}\right)\\
&\leq \frac{\beta_{1}}{\beta_{2}}\left\|\nabla\operatorname{div} u^{+, h}\right\|^{2}_{L^2}+C\left\|\nabla^{3}u^{+}\right\|_{L^2}
\left(\left\|n^{+}\right\|_{H^{1}}\left\|\nabla^{3} u^{+}\right\|_{L^2}+\left\|u^{+}\right\|_{H^{2}}\left\|\nabla^{2}n^{+}\right\|_{L^2}\right)\\
&\leq \frac{\beta_{1}}{\beta_{2}}\left\|\nabla\operatorname{div} u^{+, h}\right\|^{2}_{L^2}+C\delta_0
\left(\left\|\nabla^{2} n^{+}\right\|^{2}_{L^2}+\left\|\nabla^{3}  u^{+}\right\|^{2}_{L^2}\right).
\end{split}\end{align}
Similarly, for the term $J_5$, we have
\begin{equation}\label{3.17}\left|J_{5}\right|\leq  \frac{\beta_{4}}{\beta_{3}}\left\|\nabla\operatorname{div} u^{-, h}\right\|^{2}_{L^2}+C\delta_0
\left(\left\|\nabla^{2} n^{-}\right\|^{2}_{L^2}+\left\|\nabla^{3}  u^{-}\right\|^{2}_{L^2}\right).\end{equation}
By virtue of $\eqref{2.11}_1$, we can rewrite $J_2+J_3$ as
\begin{align}\begin{split}\label{3.18}
J_2+J_3&= \displaystyle\frac{\nu_1^++\nu_2^+}{\beta_2}\left\langle\nabla^2\operatorname{div} u^{+, h}, \nabla^2n^{+, h}\right\rangle\\
&=-\displaystyle\frac{\nu_1^++\nu_2^+}{2\beta_1\beta_2}\frac{\mathrm{d}}{\mathrm{d} t}\left\|\nabla^2 n^{+, h}\right\|^{2}_{L^2}-
\frac{\nu_1^++\nu_2^+}{2\beta_1\beta_2\sqrt{\alpha_1}}\left\langle\nabla^2\operatorname{div}(n^+u^+)^h, \nabla^2n^{+, h}\right\rangle.
\end{split}\end{align}
On the other hand, we have
\begin{align}\begin{split}\label{3.19}
\left\langle\nabla^2\operatorname{div}(n^+u^+)^h, \nabla^2n^{+, h}\right\rangle
&=\left\langle\nabla^2\operatorname{div}(n^+u^+)-\nabla^2\operatorname{div}(n^+u^+)^l, \nabla^2n^{+, h}\right\rangle\\
&=\left\langle\nabla^2\operatorname{div}(n^{+, h}u^+), \nabla^2n^{+, h}\right\rangle+
\left\langle\nabla^2\operatorname{div}(n^{+, l}u^+), \nabla^2n^{+, h}\right\rangle\\
&\quad-\left\langle\nabla^2\operatorname{div}(n^{+}u^+)^l, \nabla^2n^{+, h}\right\rangle\\
&=:J_{2,3}^1+J_{2,3}^2+J_{2,3}^3.
\end{split}\end{align}
For the term $J_{2, 3}^1$, by using integration by parts,
\eqref{3.2}, Lemmas \ref{1interpolation}--\ref{lh2} and Young's
inequality, we have
\begin{align}\label{3.20}
&\left|J_{2, 3}^1\right|\notag\\
&=\left|-\left\langle\operatorname{div}u^+, \left|\nabla^2n^{+, h}\right|^2\right\rangle+\left\langle\nabla^2\left(u^+\cdot \nabla n^{+, h}\right)-u^+\cdot \nabla^3n^{+, h},\nabla^2n^{+, h}\right\rangle+\left\langle\nabla^2\left(n^{+, h}\operatorname{div}u^+\right),\nabla^2n^{+, h}\right\rangle\right|\notag\\
&\lesssim \left\|\operatorname{div}u^+\right\|_{L^\infty}\left\|\nabla^{2}  n^{+, h}\right\|^{2}_{L^2}+
\left\|\nabla^{2}  n^{+, h}\right\|_{L^2}\left(\left\|\nabla u^+\right\|_{L^\infty}\left\|\nabla^{2}  n^{+, h}\right\|_{L^2}+
\left\|\nabla^{2}  u^{+}\right\|_{L^6}\left\|\nabla n^{+, h}\right\|_{L^3}\right)\notag\\
&\quad+\left\|\nabla^{2}  n^{+, h}\right\|_{L^2}\left(\left\|\nabla^3 u^+\right\|_{L^2}\left\|n^{+, h}\right\|_{L^\infty}+
\left\|\operatorname{div}u^+\right\|_{L^\infty}\left\|\nabla^2 n^{+, h}\right\|_{L^2}\right)\notag\\
&\lesssim \left\|\nabla^{2}  u^{+}\right\|_{L^2}^{\frac{1}{2}}\left\|\nabla^{3}  u^{+}\right\|_{L^2}^{\frac{1}{2}}\left\|\nabla^{2}  n^{+}\right\|_{L^2}^
{\frac{1}{2}}\left\|\nabla^{2}  n^{+, h}\right\|_{L^2}^{\frac{3}{2}}
+\left\|\nabla^{3}  u^{+}\right\|_{L^2}\left\|\nabla n^{+}\right\|_{H^1}\left\|\nabla^{2}  n^{+, h}\right\|_{L^2}\notag\\
&\lesssim\delta_0\left(\left\|\nabla^{2}  n^{+, h}\right\|_{L^2}^2+\left\|\nabla^{3}  u^{+}\right\|_{L^2}^2\right).
\end{align}
By employing similar arguments, for the terms $J_{2, 3}^2$ and  $J_{2, 3}^3$, we have
\begin{align}\begin{split}\label{3.21}
\left|J_{2, 3}^2\right|&\lesssim \left\|\nabla^{2}  n^{+, h}\right\|_{L^2}\left(
\left\|\nabla^{3}  n^{+, l}\right\|_{L^2}\left\|u^+\right\|_{L^\infty}+\left\|n^{+, l}\right\|_{L^\infty}
\left\|\nabla^{3}  u^{+}\right\|_{L^2}\right)\\
&\lesssim \left\|\nabla^{2}  n^{+}\right\|_{L^2}\left(
\left\|\nabla^{2}  n^{+}\right\|_{L^2}\left\|u^+\right\|_{H^2}+\left\|n^{+}\right\|_{H^2}
\left\|\nabla^{3}  u^{+}\right\|_{L^2}\right)\\
&\lesssim \delta_0\left( \left\|\nabla^{2}  n^{+}\right\|^{2}_{L^2}+\left\|\nabla^{3}  u^{+}\right\|^{2}_{L^2}\right),
\end{split}\end{align}
\begin{align}\begin{split}\label{3.22}
\left|J_{2, 3}^3\right|&\lesssim \left\|\nabla^{2}  n^{+, h}\right\|_{L^2}\left\|\nabla^{2}\operatorname{div}\left(n^{+}u^+\right)^l\right\|_{L^2}
\\&\lesssim \left\|\nabla^{2}  n^{+}\right\|_{L^2}\left\|\nabla^{2}\left(n^{+}u^+\right)\right\|_{L^2}\\
&\lesssim \left\|\nabla^{2}  n^{+}\right\|_{L^2}\left(\left\|\nabla^{2}  n^{+}\right\|_{L^2}\left\|u^{+}\right\|_{L^\infty}+
\left\|n^{+}\right\|_{L^3}\left\|\nabla^2u^{+}\right\|_{L^6}\right)\\
&\lesssim \left\|\nabla^{2}  n^{+}\right\|_{L^2}\left(\left\|\nabla^{2}  n^{+}\right\|_{L^2}\left\|u^{+}\right\|_{H^2}+
\left\|n^{+}\right\|_{H^1}\left\|\nabla^3u^{+}\right\|_{L^2}\right)\\
&\lesssim \delta_0\left(\left\|\nabla^{2}  n^{+}\right\|^{2}_{L^2}+\left\|\nabla^{3}  u^{+}\right\|^{2}_{L^2}\right).
\end{split}\end{align}
Combining the estimates \eqref{3.18}--\eqref{3.22}, we get
\begin{equation}\label{3.23}|J_2+J_3|\lesssim -\displaystyle\frac{\nu_1^++\nu_2^+}{2\beta_1\beta_2}\frac{\mathrm{d}}{\mathrm{d} t}\left\|\nabla^2 n^{+, h}\right\|^{2}_{L^2}
+\delta_0\left(\left\|\nabla^{2}  n^{+}\right\|^{2}_{L^2}+\left\|\nabla^{3}  u^{+}\right\|^{2}_{L^2}\right).\end{equation}
Similarly for the terms $J_6$ and $J_7$, we have
\begin{equation}\label{3.24}|J_6+J_7|\lesssim -\displaystyle\frac{\nu_1^-+\nu_2^-}{2\beta_3\beta_4}\frac{\mathrm{d}}{\mathrm{d} t}\left\|\nabla^2 n^{-, h}\right\|^{2}_{L^2}
+\delta_0\left(\left\|\nabla^{2}  n^{-}\right\|^{2}_{L^2}+\left\|\nabla^{3}  u^{-}\right\|^{2}_{L^2}\right).\end{equation}
Applying the similar arguments used in \eqref{3.8}, for the terms $J_4$ and $J_8$, we have
\begin{equation}\label{3.25}|J_4|\lesssim \delta_0\left(\left\|\nabla^{2}\left(n^+,n^{-}\right)\right\|^{2}_{L^2}+\left\|\nabla^{2}  u^{+}\right\|^{2}_{H^1}\right),\end{equation}
\begin{equation}\label{3.26}|J_8|\lesssim \delta_0\left(\left\|\nabla^{2}\left(n^+,n^{-}\right)\right\|^{2}_{L^2}+\left\|\nabla^{2}  u^{-}\right\|^{2}_{H^1}\right).\end{equation}
Combining the estimates \eqref{3.14}--\eqref{3.17} and \eqref{3.23}--\eqref{3.26}, we finally conclude that
\begin{align}\begin{split}\label{3.27}\displaystyle \frac{\mathrm{d}}{\mathrm{d} t} &
\displaystyle\left\{\frac{\nu_1^++\nu_2^+}{2\beta_1\beta_2}\left\|\nabla^2 n^{+, h}\right\|^{2}_{L^2}+\left\langle\nabla u^{+, h}, \frac{1}{\beta_{2}} \nabla^{2} n^{+, h}\right\rangle\right\}\\
&\quad+\frac{\beta_{1}}{\beta_{2}}\left\|\nabla^{2} n^{+, h}\right\|^{2}_{L^2}+\left\langle\nabla^2 n^{-, h}, \nabla^{2} n^{+, h}\right\rangle\\
&\leq\displaystyle C\left(\frac{\beta_{1}}{\beta_{2}}\left\|\nabla^2\operatorname{div}u^{+}\right\|^{2}_{L^2}+
\delta_0\left(\left\|\nabla^{2}  \left(n^+,n^{-}\right)\right\|^{2}_{L^2}+\left\|\nabla^{2}  u^{+}\right\|^{2}_{H^1}\right)\right),\end{split}\end{align}
and
\begin{align}\begin{split}\label{3.28}\displaystyle \frac{\mathrm{d}}{\mathrm{d} t} &
\displaystyle\left\{\frac{\nu_1^-+\nu_2^-}{2\beta_3\beta_4}\left\|\nabla^2 n^{-, h}\right\|^{2}_{L^2}
+\left\langle\nabla u^{-, h}, \frac{1}{\beta_{3}} \nabla^{2} n^{-, h}\right\rangle\right\}\\
&\quad+\frac{\beta_{4}}{\beta_{3}}\left\|\nabla^{2} n^{-, h}\right\|^{2}_{L^2}+\left\langle\nabla^2 n^{+, h}, \nabla^{2} n^{-, h}\right\rangle\\
&\leq\displaystyle C\left(\frac{\beta_{4}}{\beta_{3}}\left\|\nabla^2\operatorname{div}u^{-}\right\|^{2}_{L^2}+
\delta_0\left(\left\|\nabla^{2}  \left(n^+,n^{-}\right)\right\|^{2}_{L^2}+\left\|\nabla^{2}  u^{-}\right\|^{2}_{H^1}\right)\right).\end{split}\end{align}

\textbf{{Step 3. Proof of Theorem \ref{3.2mainth}.}} In this step, we are in a position to prove Theorem \ref{3.2mainth}. To begin with,
noticing that
$$
\quad-\frac{s_{-}^{2}(1,1)}{\alpha^{-}(1,1)}<f^{\prime}(1)<\frac{\eta-s_{-}^{2}(1,1)}{\alpha^{-}(1,1)}<0,
$$
where $\eta$ is a positive, small fixed constant, it is easy to see that $\alpha_{2}=\mathcal{C}^{2}(1,1)+\frac{\mathcal{C}^{2}(1,1) \alpha^{-}(1,1) f^{\prime}(1)}{s^{2}(1,1)}<$
$\frac{\mathcal{C}^{2}(1,1)}{s^{2}(1,1)} \eta$ and $\alpha_{4}$ is bounded. Therefore, $\beta_{2}=\frac{\alpha_{2} \sqrt{\alpha_{1}}}{\alpha_{4}}$ is a small
positive constant which will be determined later.\par
Computing $\eqref{3.12}+C_1\times\eqref{3.27}$, we have
\begin{align*}
\displaystyle &\frac{\mathrm{d}}{\mathrm{d} t}\left\{
 \frac{\beta_{1}}{2 \beta_{2}} \left\|\nabla^{2} n^{+}\right\|^{2}_{L^2}
+\frac{\left(\nu_{1}^{+}+\nu_{2}^{+}\right) C_{1}}{2 \beta_{1} \beta_{2}}\left\|\nabla^{2} n^{+, h}\right\|^{2}_{L^2} +\frac{\beta_{1}}{2 \beta_{2}}
\left\|\nabla^{2} u^{+}\right\|^{2}_{L^2}+\frac{C_{1}}{\beta_{2}}\left\langle\nabla u^{+, h}, \nabla^2n^{+, h}\right\rangle\right\}\\
&\quad+\frac{\beta_{1}}{2 \beta_{2}} \left(\nu_{1}^{+}\left\|\nabla^{3}u^{+}\right\|^{2}_{L^2}
+\nu_{2}^{+}\left\|\nabla^{2} \operatorname{div} u^{+}\right\|^{2}_{L^2}\right)+C_{1}\left(\frac{\beta_{1}}{\beta_{2}}
\left\|\nabla^{2} n^{+, h}\right\|^{2}_{L^2}+\left\langle\nabla^2n^{-, h}, \nabla^2n^{+, h}\right\rangle\right) \\
& \leq C \delta_0 \left(\left\|\nabla^{2}\left(n^{+}, n^{-}\right)\right\|^{2}_{L^2}+\left\|\nabla^2 u^{+}\right\|_{H^1}^{2}\right)
+\frac{\beta_{1} \beta_{2}}{\nu_{1}^{+}} \left\|\nabla^{2} n^{-}\right\|^{2}_{L^2}+CC_{1} \frac{\beta_{1}}{\beta_{2}}\left\|\nabla^2 \operatorname{div} u^{+}\right\|^{2}_{L^2}.
\end{align*}
Choosing $C_{1}$ as a fixed positive constant with $C_{1} \leqq \min \left\{\frac{v_{2}^{+}}{4C}, \frac{\nu_{1}^{+}+\nu_{2}^{+}}{2C_0^2}\right\}$ and
making use of the smallness of $\delta_0$, we get
\begin{align}\begin{split}\label{3.29}
\displaystyle&\frac{\mathrm{d}}{\mathrm{d} t} {\mathcal{E}}_{1}(t)+\frac{\beta_{1}}{6 \beta_{2}} \left(\nu_{1}^{+}\left\|\nabla^{3} u^{+}\right\|^{2}
_{L^2}+\nu_{2}^{+}\left\|\nabla^{2} \operatorname{div} u^{+}\right\|^{2}_{L^2}\right) \\
&\quad+\frac{C_{1} \beta_{1}}{\beta_{2}}\left\|\nabla^{2} n^{+, h}\right\|^{2}_{L^2}+C_{1}\left\langle\nabla^2 n^{-, h}, \nabla^2n^{+, h}\right\rangle \\
&\leq\frac{\beta_{1} \beta_{2}}{\nu_{1}^{+}}\left \|\nabla^{2} n^{-}\right\|^{2}_{L^2}+C \delta_0
\left\|\nabla^{2}\left(n^{+}, n^{-}\right)\right\|^{2}_{L^2}+C \delta_0\left\|\nabla^2 u^{+}\right\|^{2}_{L^2},
\end{split}\end{align}
where $\mathcal{E}_{1}(t)$ is given by
\begin{equation}
{\mathcal{E}}_{1}(t)= \frac{\beta_{1}}{2 \beta_{2}} \left\|\nabla^{2} n^{+}\right\|^{2}_{L^2}
+\frac{\left(\nu_{1}^{+}+\nu_{2}^{+}\right) C_{1}}{2 \beta_{1} \beta_{2}}\left\|\nabla^{2} n^{+, h}\right\|^{2}_{L^2} +\frac{\beta_{1}}{2 \beta_{2}}
\left\|\nabla^{2} u^{+}\right\|^{2}_{L^2}+\frac{C_{1}}{\beta_{2}}\left\langle\nabla u^{+, h}, \nabla^2n^{+, h}\right\rangle.
\nonumber\end{equation}
By virtue of $C_{1} \leq \frac{\nu_{1}^{+}+\nu_{2}^{+}}{2C_0^2}$,
 Lemma \ref{lh2} and Young's inequality, there exists a positive constant $C_{2}$ independent of $\delta_0$ and $\beta_{2}$, such that
\begin{equation}
\displaystyle\frac{1}{C_{2} \beta_{2}}\left\|\nabla^{2} (n^{+}, u^+)(t)\right\|^{2}_{L^2}
\leq\mathcal{E}_{1}(t)\leq\frac{C_{2}}{\beta_{2}}\left\|\nabla^{2} (n^{+}, u^+)(t)\right\|^{2}_{L^2}.\nonumber\end{equation}
Similarly, by calculating $\eqref{3.29}+C_{3} \times \eqref{3.13}$ , we have
\begin{align}\begin{split}\label{3.30}
\displaystyle&\frac{\mathrm{d}}{\mathrm{d} t} {\mathcal{E}}_{1}(t)+C_{3} \frac{\mathrm{d}}{\mathrm{d} t}
\left\{\frac{\beta_{4}}{2 \beta_{3}} \left\|\nabla^{2} n^{-}\right\|^{2}_{L^2}
+\frac{\beta_{4}}{2 \beta_{3}}\left\|\nabla^{2} u^{-}\right\|^{2}_{L^2}\right\} \\
&\quad+\frac{\beta_{1}}{6 \beta_{2}} \left(\nu_{1}^{+}\left\|\nabla^{3} u^{+}\right\|^{2}
_{L^2}+\nu_{2}^{+}\left\|\nabla^{2} \operatorname{div} u^{+}\right\|^{2}_{L^2}\right)
+C_{1} \beta_{1}\left(\frac{1}{\beta_{2}}-1\right)\left\|\nabla^{2} n^{+, h}\right\|^{2}_{L^2}\\
&\quad+C_{1}\left\langle\nabla^2 n^{-, h}, \nabla^2n^{+, h}\right\rangle
+C_{3} \frac{\beta_{4}}{4 \beta_{3}} \left(\nu_{1}^{-}\left\|\nabla^{3}  u^{-}\right\|^{2}_{L^2}+\nu_{2}^{-}
\left\|\nabla^{2} \operatorname{div} u^{-}\right\|^{2}_{L^2}\right) \\
&\leq \frac{\beta_{1} \beta_{2}}{\nu_{1}^{+}} \left\|\nabla^{2} n^{-}\right\|^{2}_{L^2}+C \delta_0 \left(
\left\|\nabla^{2}\left(n^{+}, n^{-}\right)\right\|^{2}_{L^2}
+\left\|\nabla^2\left(u^{+}, u^{-}\right)\right\|^{2}_{L^2}\right)+C_{3} \frac{\beta_{3} \beta_{4}}{\nu_{1}^{-}}
\left\|\nabla^{2} n^{+,l}\right\|^{2}_{L^2},
\end{split}\end{align} where $0<C_3\leq \min\left\{1, \frac{C_1\nu_1^-\beta_1}{\beta_3\beta_4}\right\}$.
\par \noindent Next, from $\eqref{3.30}+C_4\times\eqref{3.28}$, we have
\begin{align}\begin{split}\label{3.31}
\displaystyle&\frac{\mathrm{d}}{\mathrm{d} t}\left\{ {\mathcal{E}}_{1}(t)+\frac{C_3\beta_4}{2\beta_3}\left\|\nabla^{2} n^{-}\right\|^{2}_{L^2}
+\frac{(\nu_1^-+\nu_2^-)C_4}{2\beta_3\beta_4}\left\|\nabla^{2} n^{-, h}\right\|^{2}_{L^2}
+\frac{C_3\beta_{4}}{2 \beta_{3}} \left\|\nabla^{2} u^{-}\right\|^{2}_{L^2}
+\frac{C_4}{\beta_{3}}\left\langle\nabla u^{-, h}, \nabla^2n^{-, h}\right\rangle\right\} \\
&\quad+\frac{\beta_{1}}{6 \beta_{2}} \left(\nu_{1}^{+}\left\|\nabla^{3} u^{+}\right\|^{2}
_{L^2}+\nu_{2}^{+}\left\|\nabla^{2} \operatorname{div} u^{+}\right\|^{2}_{L^2}\right)
+C_{1} \beta_{1}\left(\frac{1}{\beta_{2}}-1\right)\left\|\nabla^{2} n^{+, h}\right\|^{2}_{L^2}\\
&\quad+C_{1}\left\langle\nabla^2 n^{-, h}, \nabla^2n^{+, h}\right\rangle
+C_{3} \frac{\beta_{4}}{6 \beta_{3}} \left(\nu_{1}^{-}\left\|\nabla^{3}  u^{-}\right\|^{2}_{L^2}+\nu_{2}^{-}
\left\|\nabla^{2} \operatorname{div} u^{-}\right\|^{2}_{L^2}\right) \\
&\quad+ C_{4} \left(\frac{\beta_{4}}{\beta_{3}}\left\|\nabla^{2}  n^{-, h}\right\|^{2}_{L^2}+\left\langle\nabla^2 n^{+, h}, \nabla^2n^{-, h}\right\rangle\right)\\
&\leq C \delta_0 \left(
\left\|\nabla^{2}\left(n^{+}, n^{-}\right)\right\|^{2}_{L^2}
+\left\|\nabla^2\left(u^{+}, u^{-}\right)\right\|^{2}_{L^2}\right)+\frac{\beta_{1} \beta_{2}}{\nu_{1}^{+}} \left\|\nabla^{2} n^{-}\right\|^{2}_{L^2}+C_{3} \frac{\beta_{3} \beta_{4}}{\nu_{1}^{-}}
\left\|\nabla^{2} n^{+,l}\right\|^{2}_{L^2},
\end{split}\end{align}
where $C_{4} \leq \min \left\{\frac{C_{2} v_{2}^{-}}{24}, \frac{C_3(\nu_1^-+\nu_2^-)}{2C_0^2}\right\}$ and we have used the fact that  $\delta_0$ is sufficiently small.
Taking $\beta_{2} \leq\min \left\{\frac{1}{2}, \frac{C_{4} \nu_{1}^{+} \beta_{4}}{4 \beta_{1} \beta_{3}}\right\}$, we get
\begin{align}\begin{split}\label{3.32}
\displaystyle&\frac{\mathrm{d}}{\mathrm{d} t}{\mathcal{E}}(t)+\frac{\beta_{1}}{6 \beta_{2}} \left(\nu_{1}^{+}\left\|\nabla^{3} u^{+}\right\|^{2}
_{L^2}+\nu_{2}^{+}\left\|\nabla^{2} \operatorname{div} u^{+}\right\|^{2}_{L^2}\right)\\
&\quad+C_{3} \frac{\beta_{4}}{6 \beta_{3}} \left(\nu_{1}^{-}\left\|\nabla^{3}  u^{-}\right\|^{2}_{L^2}+\nu_{2}^{-}
\left\|\nabla^{2} \operatorname{div} u^{-}\right\|^{2}_{L^2}\right)+C_{1} \beta_{1}\left(\frac{1}{\beta_{2}}-1\right)\left\|\nabla^{2} n^{+, h}\right\|^{2}_{L^2}\\
&\quad+C_4\frac{\beta_{4}}{2\beta_{3}}\left\|\nabla^{2}  n^{-, h}\right\|^{2}_{L^2}+(C_1+ C_{4}) \left\langle\nabla^2 n^{+, h}, \nabla^2n^{-, h}\right\rangle\\
&\leq C \delta_0 \left(
\left\|\nabla^{2}\left(n^{+}, n^{-}\right)\right\|^{2}_{L^2}
+\left\|\nabla^2\left(u^{+}, u^{-}\right)\right\|^{2}_{L^2}\right)+\frac{\beta_{1} \beta_{2}}{\nu_{1}^{+}} \left\|\nabla^{2} n^{-, l}\right\|^{2}_{L^2}+C_{3} \frac{\beta_{3} \beta_{4}}{\nu_{1}^{-}}
\left\|\nabla^{2} n^{+,l}\right\|^{2}_{L^2},
\end{split}\end{align}
where
\begin{equation}\mathcal{E}(t)=\mathcal{E}_1(t)+\mathcal{E}_2(t),\nonumber\end{equation}
and
\begin{equation}\mathcal{E}_2(t)=\frac{C_3\beta_4}{2\beta_3}\left\|\nabla^{2} n^{-}\right\|^{2}_{L^2}
+\frac{(\nu_1^-+\nu_2^-)C_4}{2\beta_3\beta_4}\left\|\nabla^{2} n^{-,
h}\right\|^{2}_{L^2} +\frac{C_3\beta_{4}}{2 \beta_{3}}
\left\|\nabla^{2} u^{-}\right\|^{2}_{L^2}
+\frac{C_4}{\beta_{3}}\left\langle\nabla u^{-, h}, \nabla^2n^{-,
h}\right\rangle.\nonumber\end{equation} Due to $C_{4}
\leq\frac{C_3(\nu_1^-+\nu_2^-)}{2C_0^2}$, we have from Lemma
\ref{lh2} and Young's inequality that
\begin{equation}
\displaystyle\frac{1}{C_{5}}\left\|\nabla^{2} (n^{-}, u^-)(t)\right\|^{2}_{L^2}
\leq\mathcal{E}_{2}(t)\leq C_5\left\|\nabla^{2} (n^{-}, u^-)(t)\right\|^{2}_{L^2},\nonumber\end{equation}
for some positive constant $C_5$ independent of $\delta_0$ and $\beta_2$.
Therefore, we have
\begin{align*}
\displaystyle&\frac{1}{C_{6}}\left(\frac{1}{\beta_2}\left\|\nabla^{2} (n^{+}, u^+)(t)\right\|^{2}_{L^2}+\left\|\nabla^{2} (n^{-}, u^-)(t)\right\|^{2}_{L^2}\right)\\
&\quad\quad\leq\mathcal{E}(t)\leq C_6\left(\frac{1}{\beta_2}\left\|\nabla^{2} (n^{+}, u^+)(t)\right\|^{2}_{L^2}+\left\|\nabla^{2} (n^{-}, u^-)(t)\right\|^{2}_{L^2}\right),\end{align*}
for some positive constant $C_6$ independent of $\delta_0$ and $\beta_2$.
Choosing $\beta_{2}$ sufficiently small, such that
$$
\frac{C_{1} \beta_{1}}{2}\left(\frac{1}{\beta_{2}}-1\right) \cdot \frac{C_{4} \beta_{4}}{2 \beta_{3}} \geq 4\left(C_{1}+C_{4}\right)^{2},
$$
that is,
\begin{equation}\label{3.33}
\beta_{2} \leq \min \left\{\frac{1}{2}, \frac{C_{4} v_{1}^{+} \beta_{4}}{4 \beta_{1} \beta_{3}},
 \frac{1}{1+\frac{16 \beta_{3}\left(C_{1}+C_{4}\right)^{2}}{C_{1} C_{4} \beta_{1} \beta_{4}}}\right\},
\end{equation}
we have
\begin{align}\begin{split}\label{3.34}
\displaystyle&\frac{C_{1} \beta_{1}\left(\frac{1}{\beta_{2}}-1\right)}{2}\left\|\nabla^{2} n^{+, h}\right\|^{2}_{L^2}
+C_{4} \frac{\beta_{4}}{2 \beta_{3}}\left\|\nabla^{2} n^{-, h}\right\|^{2}_{L^2}
+\left(C_{1}+C_{4}\right)\left\langle\nabla \nabla n^{-, h}, \nabla \nabla n^{+, h}\right\rangle \\
&\geq \frac{C_{4} \beta_{1}\left(\frac{1}{\beta_{2}}-1\right)}{4}\left\|\nabla^{2} n^{+, h}\right\|^{2}_{L^2}+C_{3} \frac{\beta_{4}}{4 \beta_{3}}
\left\|\nabla^{2} n^{-, h}\right\|^{2}_{L^2} \\
&\geq \frac{C_{1} \beta_{1}}{8} \frac{1}{\beta_{2}}\left\|\nabla^{2} n^{+, h}\right\|^{2}+C_{4} \frac{\beta_{4}}{4 \beta_{3}}\left\|\nabla^{2} n^{-, h}\right\|^{2}_{L^2}.
\end{split}\end{align}
As a result, combining \eqref{3.32} with \eqref{3.34} and using
Lemma \ref{lh2} and the smallness of $\delta_0$, we have
\begin{align}\begin{split}\label{3.35}
\displaystyle\frac{\mathrm{d}}{\mathrm{d} t}{\mathcal{E}}(t)+C_7&\left\{\frac{1}{\beta_{2}} \left(\left\|\nabla^{3} u^{+}\right\|^{2}
_{L^2}+\left\|\nabla^{2} n^{+, h}\right\|^{2}_{L^2}\right)\right.\\
&\quad\quad+\left.\left(\left\|\nabla^{3} u^{-}\right\|^{2}
_{L^2}+\left\|\nabla^{2} n^{-, h}\right\|^{2}_{L^2}\right)\right\}\lesssim \left\|\nabla^{2}\left( n^{+, l}, u^{+, l},  n^{-, l}, u^{-, l}\right)\right\|^{2}_{L^2},
\end{split}\end{align}
for some positive constant $C_7$ independent of $\delta_0$ and $\beta_2$.
Plugging $C_7\left(\frac{1}{\beta_{2}}\left\|\nabla^{2} n^{+, h}\right\|^{2}_{L^2}+\left\|\nabla^{2} n^{-, h}\right\|^{2}_{L^2}\right)$ into two sides of \eqref{3.35}, we have
\begin{align}\begin{split}\label{3.36}
\displaystyle\frac{\mathrm{d}}{\mathrm{d} t}{\mathcal{E}}(t)+C_7&\left\{\frac{1}{\beta_{2}} \left(\left\|\nabla^{3} u^{+}\right\|^{2}
_{L^2}+\left\|\nabla^{2} n^{+}\right\|^{2}_{L^2}\right)\right.\\
&\quad\quad+\left.\left(\left\|\nabla^{3} u^{-}\right\|^{2}
_{L^2}+\left\|\nabla^{2} n^{-}\right\|^{2}_{L^2}\right)\right\}\lesssim \left\|\nabla^{2}\left( n^{+, l}, u^{+, l},  n^{-, l}, u^{-, l}\right)\right\|^{2}_{L^2}.
\end{split}\end{align}
Using Lemma \ref{lh2} again, we have
\begin{equation}
\left\|\nabla^{2}\left( u^{+}, u^{-}\right)\right\|^{2}_{L^2}\leq C_0^2\left\|\nabla^{3}\left( u^{+}, u^{-}\right)\right\|^{2}_{L^2}+\left\|\nabla^{2}\left( u^{+,l}, u^{-, l}\right)\right\|^{2}_{L^2},
\nonumber\end{equation}
which together with \eqref{3.36} implies
\begin{equation}\label{3.37}
\displaystyle\frac{\mathrm{d}}{\mathrm{d} t}{\mathcal{E}}(t)+C_8{\mathcal{E}}(t)\lesssim \left\|\nabla^{2}\left( n^{+, l}, u^{+, l},  n^{-, l}, u^{-, l}\right)\right\|^{2}_{L^2},
\end{equation}
for some positive constant $C_8$ independent of $\delta_0$ and $\beta_2$.

Next, we employ Proposition \eqref{Prop2.6} to deduce the optimal decay rate of $\left\|\nabla^{2}\left( n^{+, l}, u^{+, l},  n^{-, l}, u^{-, l}\right)\right\|_{L^2}$.
To begin with, by defining $U=(n^+, u^+, n^-, u^-)^t$ and $\mathcal
F=(\mathcal{F}^1,\mathcal{F}^2,\mathcal{F}^3,\mathcal{F}^4)^t$, we have from Duhamel's principle that
\begin{equation}\label{3.38} U=\text{e}^{t\mathcal{A}}U(0)+\int_0^t\text{e}^{(t-\tau)\mathcal{A}}\mathcal F(\tau)\mathrm{d}\tau.
\end{equation}
By virtue of Proposition \ref{Prop2.6}, \eqref{3.38}, Plancherel theorem and H\"older's inequality, we have
\begin{equation}\label{3.39}\begin{split}&\left\|\nabla^2(n^{+, l}, u^{+, l},  n^{-, l}, u^{-, l})(t)\right\|_{L^{2}}\\
=&~\left\||\xi|^2(\widehat{n^{+, l}}, \widehat{u^{+, l}},  \widehat{n^{-, l}}, \widehat{u^{-, l}})(t)\right\|_{L^{2}}\\
 \lesssim &~(1+t)^{-\frac{7}{4}}\left\|
(n^+, u^+, n^-, u^-)(0)\right\|_{L^1}+\displaystyle\int_0^t(1+t-\tau)^{-\frac{7}{4}}\left\|
\mathcal{F}(\tau)\right\|_{L^1}\mathrm{d}\tau.
\end{split}\end{equation}
Next, we shall estimate the second term on the right--hand side of
\eqref{3.39}. To do this, by virtue of the definition of
$\mathcal{F}$ and \eqref{1.14}--\eqref{1.15}, we can bound the term
$\big{\|}{\mathcal F}(t)\big{\|}_{L^1}$ by
\begin{equation}\label{3.40}\begin{split}&\left\|\mathcal {F}(t)\right\|_{L^1}\\
 &\lesssim\Big{\|}\left(\text{div}(n^\pm u^\pm),n^\pm \nabla n^{\pm},n^\pm \nabla n^{\mp},
u^\pm \nabla u^{\pm}, \nabla n^\pm\cdot\nabla u^{\pm}, \nabla n^\pm\cdot\nabla u^{\mp}, n^+\nabla^{2}u^{\pm}, n^-\nabla^{2}u^{\pm}\right)(t)\Big{\|}_{L^1}\\
&\lesssim\left\|(n^+, u^+, n^-, u^-)(t)\right\|_{L^2}\left\|\nabla(n^+, u^+, n^-, u^-)(t)\right\|_{L^2}
+\left\|(n^+, u^+, n^-, u^-)(t)\right\|_{L^2}^2\\
&\quad+\left\|(n^+,n^-)(t)\right\|_{L^2}\left\|\nabla^2(u^+,u^-)(t)\right\|_{L^2}
\\ &\lesssim C(N_0)(1+t)^{-2}.
\end{split}\end{equation}
Plugging \eqref{3.40} into \eqref{3.39} gives that
\begin{equation}\label{3.41}\left\|\nabla^2(n^{+, l}, u^{+, l},  n^{-, l}, u^{-, l})(t)\right\|_{L^{2}}\lesssim
C(N_0)(1+t)^{-\frac{7}{4}}.\end{equation}
\par Finally, substituting \eqref{3.41} into \eqref{3.37} and using  Gronwall's inequality, we conclude that
\begin{equation}\mathcal{E}(t)\lesssim (1+t)^{-\frac{7}{2}},\nonumber\end{equation}
which implies \eqref{3.3}.
\par
 Therefore, the proof of Theorem \ref{3.2mainth} is completed.
\end{proof}
In what follows, we will devote ourselves to dealing with {\bf{Case
II:} $\bf{N>2}$}. To begin with, for any $0\le \ell\le N$, we define
the time--weighted energy functional as
\begin{equation}\label{3.42} \mathcal{E}_{\ell}^{N}(t)=\sup\limits_{0\leq\tau\leq t}\left
\{(1+\tau)^{\frac{3}{4}+\frac{\ell}{2}}\|\nabla^{\ell}(n^+, u^+, n^-, u^-)(\tau)\|_{H^{N-\ell}}\right
\}.
\end{equation}
Therefore, it suffices to prove that for any $0\le \ell\le N$,
$\mathcal{E}_{\ell}^{N}(t)$ has an uniform time--independent bound
for $t\geq t_0$ with a sufficiently large positive constant $t_0$.
In what follows, we always assume $t\geq t_0$. We will take
advantage of the low--frequency and high--frequency decomposition
and use the key linear convergence estimates obtained in Section 2
to achieve this goal by induction.

\bigskip

\begin{Theorem}\label{3.3mainth}({\bf{Case II:} $\bf{N>2}$}) Assume that the hypotheses of Theorem \ref{3.1mainth}
and \eqref{1.21} are in force. Then there exists a positive constant
$C$, which is independent of $t$, such that
\begin{equation}\nonumber \mathcal{E}_\ell^{N}(t)\le C(N_0),
\end{equation}
for $0\le \ell\le N$.
\end{Theorem}
\begin{proof} We will employ finite mathematical induction to prove
Theorem \ref{3.3mainth}. Therefore, it suffices to prove the
following Lemma \ref{Lemma3.1} and Lemma \ref{Lemma3.2}. Thus, the
proof Theorem \ref{3.3mainth} is completed.
\end{proof}

\smallskip

The first lemma is concerned with the estimate on $\mathcal{E}_0^{N}(t)$.

\begin{Lemma}\label{Lemma3.1} Assume that the  hypotheses of Theorem \ref{3.1mainth} and \eqref{1.21} are in force. Then there
 exists a positive constant $C$, which is independent of $t$, such that
\begin{equation}\label{3.43}\mathcal{E}_0^{N}(t)\le CN_0.
\end{equation}
\end{Lemma}
\begin{proof}  Using the similar argument of Evje--Wang--Wen \cite{Evje9} for the a priori estimates on
$(n^+, u^+, n^-, u^-)$, it is straightforward to deduce that there
exists a temporal energy functional $\mathcal E_0^{N}(t)$,
which is equivalent to $\|(n^+, u^+, n^-, u^-)(t)\|^2_{H^{N}}$
and satisfies
\begin{equation}\nonumber\frac{\mathrm{d}}{\mathrm{d}t}
\mathcal{E}_0^{N}(t)+\|\nabla(n^+, n^-)(t)\|^2_{H^{N-1}}+\|\nabla(u^+,u^-)(t)\|_{H^N}^2\le 0,
\end{equation}
which implies that there exists a positive constant $D_1$ such that
 \begin{equation}\label{3.44}\frac{\mathrm{d}}{\mathrm{d}t}\mathcal{E}_0^{N}(t)+\frac{1}{D_1}\mathcal{E}_0^{N}(t)\le
 C\|(n^{+, l}, u^{+, l}, n^{-, l}, u^{-, l})(t)\|^2_{L^2},
\end{equation}
where we have used Lemma \ref{lh2}.\par \noindent Similar to the
proof of \eqref{3.41}, we have
\begin{equation}\label{3.45}\left\|(n^{+, l}, u^{+, l},  n^{-, l}, u^{-, l})(t)\right\|_{L^{2}}\lesssim
C(1+t)^{-\frac{3}{4}}\left(N_0+\delta_0\mathcal{E}_0^{N}(t)\right).\end{equation}
Substituting \eqref{3.45} into \eqref{3.44} yields that
 \begin{equation}\nonumber\frac{\mathrm{d}}{\mathrm{d}t}\mathcal E_0^{N}(t)+\frac{1}{D_1}\mathcal E_0^{N}(t)
 \le C(1+t)^{-\frac{3}{2}}\left(N_0+\delta_0\mathcal{E}_0^{N}(t)\right)^2,
\end{equation}
which together with Gronwall's argument gives
\begin{equation}\begin{split}\nonumber\mathcal E_0^{N}(t)\leq &~\text{e}^{-\frac{1}{D_1}t}\mathcal E_0^{N}(0)
+C\int_0^t\text{e}^{-\frac{1}{D_1}(t-\tau)}(1+t-\tau)^{-\frac{3}{2}}\left(N_0+\delta_0\mathcal{E}_0^{N}(t)\right)^2\mathrm{d}\tau\\
\leq& ~C(1+t)^{-\frac{3}{2}}\left(N_0^2+\delta_0^2
{\mathcal{E}_0^{N}(t)}\right)^2,
 \end{split}\end{equation}
  which implies that
 \begin{equation}\label{3.46}(1+t)^{\frac{3}{2}}\|(n^+, u^+, n^-, u^-)(t)\|^2_{H^N}\leq C\left(N_0^2+\delta_0^2 (\mathcal{E}_0^{N}(t))^2\right).\end{equation}
Since $\mathcal{E}_0^{N}(t)$ is non--decreasing, it follows from
\eqref{3.46} that
$$(\mathcal{E}_0^{N}(t))^2\leq CN_0^2+\delta_0^2 (\mathcal{E}_0^{N}(t))^2,$$
which implies
\begin{equation}\mathcal{E}_0^N(t)\leq CN_0,\label{3.47}\end{equation}
if $\delta_0$ is sufficiently small. Therefore, the
proof of Lemma \ref{Lemma3.1} has been completed.
\end{proof}

 The next lemma is devoted to closing the estimates $\mathcal{E}_\ell^{N}(t)$, $1\le \ell\le N$.

\begin{Lemma}\label{Lemma3.2} Assume that the  hypotheses of Theorem \ref{3.1mainth} and \eqref{1.21}.
If additionally  \begin{equation}\label{3.48}\mathcal{E}_{\ell-1}^{N}(t)\le
C(N_0),\end{equation} then it holds that
\begin{equation}\label{3.49}\mathcal{E}_\ell^{N}(t)\le C(N_0),
\end{equation}
for $1\le \ell\le N$ and $t\geq t_0$.
\end{Lemma}
\begin{proof}
We will combine the key linear estimates with delicate nonlinear
energy analysis based on good properties of the low--frequency and
high--frequency decomposition to prove Lemma \ref{Lemma3.2}, and the
process involves the following four steps.

\bigskip

\textbf{{Step 1. $L^2$--estimate of $(\nabla^j n^{+, l},\nabla^j
u^{+, l},\nabla^jn^{-, l},\nabla^j u^{-, l}$)} with $\ell\leq j\leq
N$.} First, similar to the proof of \eqref{3.41}, we also have
\begin{equation}\label{3.50}\begin{split}&\left\|\nabla^j(n^{+, l}, u^{+, l},  n^{-, l}, u^{-, l})(t)\right\|_{L^{2}}\\
=&~\left\||\xi|^j(\widehat{n^{+, l}}, \widehat{u^{+, l}},  \widehat{n^{-, l}}, \widehat{u^{-, l}})(t)\right\|_{L^{2}}\\
 \lesssim &~(1+t)^{-\frac{3}{4}-\frac{j}{2}}\left\|
(n^+, u^+, n^-, u^-)(0)\right\|_{L^1}+\displaystyle\int_0^{\frac{t}{2}}(1+t-\tau)^{-\frac{3}{4}-\frac{j}{2}}\left\|
\mathcal{F}(\tau)\right\|_{L^1}\mathrm{d}\tau\\
&+\displaystyle\int_{\frac{t}{2}}^t(1+t-\tau)^{-\frac{5}{4}}\big{\|}|\xi|^{j-1}\widehat{\mathcal
F}^l(\tau)\big{\|}_{L^\infty}\mathrm{d}\tau\\
\lesssim&~ C(N_0)(1+t)^{-\frac{3}{4}-\frac{j}{2}}+\int_\frac{t}{2}^t(1+t-\tau)^{-\frac{5}{4}}\big{\|}|\xi|^{j-1}\widehat{\mathcal
F}^l(\tau)\big{\|}_{L^\infty}\mathrm{d}\tau.
\end{split}\end{equation}
Next, we shall estimate the second term on the right--hand of
\eqref{3.50}. The main idea of our approach is to make full use of
the benefit of the low--frequency and high--frequency decomposition.
To see this, by virtue of the assumption \eqref{3.48}, Lemma
\ref{lh2}, we can bound the term
$\big{\|}|\xi|^{j-1}\widehat{\mathcal F}^l(t)\big{\|}_{L^\infty}$ by
\begin{equation}\label{3.51}\begin{split}&\left\||\xi|^{j-1}\widehat{\mathcal F}^l(t)\right\|_{L^\infty}\\
 &\lesssim\Big{\|}\nabla^{j-1}\left(\text{div}(n^\pm u^\pm),n^\pm \nabla n^{\pm},n^\pm \nabla n^{\mp},
u^\pm \nabla u^{\pm}, \nabla n^\pm\cdot\nabla u^{\pm}, \nabla n^\pm\cdot\nabla u^{\mp}\right)(t)\Big{\|}_{L^1}
\\ &\quad+\Big{\|}\nabla^{\max\{0, j-2\}}\left(n^+\nabla^{2}u^{\pm}\right)(t)\Big{\|}_{L^1}+\Big{\|}\nabla^{\max\{0, j-2\}}\left(n^-\nabla^{2}u^{\pm}\right)(t)\Big{\|}_{L^1}\\
&\lesssim\left\|(n^+, u^+, n^-, u^-)(t)\right\|_{L^2}\left\|\nabla^{j}(n^+, u^+, n^-, u^-)(t)\right\|_{L^2}\\
&\quad+\left\|\nabla(n^+, u^+, n^-, u^-)(t)\right\|_{L^2}\left\|\nabla^{j-1}(n^+, u^+, n^-, u^-)(t)\right\|_{L^2}\\
&\quad+\left\|\nabla(n^+, u^+, n^-, u^-)(t)\right\|_{L^2}\left\|\nabla^{j}(n^+, u^+, n^-, u^-)(t)\right\|_{L^2}\\
&\quad+\left\|(n^+,n^-)(t)\right\|_{L^2}\left\|\nabla^j(u^+,u^-)(t)\right\|_{L^2}+\left\|\nabla^{\max\{0, j-2\}}(n^+,n^-)(t)\right\|_{L^2}\left\|\nabla^2(u^+,u^-)(t)\right\|_{L^2}
\\ &\lesssim
(1+t)^{-\frac{3}{4}}\mathcal{E}_0^{N}(t)(1+t)^{-\frac{3}{4}-\frac{j-1}{2}}\mathcal{E}_{j-1}^{N}(t)
\\ &\lesssim C(N_0)(1+t)^{-1-\frac{j}{2}}.
\end{split}\end{equation}

Substituting \eqref{3.51} into \eqref{3.50} yields that
\begin{equation}\|\nabla^j(n^{+, l}, u^{+, l},  n^{-, l}, u^{-, l})(t)\|_{L^{2}}\le
C(N_0)(1+t)^{-\frac{3}{4}-\frac{j}{2}}.\label{3.52}\end{equation}\par
It should be mentioned that the low--frequency convergence estimate
\eqref{3.52} plays a critical role in proving the optimal
convergence rates of the highest--order spatial derivatives of
solutions.

\bigskip

\textbf{{Step 2. $L^2$--estimate of $(\nabla^j n^{+},\nabla^j
u^{+},\nabla^jn^{-},\nabla^j u^{-}$)} with $\ell\leq j\leq N$.}
 Multiplying $\nabla^{j}\eqref{2.11}_{1}$, $\nabla^{j}\eqref{2.11}_{2}$, $\nabla^{j}\eqref{2.11}_{3}$ and $\nabla^{j}\eqref{2.11}_{4}$ by $\frac{\beta_1}{\beta_{2}} \nabla^{j} n^{+}$,
$\frac{\beta_1}{\beta_2} \nabla^{j} u^{+}, \frac{\beta_4}{\beta_3} \nabla^{j} n^{-}$ and $\frac{\beta_4}{\beta_3} \nabla^{j} u^{-}$ respectively, and then integrating over $\mathbb{R}^{3},$ we have
\begin{align}\begin{split}\label{3.53}
\displaystyle\frac{1}{2} \frac{\mathrm{d}}{\mathrm{d} t} &\left\{\frac{\beta_1}{\beta_2}\left\|\nabla^{j} n^{+}\right\|^{2}
+\frac{\beta_1}{\beta_2}\left\|\nabla^{j} u^{+}\right\|^{2}\right\}+
\frac{\beta_1}{\beta_2}\left(\nu_{1}^{+}\left\|\nabla^{j+1} u^{+}\right\|^{2}+\nu_{2}^{+}\left\|\nabla^{j} \operatorname{div} u^{+}\right\|^{2}\right) \\
&=\left\langle\nabla^{j} \mathcal{{F}}_{1}, \frac{\beta_1}{\beta_2} \nabla^{j} n^{+}\right\rangle+
\left\langle\nabla^{j} \mathcal{F}_{2}, \frac{\beta_1}{\beta_2} \nabla^{j} u^{+}\right\rangle-
\left\langle\nabla^{j} \nabla n^{-}, \beta_{1} \nabla^{j} u^{+}\right\rangle \\
&=:K_{1}^j+K_{2}^j+K_{3}^j,
\end{split}\end{align}
and
\begin{align}\label{3.54}
\frac{1}{2} \frac{\mathrm{d}}{\mathrm{d} t}&\left\{\frac{\beta_4}{\beta_3}\left\|\nabla^{j} n^{-}\right\|^{2}+\frac{\beta_4}{\beta_3}\left\|\nabla^{j} u^{-}\right\|^{2}\right\}+\frac{\beta_4}{\beta_3}\left(\nu_{1}^{-}\left\|\nabla^{j+1}u^{-}\right\|^{2}+\nu_{2}^{-}\left\|\nabla^{j} \operatorname{div} u^{-}\right\|^{2}\right)\notag\\
&=\left\langle\nabla^{j} \mathcal{F}_{3}, \frac{\beta_{4}}{\beta_{3}} \nabla^{j} n^{-}\right\rangle+
\left\langle\nabla^{j} \mathcal{F}_{4}, \frac{\beta_{4}}{\beta_{3}} \nabla^{j} u^{-}\right\rangle-\left\langle\nabla^{j} \nabla n^{+}, \beta_{4}
\nabla^{j} u^{-}\right\rangle\\
&=:K_{4}^j+K_{5}^j+K_{6}^j.\notag\end{align} \noindent From Lemmas
\ref{1interpolation}--\ref{es-product}, integration by parts and
Young's inequality, we have
\begin{align}\begin{split}\label{3.55}
\left|K_{1}^j\right| &\lesssim \left|\left\langle\nabla^{j} n^{+}, \nabla^{j}\left(n^{+} \operatorname{div} u^{+}\right)\right\rangle\right|
+\left|\left\langle\nabla^{j} n^{+}, \nabla^{j}\left(\nabla n^{+} \cdot u^{+}\right)\right\rangle\right| \\
& \lesssim\left\|\nabla^{j} n^{+}\right\|_{L^2}\left(\left\|\nabla^{j} n^{+}\right\|_{L^2}\left\|\nabla u^{+}\right\|_{L^{\infty}}
+\left\|n^{+}\right\|_{L^{\infty}}\left\|\nabla^{j+1} u^{+}\right\|_{L^2}\right) \\
&\quad+\left\|\nabla^{j} n^{+}\right\|^{2}_{L^2}\left\|\nabla u^{+}\right\|_{L^{\infty}}+
\left\|\nabla^{j} n^{+}\right\|_{L^2}\left\|\nabla^{j}\left(\nabla n^{+} \cdot u^{+}\right)-\nabla^{j+1}  n^{+} \cdot u^{+}\right\|_{L^2} \\
& \lesssim\left\|\nabla^{j} n^{+}\right\|_{L^2}\left(\left\|\nabla^{j} n^{+}\right\|_{L^2}\left\|\nabla^2 u^{+}\right\|_{L^2}^{\frac{1}{2}}\left\|\nabla^{3} u^{+}\right\|_{L^2}^{\frac{1}{2}}
+\left\|n^{+}\right\|_{H^2}\left\|\nabla^{j+1} u^{+}\right\|_{L^2}\right)\\
&\quad+\left\|\nabla^{j} n^{+}\right\|_{L^2}\left(\left\|\nabla^{2} n^{+}\right\|_{L^2}\left\|\nabla^2 u^{+}\right\|_{L^2}^{\frac{1}{2}}\left\|\nabla^3 u^{+}\right\|_{L^2}^{\frac{1}{2}}
+\left\|\nabla n^{+}\right\|_{L^3}\left\|\nabla^{j} u^{+}\right\|_{L^6}\right) \\
& \lesssim (1+t_0)^{-\frac{7}{8}}\left(\left\|\nabla^{j}
n^{+}\right\|_{L^2}^{2}+\left\|\nabla^{j+1}
u^{+}\right\|_{L^2}^{2}\right),
\end{split}\end{align}
where we have used \eqref{3.2} for $N=3$ and \eqref{3.3}. Similarly,
for the term $K_4^j$, we have
\begin{equation}\label{3.56}\left|K_{4}^j\right| \lesssim (1+t_0)^{-\frac{7}{8}}
\left(\left\|\nabla^{j} n^{-}\right\|_{L^2}^{2}+\left\|\nabla^{j+1}
u^{-}\right\|_{L^2}^{2}\right).\end{equation} Employing the similar
arguments used in \eqref{3.56}, we also have
\begin{align}\label{3.57}
\left|K_{2}^j\right| &\lesssim \left|\left\langle\nabla^{j-1}\left[g_{+}\left(n^{+}, n^{-}\right) \nabla n^{+}\right], \nabla^{j+1}
u^{+}\right\rangle\right| \notag\\
&\quad+\left|\left\langle\nabla^{j-1}\left[\bar{g}_{+}\left(n^{+}, n^{-}\right) \nabla n^{-}\right], \nabla^{j+1} u^{+}\right\rangle\right| \notag\\
&\quad+\left|\left\langle\nabla^{j}\left[\left(u^{+} \cdot \nabla\right) u^{+}\right], \nabla^{j} u^{+}\right\rangle\right| \notag\\
&\quad+\left|\left\langle\nabla^{j-1}\left[h_{+}\left(n^{+}, n^{-}\right)\left(\nabla n^{+} \cdot \nabla\right) u^{+}\right], \nabla^{j+1} u^{+}\right\rangle\right| \notag\\
&\quad+\left|\left\langle\nabla^{j-1}\left[k_{+}\left(n^{+}, n^{-}\right)\left(\nabla n^{-} \cdot \nabla\right) u^{+}\right], \nabla^{j+1} u^{+}\right\rangle\right| \notag\\
&\quad+\left|\left\langle\nabla^{j-1}\left[h_{+}\left(n^{+}, n^{-}\right) \nabla n^{+} \cdot \nabla^{t} u^{+}\right], \nabla^{j+1} u^{+}\right\rangle\right| \notag\\
&\quad+\left|\left\langle\nabla^{j-1}\left[k_{+}\left(n^{+}, n^{-}\right) \nabla n^{-} \cdot \nabla^{t} u^{+}\right], \nabla^{j+1} u^{+}\right\rangle\right| \notag\\
&\quad+\left|\left\langle\nabla^{j-1}\left[h_{+}\left(n^{+}, n^{-}\right) \nabla n^{+} \operatorname{div} u^{+}\right], \nabla^{j+1} u^{+}\right\rangle\right| \notag\notag\\
&\quad+\left|\left\langle\nabla^{j-1}\left[k_{+}\left(n^{+}, n^{-}\right) \nabla n^{-} \operatorname{div} u^{+}\right], \nabla^{j+1} u^{+}\right\rangle\right| \notag\\
&\quad+\left|\left\langle\nabla^{j-1}\left[l_{+}\left(n^{+}, n^{-}\right) \Delta u^{+}\right], \nabla^{3} u^{+}\right\rangle\right| \notag\\
&\quad+\left|\left\langle\nabla^{j-1}\left[l_{+}\left(n^{+}, n^{-}\right) \nabla \operatorname{div} u^{+}\right], \nabla^{j+1} u^{+}\right\rangle\right|\notag\\
&\lesssim\left\|\nabla^{j+1} u^{+}\right\|_{L^2}\left(\left\|g_{+}\left(n^{+}, n^{-}\right)\right\|_{L^\infty}\left\|\nabla^{j} n^{+}\right\|_{L^2}+\left\|\nabla^{j-1} g_{+}\left(n^{+}, n^{-}\right)\right\|_{L^{6}}\left\|\nabla n^{+}\right\|_{L^{3}}\right)\\
&\quad+\left\|\nabla^{j+1} u^{+}\right\|_{L^2}\left(\left\|\bar{g}_{+}\left(n^{+}, n^{-}\right)\right\|_{L^{\infty}}
\left\|\nabla^{j} n^{-}\right\|_{L^2}+\left\|\nabla^{j-1}\bar{g}_{+}\left(n^{+}, n^{-}\right)\right\|_{L^{6}}\left\|\nabla n^{-}\right\|_{L^{3}}\right)\notag\\
&\quad+\left\|\nabla^{j} u^{+}\right\|_{L^{6}}\left(\left\|u^{+}\right\|_{L^{3}}\left\|\nabla^{j+1} u^{+}\right\|_{L^2}
+\left\|\nabla^{j} u^{+}\right\|_{L^2}\left\|\nabla u^{+}\right\|_{L^{3}}\right)\notag\\
&\quad+\left\|\nabla^{j+1} u^{+}\right\|_{L^2}\left(\left\|\nabla^{j-1}\left[h_{+}\left(n^{+}, n^{-}\right) \nabla u^{+}\right]\right\|_{L^{6}}\left\|\nabla n^{+}\right\|_{L^{3}}\right.\notag\\
&\quad\left.+\left\|h_{+}\left(n^{+}, n^{-}\right) \nabla u^{+}\right\|_{L^{\infty}}\left\|\nabla^{j} n^{+}\right\|_{L^2}\right)\notag\\
&\quad+\left\|\nabla^{j+1} u^{+}\right\|_{L^2}\left(\left\|\nabla^{j-1}\left[k_{+}\left(n^{+}, n^{-}\right) \nabla u^{+}\right]\right\|_{L^{6}}\left\|\nabla n^{-}\right\|_{L^{3}}\right.\notag\\
&\quad\left.+\left\|k_{+}\left(n^{+}, n^{-}\right) \nabla u^{+}\right\|_{L^{\infty}}\left\|\nabla^{j} n^{-}\right\|\right)\notag\\
&\quad+\left\|\nabla^{j+1} u^{+}\right\|_{L^2}\left(\left\|l_{+}\left(n^{+}, n^{-}\right)\right\|_{L^{\infty}}\left\|\nabla^{j+1} u^{+}\right\|_{L^2}\right.\notag\\
&\quad\left.+\left\|\nabla^{j-1} l_{+}\left(n^{+}, n^{-}\right)\right\|_{L^{6}}\left\|\nabla^{2} u^{+}\right\|_{L^{3}}\right)\notag\\
&\lesssim (1+t_0)^{-\frac{7}{8}}\left(\left\|\nabla^{j}\left(n^{+},
n^{-}\right)\right\|^{2}_{L^2}+\left\|\nabla^j
u^{+}\right\|_{H^1}^{2}\right).\notag
\end{align}
Similarly, for the term $K_5^j$, we have
\begin{equation}\label{3.58}\left|K_{5}^j\right| \lesssim (1+t_0)^{-\frac{7}{8}}
\left(\left\|\nabla^{j}\left(n^{+},
n^{-}\right)\right\|^{2}+\left\|\nabla^j
u^{-}\right\|_{H^1}^{2}\right).\end{equation} For the terms $K_3^j$
and $K_6^j$, we have from Young's inequality that
\begin{equation}\label{3.59}\left|K_{3}^j\right| \lesssim
\frac{\beta_1\beta_2}{\nu_1^+}\left\|\nabla^{j} n^{-}\right\|_{L^2}^{2}+\frac{\beta_1}{4\beta_2}\nu_1^+\left\|\nabla^{j+1} u^{+}\right\|_{L^2}^{2},\end{equation}
and
\begin{equation}\label{3.60}\left|K_{6}^j\right| \lesssim \frac{\beta_3\beta_4}{\nu_1^-}\left\|\nabla^{j} n^{+}\right\|_{L^2}^{2}+\frac{\beta_4}{4\beta_3}\nu_1^-\left\|\nabla^{j+1} u^{-}\right\|_{L^2}^{2}.\end{equation}
Combining the relations \eqref{3.53}--\eqref{3.60} and using the
fact that $t_0$ is sufficiently large, we finally conclude that
\begin{align}\begin{split}\label{3.61}
&\displaystyle\frac{1}{2} \frac{\mathrm{d}}{\mathrm{d} t}\left\{\frac{\beta_{1}}{\beta_{2}}\left\|\nabla^{j} n^{+}\right\|^{2}_{L^2}
+\displaystyle\frac{\beta_{1}}{\beta_2}\left\|\nabla^{j} u^{+}\right\|^{2}_{L^2}\right\}+
\frac{\beta_{1}}{2{\beta_2}}\left(\nu_{1}^{+}\left\|\nabla^{j+1} u^{+}\right\|^{2}_{L^2}
+\nu_{2}^{+}\left\|\nabla^{j} \operatorname{div} u^{+}\right\|^{2}_{L^2}\right) \\
&\quad\quad \leq C
(1+t_0)^{-\frac{7}{8}}\left(\left\|\nabla^{j}\left(n^{+},
n^{-}\right)\right\|^{2}+\left\|\nabla^j
u^{+}\right\|_{L^2}^{2}\right)+\frac{\beta_{1}
\beta_{2}}{\nu_{1}^{+}}\left\|\nabla^{j} n^{-}\right\|^{2}_{L^2},
\end{split}\end{align}

and
\begin{align}\begin{split}\label{3.62}
&\displaystyle\frac{1}{2} \frac{\mathrm{d}}{\mathrm{d} t}\left\{\frac{\beta_{4}}{\beta_{3}}\left\|\nabla^{j} n^{-}\right\|^{2}_{L^2}
+\displaystyle\frac{\beta_{4}}{\beta_3}\left\|\nabla^{j} u^{-}\right\|^{2}_{L^2}\right\}+
\frac{\beta_{4}}{2{\beta_3}}\left(\nu_{1}^{-}\left\|\nabla^{j+1} u^{-}\right\|^{2}_{L^2}
+\nu_{2}^{-}\left\|\nabla^{j} \operatorname{div} u^{-}\right\|^{2}_{L^2}\right) \\
&\quad\quad \leq C
(1+t_0)^{-\frac{7}{8}}\left(\left\|\nabla^{j}\left(n^{+},
n^{-}\right)\right\|^{2}+\left\|\nabla^j
u^{-}\right\|_{L^2}^{2}\right)+\frac{\beta_{3}
\beta_{4}}{\nu_{1}^{-}}\left\|\nabla^{j} n^{+}\right\|^{2}_{L^2},
\end{split}\end{align}
for some positive constant $C$ independent of $t_0$.\par
 \textbf{{Step 3. Dissipation of $\nabla^j(n^{+, h}, n^{+, h})$ with $\ell\leq j\leq N$.}}
 Applying the operator $\nabla^{j-1} \mathcal{{F}}^{-1}(1-\phi(\xi))$ to $\eqref{2.11}_2$ and then multiplying the resultant equation by
by $\frac{1}{\beta_{2}} \nabla^{j}  n^{+, h}$, integrating over $\mathbb{R}^{3},$  we have
\begin{align}\begin{split}\label{3.63}\displaystyle &\frac{\mathrm{d}}{\mathrm{d} t}
\left\langle\nabla^{j-1} u^{+, h}, \frac{1}{\beta_{2}} \nabla^{j} n^{+, h}\right\rangle
+\frac{\beta_{1}}{\beta_{2}}\left\|\nabla^{j} n^{+, h}\right\|^{2}_{L^2}+
\left\langle\nabla^{j}  n^{-, h}, \nabla^{j} n^{+, h}\right\rangle \\
&= \frac{1}{\beta_{2}}\left\langle\nabla^{j-1} u^{+, h}, \partial_{t} \nabla^{j} n^{+, h}\right\rangle+\frac{\nu_{1}^{+}}
{\beta_{2}}\left\langle\nabla^{j-1}\Delta u^{+, h}, \nabla^{j} n^{+, h}\right\rangle+\frac{\nu_{2}^{+}}
{\beta_{2}}\left\langle\nabla^{j} \operatorname{div} u^{+, h}, \nabla^{j}  n^{+, h}\right\rangle \\
 &\quad+\frac{1}{\beta_{2}}\left\langle\nabla^{j-1}\mathcal{F}_{2}^h, \nabla^{j} n^{+, h}\right\rangle \\
&:= L_{1}^j+L_{2}^j+L_{3}^j+L_{4}^j.\end{split}\end{align} Due to
\eqref{3.2}, Lemmas \ref{1interpolation}--\ref{lh2}, we have from
integration by parts, and Young's inequality that
\begin{align}\begin{split}\label{3.64}
\left|L_{1}^j\right|&=\left|-\frac{1}{\beta_{2}}\left\langle\nabla^{j-1} u^{+, h}, \beta_{1} \nabla^{j}\operatorname{div} u^{+, h}
\right\rangle+\frac{1}{\beta_{2}}\left\langle\nabla^{j-1} u^{+,h}, \nabla^{j}\mathcal{F}_{1}^h\right\rangle\right|\\
&=\left|\frac{\beta_{1}}{\beta_{2}}\left\|\nabla^{j-1}\operatorname{div} u^{+, h}\right\|^{2}_{L^2}-\frac{1}{\beta_{2}}\left\langle\nabla^{j-1}\operatorname{div} u^{+,h}, \nabla^{j-1}\mathcal{F}_{1}^h\right\rangle\right|\\
&\leq \frac{\beta_{1}}{\beta_{2}}\left\|\nabla^{j-1}\operatorname{div} u^{+, h}\right\|^{2}_{L^2}
+C\left\|\nabla^{j-1}\operatorname{div}u^{+, h}\right\|_{L^2}
\left(\left\|n^{+}\right\|_{L^{3}}\left\|\nabla^{j} u^{+}\right\|_{L^6}+\left\|u^{+}\right\|_{L^{\infty}}\left\|\nabla^{j}n^{+}\right\|_{L^2}\right)\\
&\leq \frac{\beta_{1}}{\beta_{2}}\left\|\nabla^{j-1}\operatorname{div} u^{+, h}\right\|^{2}_{L^2}+C\left\|\nabla^{j+1}u^{+}\right\|_{L^2}
\left(\left\|n^{+}\right\|_{H^{1}}\left\|\nabla^{j+1} u^{+}\right\|_{L^2}+\left\|u^{+}\right\|_{H^{2}}\left\|\nabla^{j}n^{+}\right\|_{L^2}\right)\\
&\leq \frac{\beta_{1}}{\beta_{2}}\left\|\nabla^{j-1}\operatorname{div} u^{+, h}\right\|^{2}_{L^2}+C\delta_0
\left(\left\|\nabla^{j} n^{+}\right\|^{2}_{L^2}+\left\|\nabla^{j+1}  u^{+}\right\|^{2}_{L^2}\right).
\end{split}\end{align}
From $\eqref{2.11}_1$, $L_{2}^j+L_{3}^j$ can be rewritten as
\begin{align}\begin{split}\label{3.65}
L_{2}^j+L_{3}^j &= \displaystyle\frac{\nu_1^++\nu_2^+}{\beta_2}\left\langle\nabla^j\operatorname{div} u^{+, h}, \nabla^jn^{+, h}\right\rangle\\
&=-\displaystyle\frac{\nu_1^++\nu_2^+}{2\beta_1\beta_2}\frac{\mathrm{d}}{\mathrm{d} t}\left\|\nabla^j n^{+, h}\right\|^{2}_{L^2}-
\frac{\nu_1^++\nu_2^+}{2\beta_1\beta_2\sqrt{\alpha_1}}\left\langle\nabla^j\operatorname{div}(n^+u^+)^h, \nabla^jn^{+, h}\right\rangle.
\end{split}\end{align}
On the other hand, we have
\begin{align}\begin{split}\label{3.66}
\left\langle\nabla^j\operatorname{div}(n^+u^+)^h, \nabla^jn^{+, h}\right\rangle
&=\left\langle\nabla^j\operatorname{div}(n^+u^+)-\nabla^j\operatorname{div}(n^+u^+)^l, \nabla^jn^{+, h}\right\rangle\\
&=\left\langle\nabla^j\operatorname{div}(n^{+, h}u^+), \nabla^jn^{+, h}\right\rangle+
\left\langle\nabla^j\operatorname{div}(n^{+, l}u^+), \nabla^jn^{+, h}\right\rangle\\
&\quad-\left\langle\nabla^j\operatorname{div}(n^{+}u^+)^l, \nabla^jn^{+, h}\right\rangle\\
&=:L_{2,3}^{j, 1}+L_{2,3}^{j,2}+L_{2,3}^{j, 3}.
\end{split}\end{align}
For the term $L_{2,3}^{j, 1}$, by using integration by parts,
Lemmas \ref{1interpolation}--\ref{lh2} and Young's inequality, we
have
\begin{align}\begin{split}\label{3.67}
&\left|L_{2,3}^{j, 1}\right|\\
&=\left|-\left\langle\operatorname{div}u^+, \left|\nabla^jn^{+, h}\right|^2\right\rangle+\left\langle\nabla^j\left(u^+\cdot \nabla n^{+, h}\right)-u^+\cdot \nabla^{j+1}n^{+, h},
\nabla^jn^{+, h}\right\rangle+\left\langle\nabla^j\left(n^{+, h}\operatorname{div}u^+\right),\nabla^jn^{+, h}\right\rangle\right|\\
&\lesssim \left\|\operatorname{div}u^+\right\|_{L^\infty}\left\|\nabla^{j}  n^{+, h}\right\|^{2}_{L^2}+
\left\|\nabla^{j}  n^{+, h}\right\|_{L^2}\left(\left\|\nabla u^+\right\|_{L^\infty}\left\|\nabla^{j}  n^{+, h}\right\|_{L^2}+
\left\|\nabla^{j}  u^{+}\right\|_{L^6}\left\|\nabla n^{+, h}\right\|_{L^3}\right)\\
&\quad+\left\|\nabla^{j}  n^{+, h}\right\|_{L^2}\left(\left\|\nabla^{j+1} u^+\right\|_{L^2}\left\|n^{+, h}\right\|_{L^\infty}+
\left\|\operatorname{div}u^+\right\|_{L^\infty}\left\|\nabla^j n^{+, h}\right\|_{L^2}\right)\\
&\lesssim \left\|\nabla^{2}  u^{+}\right\|_{L^2}^{\frac{1}{2}}\left\|\nabla^{3}  u^{+}\right\|_{L^2}^{\frac{1}{2}}\left\|\nabla^{j}  n^{+, h}\right\|_{L^2}^
2
+\left\|\nabla n^{+}\right\|_{H^1}\left\|\nabla^{j}  n^{+, h}\right\|_{L^2}\left\|\nabla^{j+1}  u^{+}\right\|_{L^2}\\
&\lesssim\left(\delta_0+(1+t_0)^{-\frac{7}{8}}\right)\left(\left\|\nabla^{j}  n^{+,
h}\right\|_{L^2}^2+\left\|\nabla^{j+1}
u^{+}\right\|_{L^2}^2\right),
\end{split}\end{align}
where we have used \eqref{3.2} for $N=3$ and \eqref{3.3}. By using
similar arguments, for the terms $L_{2,3}^{j, 2}$ and  $L_{2,3}^{j,
3}$, we have
\begin{align}\begin{split}\label{3.68}
\left|L_{2,3}^{j, 2}\right|&\lesssim \left\|\nabla^{j}  n^{+, h}\right\|_{L^2}\left(
\left\|\nabla^{j+1}  n^{+, l}\right\|_{L^2}\left\|u^+\right\|_{L^\infty}+\left\|n^{+, l}\right\|_{L^\infty}
\left\|\nabla^{j+1}  u^{+}\right\|_{L^2}\right)\\
&\lesssim \left\|\nabla^{j}  n^{+}\right\|_{L^2}\left(
\left\|\nabla^{j}  n^{+}\right\|_{L^2}\left\|u^+\right\|_{H^2}+\left\|n^{+}\right\|_{H^2}
\left\|\nabla^{j+1}  u^{+}\right\|_{L^2}\right)\\
&\lesssim \delta_0\left( \left\|\nabla^{j}  n^{+}\right\|^{2}_{L^2}+\left\|\nabla^{j+1}  u^{+}\right\|^{2}_{L^2}\right),
\end{split}\end{align}
\begin{align}\begin{split}\label{3.69}
\left|L_{2,3}^{j, 3}\right|&\lesssim \left\|\nabla^{j}  n^{+, h}\right\|_{L^2}\left\|\nabla^{j}\operatorname{div}\left(n^{+}u^+\right)^l\right\|_{L^2}
\\&\lesssim \left\|\nabla^{j}  n^{+}\right\|_{L^2}\left\|\nabla^{j}\left(n^{+}u^+\right)\right\|_{L^2}\\
&\lesssim \left\|\nabla^{j}  n^{+}\right\|_{L^2}\left(\left\|\nabla^{j}  n^{+}\right\|_{L^2}\left\|u^{+}\right\|_{L^\infty}+
\left\|n^{+}\right\|_{L^3}\left\|\nabla^ju^{+}\right\|_{L^6}\right)\\
&\lesssim \left\|\nabla^{j}  n^{+}\right\|_{L^2}\left(\left\|\nabla^{j}  n^{+}\right\|_{L^2}\left\|u^{+}\right\|_{H^2}+
\left\|n^{+}\right\|_{H^1}\left\|\nabla^{j+1}u^{+}\right\|_{L^2}\right)\\
&\lesssim \delta_0\left(\left\|\nabla^{j}  n^{+}\right\|^{2}_{L^2}+\left\|\nabla^{j+1}  u^{+}\right\|^{2}_{L^2}\right).
\end{split}\end{align}
Combining the relations \eqref{3.65}--\eqref{3.69}, we have
\begin{equation}\label{3.70}|L_2^j+L_3^j|\lesssim -\displaystyle\frac{\nu_1^++\nu_2^+}{2\beta_1\beta_2}\frac{\mathrm{d}}{\mathrm{d} t}\left\|\nabla^j n^{+, h}\right\|^{2}_{L^2}
+\left(\delta_0+(1+t_0)^{-\frac{7}{8}}\right)\left(\left\|\nabla^{j}
n^{+}\right\|^{2}_{L^2}+\left\|\nabla^{j+1}
u^{+}\right\|^{2}_{L^2}\right).\end{equation} Applying the similar
arguments used in \eqref{3.8}, for the term $L_4^j$, we have
\begin{equation}\label{3.71}|L_4^j|\lesssim \left(\delta_0+(1+t_0)^{-\frac{7}{8}}\right)\left(\left\|\nabla^{j}\left(n^+,n^{-}\right)\right\|^{2}_{L^2}+\left\|\nabla^{j}  u^{+}\right\|^{2}_{H^1}\right).\end{equation}
Combining the relations \eqref{3.63}--\eqref{3.64} and \eqref{3.70}--\eqref{3.71}, we finally conclude that
\begin{align}\begin{split}\label{3.72}\displaystyle \frac{\mathrm{d}}{\mathrm{d} t} &
\displaystyle\left\{\frac{\nu_1^++\nu_2^+}{2\beta_1\beta_2}\left\|\nabla^j n^{+, h}\right\|^{2}_{L^2}
+\left\langle\nabla^{j-1} u^{+, h}, \frac{1}{\beta_{2}} \nabla^{j} n^{+, h}\right\rangle\right\}\\
&\quad+\frac{\beta_{1}}{\beta_{2}}\left\|\nabla^{j} n^{+, h}\right\|^{2}_{L^2}+\left\langle\nabla^j n^{-, h}, \nabla^{j} n^{+, h}\right\rangle\\
&\leq\displaystyle
C\left(\frac{\beta_{1}}{\beta_{2}}\left\|\nabla^j\operatorname{div}u^{+}\right\|^{2}_{L^2}+
\left(\delta_0+(1+t_0)^{-\frac{7}{8}}\right)\left(\left\|\nabla^{j}
\left(n^+,n^{-}\right)\right\|^{2}_{L^2}+\left\|\nabla^{j}
u^{+}\right\|^{2}_{H^1}\right)\right).\end{split}\end{align}
Similarly, for the dissipation estimate of $\nabla^jn^{-, h}$, we
also have
\begin{align}\begin{split}\label{3.73}\displaystyle \frac{\mathrm{d}}{\mathrm{d} t} &
\displaystyle\left\{\frac{\nu_1^-+\nu_2^-}{2\beta_3\beta_4}\left\|\nabla^j n^{-, h}\right\|^{2}_{L^2}
+\left\langle\nabla^{j-1} u^{-, h}, \frac{1}{\beta_{3}} \nabla^{j} n^{-, h}\right\rangle\right\}\\
&\quad+\frac{\beta_{4}}{\beta_{3}}\left\|\nabla^{j} n^{-, h}\right\|^{2}_{L^2}+\left\langle\nabla^j n^{+, h}, \nabla^{j} n^{-, h}\right\rangle\\
&\leq\displaystyle
C\left(\frac{\beta_{4}}{\beta_{3}}\left\|\nabla^j\operatorname{div}u^{-}\right\|^{2}_{L^2}+
\left(\delta_0+(1+t_0)^{-\frac{7}{8}}\right)\left(\left\|\nabla^{j}
\left(n^+,n^{-}\right)\right\|^{2}_{L^2}+\left\|\nabla^{j}
u^{-}\right\|^{2}_{H^1}\right)\right).\end{split}\end{align}

\textbf{Step 4: Closing the estimates.} Now, we are in a position to
close the estimates. Due to the key energy estimates \eqref{3.61}-\eqref{3.62} and \eqref{3.72}--\eqref{3.73},
for any $\ell\leq j\leq N$, we can follow the proof of \eqref{3.37} step by step to get
 \begin{equation}\label{3.74}
\displaystyle \frac{d}{dt}\mathcal{E}_j(t)+D_2\mathcal{E}_j(t)\lesssim \left\|\nabla^{j}  \left(n^{+, l},u^{+, l},n^{-, l},u^{-, l}\right)\right\|^{2}_{L^2},
 \end{equation}
 where $D_2>0$ is a given positive constant and $\mathcal{E}_j(t)$ is equivalent to $\left\|\nabla^{j}  \left(n^{+},u^{+},n^{-},u^{-}\right)\right\|^{2}_{L^2}$.
 Then, summing up the estimate \eqref{3.74} from $j=\ell$ to $N$, one has
 \begin{equation}\label{3.75}
\displaystyle \frac{d}{dt}\bar{\mathcal{E}}_\ell^N(t)+D_2\bar{\mathcal{E}}_\ell^N(t)\lesssim \left\|\nabla^{\ell}  \left(n^{+, l},u^{+, l},n^{-, l},u^{-, l}\right)\right\|^{2}_{H^{N-\ell}},
 \end{equation}
 where $\bar{\mathcal{E}}_\ell^N(t)$ is equivalent to $\left\|\nabla^{\ell}  \left(n^{+},u^{+},n^{-},u^{-}\right)\right\|^{2}_{H^{N-\ell}}$.
 Next, applying Gronwall's inequality to \eqref {3.75} and using \eqref{3.52} with $\ell\leq j\leq N$, we obtain
 \begin{equation}\label{3.76}
\bar{\mathcal{E}}_\ell^N(t)\leq C(N_0) \left(1+t\right)^{-\frac{3}{4}-\frac{\ell}{2}},
 \end{equation}
 which together with the definition of ${\mathcal{E}}_\ell^N(t)$ in \eqref{3.42} implies \eqref{3.49}.\par
 Therefore, the proof of  Lemma \ref{Lemma3.2} has been completed.
\end{proof}

In rest of this section, we devote ourselves to deducing the lower--bound on
the convergence rate of the global solution to complete the proof of
Theorem \ref{1mainth}.

\begin{Theorem}\label{5mainth} Assume that the  hypotheses of Theorem \ref{3.1mainth} and \eqref{1.24} are in force, then there is
a positive constant $c_2$ independent of time such that for any
large enough $t$,
\begin{equation}\begin{split}\label{3.77}\min&\{\|\nabla^\ell n^+(t)\|_{L^2},\|\nabla^\ell u^+(t)\|_{L^2},\|\nabla^\ell n^-(t)\|_{L^2},\|\nabla^\ell u^-(t)\|_{L^2}\}\\ \geq &~c_2(1+t)^{-\frac{3}{4}-\frac{\ell}{2}},
\end{split}\end{equation}
and
\begin{equation}\label{3.78}\min\left\{\|n^+(t)\|_{L^p},\|u^+(t)\|_{L^p},\|n^-(t)\|_{L^p},\|u^-(t)\|_{L^p}\right\}\ge c_2(1+t)^{-\frac{3}{2}\left(1-\frac{1}{p}\right)},
\end{equation}
for $0\leq \ell\leq N$ and $\ 2\leq p\leq\infty.$
\end{Theorem}
\begin{proof}If $t$ is large enough, it follows from \eqref{3.38}, Proposition \ref{Prop2.4}, Lemmas \ref{Lemma3.1} and \ref{lh1} that
\begin{equation}\begin{split}\nonumber\|&\wedge^{-1}(n^+, u^+, n^-, u^-)(t)\|_{L^2}
\\ \le ~&\|\wedge^{-1}(n^{+, l},u^{+, l}, n^{-, l}, u^{-, l})(t)\|_{L^2}+\|\wedge^{-1}(n^{+, h},u^{+, h}, n^{-, h}, u^{-, h})(t)\|_{L^2}\\ \le ~&C(N_0)(1+t)^{-\frac{1}{4}}+\int_0^t(1+t-\tau)^{-\frac{1}{4}}\|\mathcal F(\tau)\|_{L^1}\mathrm{d}\tau+\|(n^{+, h},u^{+, h}, n^{-, h}, u^{-, h})(t)\|_{L^2}
\\ \le ~&C(N_0)\left((1+t)^{-\frac{1}{4}}+\int_0^t(1+t-\tau)^{-\frac{1}{4}}(1+\tau)^{-\frac{3}{2}}\mathrm{d}\tau\right)\\
 \le ~&C(N_0)(1+t)^{-\frac{1}{4}},
\end{split}\end{equation}
and
\begin{equation}\begin{split}\nonumber&\min\left\{\|n^+(t)\|_{L^2},\|u^+(t)\|_{L^2},\|n^+(t)\|_{L^2},\|u^-(t)\|_{L^2}\right\}\\ \ge &~\min\left\{\|n^{+, l}(t)\|_{L^2},\|u^{+, l}(t)\|_{L^2},\|n^{-, l}(t)\|_{L^2},\|u^{-, l}(t)\|_{L^2}\right\}
\\ \gtrsim&~ \sqrt{\delta_0}N_0(1+t)^{-\frac{3}{4}}-\int_0^t(1+t-\tau)^{-\frac{3}{4}}\|\mathcal F(\tau)\|_{L^1}\mathrm{d}\tau\\
\geq&~C_1N_0\sqrt{\delta_0}(1+t)^{-\frac{3}{4}}-C\delta_0N_0(1+t)^{-\frac{3}{4}}\\
\geq&~c_3(1+t)^{-\frac{3}{4}},
\end{split}\end{equation}
since $\delta_0$ is sufficiently small. These together with the
interpolations
\begin{equation}\nonumber\|f\|_{L^2}\leq C\|\wedge^{-1}f\|_{L^2}^\frac{\ell}{\ell+1}\|\nabla^\ell f\|_{L^2}^\frac{1}{\ell\ell+1},\end{equation}
and
\begin{equation}\nonumber\|f\|_{L^2}\leq C\|\wedge^{-1}f\|_{L^2}^\frac{3p-6}{5p-6}\|f\|_{L^r}^\frac{2p}{5p-6}\end{equation}
imply \eqref{3.77} and \eqref{3.78} immediately, and thus the proof
of Theorem \ref{5mainth} is completed.\par Therefore, we have
completed the proof of Theorem \ref{1mainth}.
\end{proof}

\appendix

\section{Analytic tools}\label{1section_appendix}
We recall the Sobolev interpolation of the Gagliardo--Nirenberg inequality.
 \begin{Lemma}\label{1interpolation}
 Let $0\le i, j\le k$, then we have
\begin{equation}\nonumber
\norm{\nabla^i f}_{L^p}\lesssim \norm{  \nabla^jf}_{L^q}^{1-a}\norm{ \nabla^k f}_{L^r}^a
\end{equation}
where $a$ satisfies
\begin{equation}\nonumber
\frac{i}{3}-\frac{1}{p}=\left(\frac{j}{3}-\frac{1}{q}\right)(1-a)+\left(\frac{k}{3}-\frac{1}{r}\right)a.
\end{equation}

Especially, while $p=q=r=2$, we have
\begin{equation}\nonumber
 \norm{\nabla^if}_{L^2}\lesssim \norm{\nabla^jf}_{L^2}^\frac{k-i}{k-j}\norm{\nabla^kf}_{L^2}^\frac{i-j}{k-j}.
\end{equation}
\begin{proof}
This is a special case of  \cite[pp. 125, THEOREM]{Nirenberg}.
\end{proof}
\end{Lemma}

\begin{Lemma}\label{es-product}
For any integer $k\ge1$, we have
 \begin{equation}\nonumber
  \norm{\nabla ^k(fg)}_{L^p} \lesssim \norm{f}_{L^{p_1}}\norm{\nabla ^kg}_{L^{p_2}} +\norm{\nabla ^kf}_{L^{p_3}}\norm{g}_{L^{p_4}},
 \end{equation}
and
 \begin{equation}\nonumber
  \norm{\nabla ^k(fg)-f\nabla^kg}_{L^p} \lesssim \norm{\nabla f}_{L^{p_1}}\norm{\nabla ^{k-1}g}_{L^{p_2}} +\norm{\nabla ^kf}_{L^{p_3}}\norm{g}_{L^{p_4}},
 \end{equation}
where $p, p_1, p_{2}, p_{3}, p_{4} \in[1, \infty]$ and
$$
\frac{1}{p}=\frac{1}{p_{1}}+\frac{1}{p_{2}}=\frac{1}{p_{3}}+\frac{1}{p_{4}}.
$$
\end{Lemma}
\begin{proof}
 See \cite{Duan1}.
\end{proof}
Finally, the following two lemmas concern the estimate for the
low--frequency part and the high--frequency part of $f$.
\begin{Lemma}\label{lh1} If $f\in L^r(\mathbb R^3)$ for any $2\leq r\leq\infty$, then we have
$$\|f^l\|_{L^r}+\|f^h\|_{L^r}\lesssim \|f\|_{L^r}.$$
\end{Lemma}
\begin{proof}
For $2\leq r\leq\infty$, by Young's inequality for convolutions, for the low frequency, it holds
\begin{equation}\label{3.361}\|f^l\|_{L^r}\lesssim \|\mathfrak{F}^{-1}\phi\|_{L^1}\|f\|_{L^r}\lesssim \|f\|_{L^r},\nonumber\end{equation}
and hence
\begin{equation}\label{3.362}\|f^h\|_{L^r}\lesssim \|f\|_{L^r}+\|f^l\|_{L^r}\lesssim \|f\|_{L^r}.\nonumber\end{equation}\end{proof}

\begin{Lemma}\label{lh2} Let $f\in H^k(\mathbb R^3)$ for any integer $k\geq 2$. Then there exists a positive constant $C_0$ such that
\begin{equation}\label{3.363}\|\nabla^\ell f^h\|_{L^2}\leq C_0 \|\nabla^{\ell+1}f\|_{L^2},\nonumber\end{equation}
and
\begin{equation}\label{3.364}\|\nabla^{\ell+1} f^l\|_{L^2}\leq C_0\|\nabla^{\ell}f\|_{L^2},\nonumber\end{equation}
for any $0\leq \ell\leq k-1$.
\end{Lemma}
\begin{proof} This lemma can be shown directly by the definitions of the low--frequency and
high--frequency of $f$ and the Plancherel theorem,  and thus we omit the details.
\end{proof}

\bigskip

\section*{Acknowledgments}
Huaqiao Wang's research is
partially supported by the National Natural Science Foundation of China $\#$ 11901066, the
Natural Science Foundation of Chongqing $\#$ cstc2019jcyj-msxmX0167 and Project $\#$ 2019CDXYST0015
and $\#$ 2020 CDJQY-A040 supported by the Fundamental Research Funds for the Central Universities.
Guochun Wu's research was partially supported by National Natural
Science Foundation of China $\#$11701193, $\#$11671086, Natural
Science Foundation of Fujian Province $\#$ 2018J05005,
$\#$2017J01562 and Program for Innovative Research Team in Science
and Technology in Fujian Province University Quanzhou High-Level
Talents Support Plan $\#$2017ZT012. Yinghui Zhang's research is
partially supported by Guangxi Natural Science Foundation
$\#$2019JJG110003, $\#$2019AC20214, and National Natural Science
Foundation of China $\#$11771150, $\#$11571280, $\#$11301172 and
$\#$11226170.


\bigskip


\begin{thebibliography}{10}

\bibitem{Bear}J. Bear, Dynamics of fluids in porous media, environmental science
series. Elsevier, New York, 1972 (reprinted with corrections, New
York, Dover, 1988).

\bibitem{Brennen1}C.E. Brennen, Fundamentals of Multiphase Flow, Cambridge University Press, New York, 2005.

\bibitem{Bresch1} D. Bresch, B. Desjardins, J.-M. Ghidaglia, E. Grenier, Global
weak solutions to a generic two--fluid model, Arch. Rational Mech.
Anal. 196 (2010) 599-6293.

\bibitem{Bresch2} D. Bresch, X.D. Huang, J. Li, Global weak solutions to one-dimensional non--conservative viscous
compressible two--phase system, Commun. Math. Phys. 309 (2012)
737-755.

\bibitem{Dan1}
 R. Danchin, Global existence in critical spaces for compressible
Navier-Stokes equations, Invent. Math. 141 (2000) 579-614.

\bibitem{c1}
 H.B. Cui, W.J. Wang, L. Yao, C.J. Zhu, Decay rates of a nonconservative
 compressible generic two-fluid model, SIAM J. Math.
Anal. 48 (2016) 470-512.

\bibitem{Dan2}
 R. Danchin, Global existence in critical spaces
for flows of compressible viscous and heat-conductive gases, Arch.
Rational Mech. Anal. 16 (2001) 1-39.

\bibitem{Dan3}
R. Danchin, Fourier Analysis Methods for PDEs, in: Lecture Notes,
November 14 2005.

\bibitem{Duan1}R.J. Duan, L.Z. Ruan, C.J. Zhu, Optimal decay rates to conservation laws with
diffusion-type terms of regularity-gain and regularity-loss, Math.
Models Meth. Appl. Sci., 22(7) (2012) 1250012.

\bibitem{Evje1}S. Evje, Weak solutions for a gas-liquid model relevant for describing gas-kick in oil wells, SIAM J. Appl. Math.
43 (2011) 1887-1922.

\bibitem{Evje2}S. Evje, Global weak solutions for a compressible gas-liquid model with well-formation interaction, J. Differential
Equations 251 (2011) 2352-2386.

\bibitem{Evje3}S. Evje, T. Fl$\mathring{\text{a}}$tten, Hybrid flux-splitting schemes for a common two-fluid model, J. Comput. Phys. 192 (2003) 175-210.

\bibitem{Evje4}S. Evje, T. Fl$\mathring{\text{a}}$tten, Weakly implicit numerical schemes for a
two-fluid model, SIAM J. Sci. Comput. 26(5) (2005) 1449-1484.

\bibitem{Evje5} S. Evje, T. Fl$\mathring{\text{a}}$tten, H.A. Friis, Global weak solutions for a viscous liquid-gas model with transition to single-phase
gas flow and vacuum, Nonlinear Anal. 70 (2009) 3864-3886.

\bibitem{Evje6}S. Evje, K.H. Karlsen, Global existence of weak solutions for a viscous two-phase model, J. Differential Equations. 245 (2008) 2660-2703.

\bibitem{Evje7}S. Evje, K.H. Karlsen, Global weak solutions for a viscous liquid-gas model with singular pressure law, Commun.
Pure Appl. Anal. 8 (2009) 1867-1894.

\bibitem{Evje8}S. Evje, T. Fl$\mathring{\text{a}}$tten, On the wave structure of two-phase flow models, SIAM J. Appl. Math. 67 (2006) 487-511.

\bibitem{Evje9} S. Evje, W.J. Wang, H.Y. Wen, Global
Well-Posedness and Decay Rates of Strong Solutions to a
Non-Conservative Compressible Two-Fluid Model, Arch. Rational Mech.
Anal. {221} (2016) 2352-2386.

\bibitem{Evje10}
 S. Evje, H.Y. Wen, C.J. Zhu, On global
solutions to the viscous liquid-gas model with unconstrained
transition to single-phase flow, Math. Model. Meth. Appl. Sci. {27}
(2017) 323-346 .

\bibitem{Fan1}L. Fan, Q.Q. Liu, C.J. Zhu, Convergence rates to stationary solutions of a gas-liquid model with external forces,
Nonlinearity. 27 (2012) 2875-2901.

\bibitem{Friis1}H.A. Friis, S. Evje, T. Fl$\mathring{\text{a}}$tten, A numerical study of characteristic slow-transient behavior of a compressible 2D
gas-liquid two-fluid model, Adv. Appl. Math. Mech. 1 (2009) 166-200.

\bibitem{Guo1} Y. Guo, Y.J. Wang, Decay of dissipative equations and negative Sobolev spaces.
{Comm. Partial Differential Equations} {37(12)} (2012) 2165-2208.

\bibitem{Guo2}Z.H. Guo, J. Yang, L. Yao, Global strong solution for a three-dimensional viscous liquid-gas two-phase flow model
with vacuum, J. Math. Phys. 52 (2011) 093102.

\bibitem{Hao1}C.C. Hao, H.L. Li, Well-posedness for a multidimensional viscous liquid-gas two-phase flow model, SIAM J. Math.
Anal. 44 (3) (2012) 1304-1332.

\bibitem{Huang1}F.M. Huang, D.H. Wang, D.F. Yuan, Nonlinear stability and
exsitence of vortex sheets for inviscid liquid-gas two-phase flow,
Discrete Cont. Dyn. Sys., 39(6) (2019) 3535-3575.

\bibitem{Ishii1}M. Ishii, Thermo-Fluid Dynamic Theory of Two-Phase Flow, Eyrolles, Paris, 1975.

\bibitem{Li1} H.L. Li, A. Matsumura, G.J. Zhang, Optimal decay rate
of the compressible Navier-Stokes-Poisson system in $R^3$, Arch.
Ration. Mech. Anal. 196(2) (2010) 681-713.

\bibitem{Liu1}Q.Q. Liu, C.J. Zhu, Asymptotic behavior of a viscous liquid-gas model with mass-dependent viscosity and vacuum,
J. Differential Equations. 252 (2012) 2492-2519.

\bibitem{Mat1}A. Matsumura, T. Nishida, The initial value problem for the
equations of motion of compressible viscous and heat conductive
fluids, Proc. Japan Acad. Ser. A 55 (1979) 337-342.

\bibitem{Nirenberg} L. Nirenberg, On elliptic partial differential equations. {\it Ann. Scuola Norm. Sup. Pisa}, { 13} (1959) 115--162.

\bibitem{Prosperetti} A. Prosperetti, G. Tryggvason, Computational methods for
multiphase flow. Cambridge University Press, 2007.

\bibitem{Raja}K. R. Rajagopal, L. Tao, Mechanics of mixtures, Series on Advances
in Mathematics for Applied Sciences, Vol. 35, World Scientific,
1995.

\bibitem{ruan1} L.Z. Ruan,
Smoothing effect on the damping mechanism for an inviscid two-phase
gas-liquid model, P. Roy. Soc. Edinb. A, 144A (2014) 351-362.

\bibitem{ruan2} L.Z. Ruan, D.H Wang, S.K. Weng, C.J. Zhu,
Rectinear vortex sheets of inviscid liquid--gas two--phase flow:
linear stability, Commun.Math. Sci., 14(3) (2016) 735-776.

\bibitem{Sideris}  T.C. Sideris, B. Thomases, D.H. Wang, Long time behavior of solutions to the 3D compressible Euler
equations with damping, {Comm. Partial Differential Equations} { 28}
(2003) 795-816.

\bibitem{zw2}
Z. Tan, Y.J. Wang, On hyperbolic-dissipative systems of composite
type, J. Differential Equations 260 (2016) 1091-1125.

\bibitem{Vasseur}A. Vasseur, H.Y. Wen, C. Yu, Global weak solution to the viscous
two-fluid model with finite energy, J. Math. Pures Appl., 125 (2019)
247-282.

\bibitem{zw3}
W.J. Wang, W.K. Wang, Large time behavior for the system of a
viscous liquid-gas two-phase flow model in $\mathbb{R}^3$, J.
Differential Equations 261 (2016) 5561-5589.

\bibitem{Wen1}H.Y. Wen, L. Yao, C.J. Zhu, A blow-up criterion of strong solution to a 3D viscous liquid-gas two-phase flow model
with vacuum, J. Math. Pures Appl. 97 (2012) 204-229.

\bibitem{Wu9} G.C. Wu, Y.H. Zhang, Global well-posedness and large time behavior of the viscous liquid-gas two-phase flow model
in $\mathbb{R}^3$, P. Roy. Soc. Edinb. A 150 (2020) 1999-2024.

\bibitem{Yao1} L. Yao, T. Zhang, C.J. Zhu, Existence and asymptotic behavior of global weak solutions to a 2D viscous liquid-gas
two-phase flow model, SIAM J. Math. Anal. 42 (2010) 1874-1897.

\bibitem{Yao2} L. Yao, T. Zhang, C.J. Zhu, A blow-up criterion for a 2D viscous liquid-gas two-phase flow model, J. Differential
Equations. 250 (2011) 3362-3378.

\bibitem{Yao3}L. Yao, C.J. Zhu, Free boundary value problem for a viscous two-phase model with mass-dependent viscosity,
J. Differential Equations. 247 (2009) 2705-2739.

\bibitem{Yao4} L. Yao, C.J. Zhu, Existence and uniqueness of global weak solution to a two-phase flow model with vacuum, Math.
Ann. 349 (2011) 903-928.

\bibitem{Yao5} L. Yao, C.J. Zhu, R.Z. Zi, Incompressible limit of viscous liquid-gas two-phase flow model, SIAM J. Math. Anal.
44 (2012) 3324-3345.

\bibitem{ZYH3} Y.H. Zhang, Weak solutions for an inviscid two-phase flow model in physical vacuum, J. Differential Equations, 265 (2018) 6251-6294.

\bibitem{Zhang4} Y.H. Zhang, C.J. Zhu, Global existence and optimal convergence rates for the strong solutions in $H^2$
to the 3D viscous liquid-gas two-phase flow model,  J. Differential
Equations. 258 (2015) 2315-2338.



\end{thebibliography}
\end{document}